\documentclass[reqno, 11pt]{amsart}
\usepackage{color}
\usepackage{mathrsfs}
\usepackage{amscd}
\usepackage{amsmath}
\usepackage{latexsym}
\usepackage{amsfonts}
\usepackage{amssymb}
\usepackage{amsthm}
\usepackage{graphicx}
\usepackage{hyperref}
\usepackage{makecell}
\usepackage{array,color}
\usepackage{booktabs}
\usepackage{multirow}

\newtheorem{theorem}{Theorem}[section]
\newtheorem{proposition}[theorem]{Proposition}

\newtheorem{lemma}[theorem]{Lemma}
\newtheorem{definition}[theorem]{Definition}

\newtheorem{remark}[theorem]{Remark}

\newcommand{\cqd}{\hfill$\Box$}

\numberwithin{equation}{section}

\title[Decay estimates for fourth-order Schr\"odinger operators in dimension two  ]
{Decay estimates for fourth-order Schr\"odinger operators  in  dimension two }

\author{Ping Li, Avy Soffer and Xiaohua Yao\textsuperscript{\dag }  }
\address{Ping Li, School of Informations  and  Mathematics , Yangtze University, Jingzhou 434000, China}
\email{liping@whu.edu.cn}
\address{Avy Soffer, Mathematics Department, Rutgers University, New Brunswick, NJ, 08903, USA}
\email{soffer@math.rutgers.edu}

\address{Xiaohua Yao, Department of Mathematics and  Hubei Province Key Laboratory of Mathematical Physics, Central China Normal University, Wuhan, 430079, P.R. China}
\email{yaoxiaohua@ccnu.edu.cn}

\date{\today}

\begin{document}

\begin{abstract}\baselineskip=13pt
 In this paper we study the decay estimates of the fourth order Schr\"{o}dinger  operator $H=\Delta^{2}+V(x)$ on $\mathbb{R}^2$ with a bounded decaying potential $V(x)$. We first deduce the asymptotic expansions of resolvent of $H$ near  zero threshold in the presence of resonances or eigenvalue, and then use them to establish
the $L^1-L^\infty$  decay estimates of $e^{-itH}$
generated by the fourth order Schr\"{o}dinger operator $H$. Our methods used in the decay estimates depend on Littlewood-Paley decomposition and oscillatory integral theory. Moreover, we also classify these zero resonances as the distributional solutions of $H\phi=0$ in  suitable weighted spaces.  Due to the degeneracy of $\Delta^{2}$  at zero threshold,   we remark that the asymptotic expansions of resolvent $R_V(\lambda^4)$ and the classifications of resonances are much more involved than  Schr\"odinger operator $-\Delta+V$ in dimension two.
\end{abstract}
		\maketitle

	\baselineskip=15pt
\section{Introduction}

\subsection{Backgrounds}
In this paper we will consider the time decay estimates of the propagator $e^{-itH}$ generated by the following fourth order Schr\"{o}dinger operator on $L^2(\mathbb{R}^2)$:
\begin{equation*}
   \begin{split}
 H= H_0 +V(x),\ H_0=\Delta^2,
  \end{split}
\end{equation*}
where $V(x)$ is a real measurable function on $\mathbb{R}^2$ satisfying $|V(x)|\lesssim (1+|x|)^{-\beta}$ for
some $\beta >0$.
	

 As $V=0$, we recall that the free propagator $e^{-it\Delta^2}$ on $L^2(\mathbb{R}^n)$ can be expressed by
\begin{equation}\label{Free group}
 e^{-it\Delta^2}f=\mathfrak{F^{-1}}(e^{-it|\xi|^4}\widehat{f})=\int_{\mathbb{R}^n}I_0(t,x-y) f(y)dy,
\end{equation}
where $\hat{f}$ ( or $\mathfrak{F}(f)$) denotes Fourier transform of $f$, $\mathfrak{F^{-1}}(f)$ denotes its inverse Fourier transform, and $I_0(t,x):=\mathfrak{F^{-1}}(e^{-it|\xi|^4})(x)$ is the kernel of $e^{-it\Delta^2}$.  It was well-known that the following sharp pointwise estimates hold for any $\alpha\in \mathbb{N}^n$ (see e.g. \cite{BKS}),
 \begin{equation}\label{dive-esti1}
 |D^\alpha I_0(t,x)| \lesssim |t|^{-\frac{n+|\alpha|}{4}} \big(1+|t|^{-\frac{1}{4}}|x|\big)^{-\frac{n-|\alpha|}{3}},\,
 |t|\neq 0,\,x\in\mathbb{R}^n,
\end{equation}
where $D=(\partial_{x_1}, \cdots, \partial_{x_n})$. Then by the identity \eqref{Free group} and Young's inequality,  the estimate (\ref{dive-esti1}) immediately implies the following
$L^1-L^\infty$ decay estimates:
 \begin{equation}\label{dive-esti2}
 \|D^\alpha e^{-itH_0}\|_{L^1(\mathbb{R}^n)\rightarrow L^\infty(\mathbb{R}^n)} \lesssim |t|^{-\frac{n+|\alpha|}{4}},\, |\alpha|\leq n.
\end{equation}

Recently, several works were devoted to the study of the time decay estimates of $e^{-itH}$ generated by the fourth order Schr\"{o}dinger operator $H=\Delta^2 +V$ with a decaying  potential $V$.  Feng, Soffer and Yao \cite{FSY} first gave the asymptotic expansion of the
resolvent $R_V=(H-z)^{-1}$ around zero threshold assume that zero is a regular point of $H$ for $n\geq 5$
 and $n=3$, and proved that the Kato-Jensen decay estimate of $e^{-itH}$ is
 $(1+|t|)^{-n/4}$ for $n\geq 5$  and $L^1-L^\infty$ decay estimate is $O(|t|^{-1/2})$ for $n=3$ in the regular
 case.

  Later, based on the higher energy estimates of the resolvent $R_V(z)$ of $H$ in \cite{FSY}, Erdog\u{a}n, Green
and Toprak \cite{Erdogan-Green-Toprak} for $n=3$ and Green-Toprak \cite{Green-Toprak19} for
$n=4$, derived the asymptotic expansions of $R_V(z)$ near zero in the presence of zero resonance or
eigenvalue, and  proved that
the  $L^1-L^\infty$ estimates of $e^{-itH}$ is $O(|t|^{-n/4})$ for $n=3,4$ if zero is a regular point,  and the time decay rate
will be changed as zero energy resonance occurs. More recently, Soffer, Wu and Yao \cite{SWY21} proved
the $L^1-L^\infty$ estimates of $e^{-itH}$ is $O(|t|^{-1/4})$ for dimension $n=1$ whatever zero is a regular point
or resonance. It should be emphasized that the different types of zero resonance do not change
the optimal time decay rate of $e^{-itH}$ in dimension one, except at the cost of faster decay rate of
potential. Moreover, we also note that there exists some interesting works \cite{Goldberg-Green21} and \cite{Erdogan-Green21} on the $L^p$ bounds of wave operators for $n\ge 3$ in the regular case.

Besides  the bi-lapalce operator $\Delta^2$, we remark that Kato-Jensen estimates and  asymptotic expansion of resolvent have been first established by Feng, et al \cite{FSWY} for poly-Schr\"odinger operators $(-\Delta)^m+V$ with $n\ge 2m\ge 4$. In  the previous work, Murata \cite{MM} has generalized Kato and Jensen's works \cite{JK, J}  to a certain class of
 $P(D)+V$ assume that $P(D)$ satisfies
 \begin{equation}\label{elliptic-condition}
 (\nabla P)(\xi_0)=0,\  \hbox{det}\big[ \partial_i\partial_j P(\xi)\big]\Big|_{\xi_0} \neq 0.
\end{equation}
However, the poly-harmonic operators $H_0=(-\Delta)^m$ do not satisfy the nondegenerate condition
(\ref{elliptic-condition}) at $\xi=0$ but the case $m=1$. Hence the results in \cite{FWY} and \cite{FSWY} were not
covered by \cite{MM}. In particular, the degeneracy of $(-\Delta)^m$ at zero energy  leads to
more complicated classifications of zero resonances and asymptotic expansion  of
$(-\Delta)^m+V $ as $m\geq2.$ These are actually some of  the main difficult points for the higher
order operators.

In this paper, we establish the $L^1-L^\infty$ estimates of the fourth order
Schr\"{o}dinger operator $\Delta^2+V$ in dimension two. We first deduce the asymptotic expansion of $R_V(z)$ near zero in the presence of resonance or eigenvalue, and then study the $L^1-L^\infty$ estimates of $e^{-itH}$ for each kind of zero resonance. In particular, we show that the $L^1-L^\infty$ estimates of $e^{-itH}$ have the optimal decay rate $O(|t|^{-\frac{1}{2}})$ if zero is a regular point or the first type of resonance. Moreover, we classify these zero resonance types as the distributional solutions of $H\phi=0$ in suitable weighted spaces. Due to the degeneracy of $\Delta^{2}$  at zero threshold and the lower even dimension, we remark that the asymptotic expansions of resolvent $R_V(\lambda^4)$ at $\lambda=0$  are much complicated compared with Schr\"odinger operator $-\Delta+V$ in dimension two ( see e.g. \cite {JN,JN2} ). One can see  Section \ref{the proof of inverse Pro} for more details.

 The studies of higher order elliptic operators were motivated by Schr\"odinger operator $-\Delta+V$ ( e.g. see \cite{Agmon, BS, H2, Kur, Schechter, DaHi, DDY, SYY, Herbst-Skibsted-Adv-2015, Herbst-Skibsted-Adv-2017} and so on ). The decay estimates of Schr\"odinger operator  have been  active topics of research in the last thirty years, and applied broadly to  nonlinear Schr\"odinger equations. For instance, Journ\'{e}, Soffer and Sogge \cite{JSS} first established the $L^1-L^\infty$ estimates in regular case when $n\geq 3$.  Later, Weder \cite{Weder1}, Rodnianski and Schlag \cite{RodSchl},
Schlag and Goldberg \cite{Goldberg-Schlag}, Schlag \cite{Schlag-CMP} have made further contributions for $n\le 3$. In particular, Yajima in \cite{Yajima-JMSJ-95}  has established the $L^{1}- L^{\infty}$  estimates for $-\Delta+V$ by wave operator method.
For more further studies, we refer to \cite{Weder,DaFa, ES1, ES2,  Goldberg-Green-1, Goldberg-Green-2,Schlag, Schlag21} and therein references.


\subsection{Main results}\label{main result}
In this subsection, we first introduce some notations and the definition of zero resonance, and then state our main results.

For $a\in \mathbb{R}$, we use $a\pm$ to denote $a\pm \epsilon$ for any small $\epsilon >0$. $[a]$ denotes the largest integer less than or equal to $a$.
For $a,b\in \mathbb{R}^+$, $a \lesssim b$ (or $a \gtrsim b$) means that there exists some constant $c>0$ such that $a\leq c b$(or $a \geq c b$). Moreover, if $a \lesssim b$ and $b\lesssim a $, then we write $a\thicksim b$. Let $\langle x \rangle=  (1+|x|^2)^{1/2}$ and denote
the $L^p$-weighted spaces by $ L^p_s(\mathbb{R}^2)=\big\{ f:\, \langle x \rangle^s f\in L^p(\mathbb{R}^2) \big\}$ for $s\in\mathbb{R}$ and $1\le p\le\infty.$
We use the smooth, even low energy cut-off $\chi$ defined by $\chi(\lambda)=1$ if $|\lambda|<\lambda_0\ll1$ and $\chi(\lambda)=0$ when $|\lambda|>2\lambda_0$ for some sufficiently small constant $0<\lambda_0\ll 1.$ In analyzing the high energy we utilize the complementary cut-off $\widetilde{\chi}(\lambda)=1-\chi(\lambda).$

In this paper, we will give the precise definitions of different types of zero resonances in terms of the projection operators $S_j$ onto the kernel spaces of operators $T_j$ ( see Definition \ref{definition of resonance} of Section 2 below ). Equivalently, we also give the characterizations of the zero resonance in terms of the distributional solutions to $H \phi=0$ in Theorem \ref{resonance solutions} of Section \ref{classification} below.

More specifically, we say that zero is a {\it resonance of the first kind} if $\exists 0\neq \phi \in L^\infty_{-1}(\mathbb{R}^2)$ but have no nonzero $\phi\in  L^\infty(\mathbb{R}^2)$ such that $H \phi=0$,  {\it the second kind} if $\phi \in L^\infty(\mathbb{R}^2)$ but have no nonzero $\phi\in L^\infty_{1}(\mathbb{R}^2)$, {\it the third kind} if $\phi \in L^\infty_1(\mathbb{R}^2)$ but have no nonzero $\phi\in L^2(\mathbb{R}^2)$, and  {\it the fourth kind} if $\exists 0\neq \phi \in L^2(\mathbb{R}^2)$
 ( i.e. the fourth kind resonance is exactly the zero eigenvalue ).  If zero is neither a resonance nor an eigenvalue, then we say that zero is a {\it regular point} of $H$.

 Now we are ready to present the main results as follows.
\begin{theorem}\label{thm-main results}
Let $|V(x)|\lesssim (1+|x|)^{-\beta}\ ( x\in \mathbb{R}^2 )$ for
some $\beta >0$. Assume that $H=\Delta^2+V(x)$ has no positive eigenvalues. Let $P_{ac}(H)$
denote the projection onto the absolutely continuous spectrum space of $H$. Then

(i)~If zero is the regular point and $\beta>10$, then
\begin{equation}\label{eq-main results-1}
\|e^{-itH}P_{ac}(H)\|_{L^1\rightarrow L^\infty}\lesssim |t|^{-\frac{1}{2}}.
\end{equation}

(ii)~If zero is  the first kind resonance and $\beta>14$, then
\begin{equation}\label{eq-main results-2}
\| e^{-itH}P_{ac}(H)\|_{L^1\rightarrow L^\infty}\lesssim |t|^{-\frac{1}{2}}.
\end{equation}

(iii)~If zero is the second kind resonance and $\beta>18$, then
\begin{equation}\label{eq-main results-3}
\| H^{\frac{1}{2}}e^{-itH}P_{ac}(H)\chi(H)\|_{L^1\rightarrow L^\infty}
+\| e^{-itH}\widetilde{\chi}(H)\|_{L^1\rightarrow L^\infty}
\lesssim |t|^{-\frac{1}{2}}.
\end{equation}

(iv) If zero is the third or fourth kind resonance and $\beta>18$, then
\begin{equation}\label{eq-main results-4}
\| H^{\frac{3}{2}}e^{-itH}P_{ac}(H)\chi(H)\|_{L^1\rightarrow L^\infty}
+\| e^{-itH}\widetilde{\chi}(H)\|_{L^1\rightarrow L^\infty}
\lesssim |t|^{-\frac{1}{2}}.
\end{equation}
Where $\chi(\lambda)$ and $\widetilde{\chi}(\lambda)$ in the \eqref{eq-main results-3} and \eqref{eq-main results-4} denote the low and high energy smooth cut-off functions,  respectively.
\end{theorem}

\begin{remark} Some further comments on Theorem \ref{thm-main results} :

(i) If zero is the regular point or the first kind resonance, then the decay estimates \eqref{eq-main results-1} and \eqref{eq-main results-2} of
$e^{-itH}$  are optimal base on the fact $$\|e^{-it\Delta^2}\|_{L^1(\mathbb{R}^2)\rightarrow L^\infty(\mathbb{R}^2)}\lesssim |t|^{-1/2}.$$  Moreover, in view of the decay estimates \eqref{dive-esti2} with a derivatives term $D^\alpha (|\alpha|\le 2)$ as $n=2$, we remark that it is possible to use the present methods to improve the estimates \eqref{eq-main results-1} and \eqref{eq-main results-2} to $O(|t|^{-(2+s)/4})$ by adding a regular term $H^{s/4}$ for $0\le s\le 2$.   For some related results, we refer to \cite{SWY21} with  a regular term $H^s$ for dimension $n=1$ (also see \cite{Hill} for the regular case), \cite{LWWY21} with a regular term $H^s$ and \cite{Goldberg-Green-2} with  weights for dimension $n=3$. In addition, note that the authors in  \cite{Erdogan-Green-13a, Erdogan-Green} and \cite{Toprak17} improved the decay estimates in the weighted spaces for Schr\"{o}dinger operators of dimension $n=2$.

(ii) If  zero is the second, third or fourth kind resonance,  then the high energy part $e^{-itH}\widetilde{\chi}(H)$ has always the sharp time decay estimates $O(|t|^{-1/2})$.

 For low energy part, however, since the zero singularity of the asymptotic expansions of $\big(M^\pm(\lambda)\big)^{-1}$ are of too high order to estimate $e^{-itH}P_{ac}(H)\chi(H)$ (See Section 2 below),
we need to put suitable regularity terms $H^s$  into the low energy part $e^{-itH}P_{ac}(H)\chi(H)$ to temper zero singularities, as shown in the above decay estimates \eqref{eq-main results-3}  and \eqref{eq-main results-4}, also see Remark \ref{Remark of second kind} below for more explanations.
\end{remark}

Besides  the assumptions of zero resonance,  Theorem \ref{thm-main results} also requires that $H=\Delta^2 +V$ has no positive
eigenvalues embedding into the absolutely continuous spectrum, which has been the indispensable assumption
in dispersive estimates. For Schr\"{o}dinger operator $-\Delta +V$, Kato in \cite{K} showed
the absence of positive eigenvalues for $H=-\Delta +V$ with the bounded decay potentials $V= o(|x|^{-1})$ as
$|x|\rightarrow +\infty$. Moreover, one can see \cite{Simon3, FHHH, IJ, KoTa} for more further results on the absence of positive eigenvalues of Schr\"{o}dinger operator.

However, the situation for $H=\Delta^2+V$
seems to be more complicated than the second order cases.  There exists compactly supported
smooth potentials such that the $H=\Delta^2+V$ has some positive eigenvalues (see \cite{FSWY}).
On the other hand, more interestingly, \cite{FSWY} has showed that $H=\Delta^2+V$ has
no positive eigenvalues by assuming that the potential $V$ is bounded and satisfies the repulsive
condition (i.e. $(x\cdot \nabla)V \leq 0$). Moreover, we also notice that for a general self-adjoint operator $\mathcal{H}$ on $L^{2}(\mathbf{R}^n)$, even if $\mathcal{H}$ has a simple embedded eigenvalue, Costin and Soffer in \cite{CoSo} have proved that $\mathcal{H}+\epsilon W$ can kick off the eigenvalue located in a small interval under generic small perturbation of potential.

\subsection{The outline of proof}

In order to  show Theorem \ref{thm-main results},
we will use the following Stone's formula:
\begin{equation}\label{stoneformula}
   \begin{split}
 H^\alpha e^{-itH}P_{ac}(H)f(x)=\frac{2}{\pi i} \int_0^\infty e^{-it\lambda^4}
 \lambda^{3+4\alpha}[ R_V^+(\lambda^4)-R_V^-(\lambda^4)]f(x)d\lambda, \ \alpha\ge0.
  \end{split}
\end{equation}
Hence we needs to investigate  the boundary behaviour of  the  resolvent  $R_V(z)=(H-z)^{-1}$ by
perturbations of the free resolvent $R_0(z)=(\Delta^2 -z\big)^{-1}$,  which depends on the following decomposition identity:
\begin{equation}\label{R0zlaplace}
   \begin{split}
   \big(\Delta^2 -z\big)^{-1}=\frac{1}{2z^\frac{1}{2}} \big(R(-\Delta;z^\frac{1}{2})
    -R(-\Delta;-z^\frac{1}{2})\big),\ z\in \mathbb{C} \setminus [0,\infty),
  \end{split}
\end{equation}
where the resolvent $ R(-\Delta;z^\frac{1}{2}):=(-\Delta-z^\frac{1}{2})^{-1}$ with $\Im z^\frac{1}{2}>0$.

For $\lambda>0$,
the incoming and outgoing resolvent operators $R_0^\pm(\lambda^4)$ of $\Delta^2$ are defined by
\begin{equation}\label{R0lambda-pm}
   \begin{split}
 R_0^\pm(\lambda^4):=R_0^\pm(\lambda^4 \pm i0)= \lim_{\epsilon \rightarrow 0^+}
 \big( \Delta^2-(\lambda^4 \pm i\epsilon) \big)^{-1},
  \end{split}
\end{equation}
and then we have
\begin{equation}\label{R0lambda-4pm}
   \begin{split}
 R^\pm_0(\lambda^4)=\frac{1}{2\lambda^2}\big( R^\pm(-\Delta; \lambda^2)-R(-\Delta; -\lambda^2) \big),
  \ \lambda>0.
  \end{split}
\end{equation}
It is well-known that $R^\pm(-\Delta;\lambda^2)$ are well-defined  as the bounded operators of $B(L^2_s,L^2_{-s})$ for any $s>1/2$ by the limiting absorption principle (see e.g. \cite{Agmon}),
 therefore $R^\pm_0(\lambda^4) $ are also well-defined between these weighted $L^2$-space. These bounded properties have  been extended to $R^\pm_V(\lambda^4)$ of $H$ for $\lambda >0$ and certain decay bounded potentials $V$, see \cite{FSY}.

In order to obtain time decay of the integral (\ref{stoneformula}),
 we also need to study the asymptotical properties of $R^\pm_V(\lambda^4)$ near $\lambda=0 $ and $+\infty$, which correspond to the low energy estimate and high energy estimate respectively. For lower  energy parts,  these investigations are much more involved. Note that $R^\pm_V(\lambda^4)$ are related to $R^\pm_0(\lambda^4)$ by the following symmetric  resolvent formula:
  \begin{equation*}
R^\pm_V(\lambda^4) = R_0^\pm(\lambda^4) -R^\pm_0(\lambda^4)v(M^\pm(\lambda))^{-1}vR^\pm_0(\lambda^4),
\end{equation*}
where $M^{\pm}(\lambda)=U+ vR^\pm_0(\lambda^4)v$, $U(x)=\hbox{sign}\big(V(x)\big)$ and $v(x)=|V(x)|^{1/2}$. Hence we will analyze the expansion of $\big(M^\pm(\lambda)\big)^{-1}$ at zero ( see Theorem \ref{thm-main-inver-M} below ).

By using the expansions of the free resolvent of $R^\pm_0(\lambda^4)$ in \eqref{R0lambda-4pm},  we can obtain that
 \begin{equation*}
   M^\pm(\lambda)=\frac{a_\pm}{\lambda^2}P +g_0^\pm(\lambda)v G_{-1} v +T_0+ c_\pm\lambda^2 vG_1 v
                   + \tilde{g}_2^\pm(\lambda)\lambda^4vG_2v+  \lambda^4 vG_3v
     +E_2^\pm(\lambda),
 \end{equation*}
 where $P,\ vG_iv(i=-1,0,1,2,3)$ are the specific bounded operators on $L^2(\mathbb{R}^2)$ and $E_2^\pm(\lambda)$ are the error terms, see Section \ref{expansion} below for specific descriptions.
Note that the expansions of the $M^\pm(\lambda)$ include two terms with logarithm factors $g_0^\pm(\lambda)$ and $\tilde{g}_2^\pm(\lambda)$, which make the calculation of the inversion processes $\big(M^\pm(\lambda)\big)^{-1}$ to be quite complicated.  We need to use  Lemma \ref{lemma-JN-matrix} twice to complete the whole inversion processes in Section \ref{the proof of inverse Pro}.

 Firstly, in the regular case, we need to show that the inverse operator
\begin{equation*}
   \begin{split}
 d= \Big( g^\pm_0(\lambda)(Q-S_0)vG_{-1}v(Q-S_0) + (Q-S_0)\big(T_0- T_0D_1T_0\big)(Q-S_0) \Big)^{-1}
   \end{split}
\end{equation*}
exist on  $(Q-S_0)L^2 $ for enough small $\lambda$.
Since dim$(Q-S_0)L^2 =2$, so the operators
$$(Q-S_0)vG_{-1}v(Q-S_0)\
{\rm and}
\ (Q-S_0)(T_0- T_0D_1T_0)(Q-S_0)$$
are essentially  two $2\times 2$-matrix linear transforms  acting on $(Q-S_0)L^2 $. Hence in order to obtain the existence of the inverse operator $d$  on $(Q-S_0)L^2$ as $\lambda$ is small, it suffices to prove the matrix $ (Q-S_0) vG_{-1}v (Q-S_0)$ is invertible  on the subspace $(Q-S_0)L^2$. This will be proved on Lemma \ref{lem-d-inver}.

Secondly, in the second kind of resonance, we will need to prove  that another inverse operator
$$ d_1:= \Big(g^\pm_0(\lambda)^2 L_1 + g^\pm_0(\lambda)L_2+ L_3\Big)^{-1}$$
exists on $(S_2-S_3)L^2$ for enough small $\lambda$, where
\begin{equation*}
   \begin{split}
L_1=& (S_2-S_3)\Big(vG_{-1}vPvG_{-1}v- vG_{-1}vT_0D_4T_0vG_{-1}v \Big)(S_2-S_3),\\
L_2=& (S_2-S_3)\Big(vG_{-1}vT_0 + T_0vG_{-1}v- vG_{-1}vT_0D_4(T_0^2-c_+a_+vG_1v) \\
      &-(T_0^2-c_+a_+vG_1v)D_4T_0vG_{-1}v \Big)(S_2-S_3),\\
L_3=&(S_2-S_3)\Big(T_0^2-c_+a_+vG_1v  - (T_0^2-c_+a_+vG_1v)D_4(T_0^2-c_+a_+vG_1v) \Big)(S_2-S_3).
 \end{split}
\end{equation*}
Since $(S_2-S_3)L^2$ is actually one dimensional space,  hence it suffices to prove that  $tr(L_1)\neq 0$ in order that $L_1$ is invertible acting on $(S_2-S_3)L^2$.  One can see Lemma \ref{lem-d1-inver} for the more details.

Moreover,
we will make use of the following Littlewood-Paley decompositions:
\begin{equation}
   \begin{split}
e^{-itH}P_{ac}(H)(x,y)
&= \sum_{N=-\infty}^{\infty}\frac{2}{ \pi i }\int_0^\infty e^{-it\lambda^4}\lambda^{3}
\varphi_0(2^{-N}\lambda)[R_V^+(\lambda^4) -R_V^-(\lambda^4)](x,y) d\lambda\\
&=\sum_{N=-\infty}^{\infty} K_N(t,x,y),
\end{split}
\end{equation}
where $\sum_{N=-\infty}^{\infty}
\varphi_0(2^{-N}\lambda)=1$ for $\lambda >0,$
and then estimate the bound of each dyadic kernel $K_N(t,x,y)$ by oscillatory integrals
methods.

 For each dyadic kernel $K_N(t,x,y)$,  we will make use of the asymptotic expansion of $R_V(\lambda^4)$ near zero for $N\ll 0$ ( i.e.  Theorem \ref{thm-main-inver-M} ) and the higher energy estimates of resolvent $R_V(\lambda^4)$ for $N\gtrsim0$ ( see \cite{FSY} or Lemma \ref{lem-largeenergy} below ).  By the scaling reductions for each $N$, we will  reduce  $K_N(t,x,y)$ into the following kind of oscillatory integrals:
\begin{equation}\label{oscillatory inte}
\int_0^\infty e^{-it2^{4N}s^4}
e^{\pm i2^Ns\Phi(z)} \Psi(2^Ns,z)\varphi_0(s) ds,
\end{equation}
where $\Phi(z)$ is some phase function on $\mathbb{R}^m$,  and  $\Psi(s,z)$ is suitable amplitude functions $\mathbb{R}^+\times \mathbb{R}^m$.  Under the suitable assumptions on $\Phi(z)$ and  $\Psi(s,z)$,  the stationary phase method can be applied to estimate the oscillatory integral \eqref{oscillatory inte}, see  Lemma \ref{lem-LWP} below for the details.

Noting that the similar oscillatory integrals \eqref{oscillatory inte} have appeared in one dimensional case \cite{SWY21}, however the estimates established in \cite{SWY21} cannot work in the two dimensional case.  Hence we need to establish  refined estimates for the oscillatory integrals \eqref{oscillatory inte}, i.e. Lemma \ref{lem-LWP}, which allow us to work well on the two and other dimensional
cases (see e.g. \cite{LWWY21}).


The paper is organized as follows. In Section 2, we first establish the expansions for the free resolvent,
and then state  the asymptotic expansion of $(M^\pm(\lambda))^{-1}$ when $\lambda$ is near zero.
In Section 3, we are devoted to proving Theorem \ref{thm-main results}. In Section 4, we prove Theorem \ref{thm-main-inver-M}. In Section 5, we characterize the resonance spaces for each kind of zero resonance.

\bigskip

\section{Resolvent asymptotic expansions near zero }\label{expansion}
In this section, we first study the free resolvent asymptotic
expansions of $R_0^\pm(\lambda^4)$, then we give the asymptotic expansions of $M^\pm(\lambda)^{-1}$.

We first gathered some frequently used notations throughout the paper.
For an operator $\mathcal{E}(\lambda),$ we write $ \mathcal{E}(\lambda)=O_1(\lambda^{-s})$ if it's kernel $\mathcal{E}(\lambda)(x,y)$ has the property
\begin{equation*}
\sup\limits_{x,y\in\mathbb{R}^2}\Big(\lambda^s|\mathcal{E}(\lambda)(x,y)|
+\lambda^{s+1}|\partial_\lambda\mathcal{E}(\lambda)(x,y)|\Big)<\infty, \,\lambda>0.
\end{equation*}
Similarly, we use the notation $\mathcal{E}(\lambda)=O_1\big(\lambda^{-s}g(x,y)\big)$ if $\mathcal{E}(\lambda)(x,y)$ satisfies
\begin{equation*}
 |\mathcal{E}(\lambda)(x,y)|+\lambda|\partial_\lambda\mathcal{E}(\lambda)(x,y)|\le C\lambda^{-s}g(x,y),\ \lambda>0,
\end{equation*}
for some constant $C$.

\subsection{The free resolvent  }
In this subsection, by using the identity (\ref{R0lambda-4pm}) and the Bessel
function representation of the free resolvents $R^\pm(-\Delta;\lambda^2)$,
we now derive the asymptotic expansions of the free resolvent $R_0^\pm(\lambda^4)$ of the bi-Laplace
operator $\Delta^2$.

Recall first  the free resolvent of Laplacian (see e.g. \cite{Erdogan-Green}):
\begin{equation}\label{reslolent-R}
 R^\pm (-\Delta;\lambda^2)(x,y)= \pm\frac{i}{4}H_0^\pm(\lambda|x-y|),
   \end{equation}
where $H_0^\pm (z) $ are Hankel functions of order zero defined by
\begin{equation}\label{id-H0}
H_0^\pm(z)= J_0(z)\pm i Y_0(z),
\end{equation}
and $H_0^-(z)= \overline{H_0^+(z)}$. From the series expansions for Bessel functions
( see e.g. \cite{AS64} ), we have that as $z \rightarrow 0, $
\begin{equation}\label{id-J0}
J_0(z)= 1- \frac{1}{4}z^2 + \frac{1}{64}z^4 -\frac{1}{2304}z^6 +O(z^8),
\end{equation}
\begin{equation}\label{id-Y0}
   \begin{split}
	 Y_0(z)=&\frac{2}{\pi}\big(\ln(\frac{z}{2})+\gamma\big) J_0(z)
              +\frac{2}{\pi}\big( \frac{1}{4}z^2 - \frac{3}{128}z^4 +\frac{11}{13824}z^6 +O(z^8)  \big),
		\end{split}
	\end{equation}
where $\gamma $ is Euler's constant. In addition, for $|z|>1$, we have the representation ( see e.g. \cite{AS64} )
\begin{equation}\label{id-H0big}
H^\pm_0(z)=e^{\pm iz}w_\pm(z),\ \ |w_\pm^{(\ell)}(z)|\lesssim (1+|z|)^{-\frac{1}{2}-\ell},
\ell= 0,1,2,\cdots.
\end{equation}
Noting that
\begin{equation*}
   \begin{split}
 R^\pm_0(\lambda^4)=\frac{1}{2\lambda^2}\big( R^\pm(-\Delta; \lambda^2)-R(-\Delta; -\lambda^2) \big),
  \ \lambda>0.
  \end{split}
\end{equation*}
Hence if $ \lambda|x-y|>1$, then by using (\ref{reslolent-R}) and (\ref{id-H0big}) it follows that
\begin{equation}\label{reso-big}
   \begin{split}
R_0^\pm(\lambda^4)(x,y)
  =&   \frac{i}{8\lambda^2}\Big[ \pm e^{\pm i\lambda |x-y|}w_\pm(\lambda|x-y|)
        - e^{- \lambda |x-y|}w_+(  i\lambda|x-y|)\Big] .
  \end{split}
\end{equation}

Now we compute the representation of $R^\pm_0(\lambda^4)(x,y)$ when $\lambda|x-y|\le 1 $.
Let $G_k$ denote the integral operators with the following kernels $G_k(x,y)$:
\begin{equation}\label{def-G0-G3}
   \begin{split}
   &G_{-1}(x,y)= |x-y|^2, \,
   G_0(x,y)= a_0|x-y|^2\ln|x-y|, \,
   G_1(x,y)= |x-y|^4, \,\\
  &G_2(x,y) = |x-y|^6, \,
   G_3(x,y)= a_2|x-y|^6\ln|x-y|,
     \end{split}
 \end{equation}
where the coefficients $a_0= \frac{1}{8\pi}$ and $ a_2= \frac{1}{4608\pi}$.

\begin{lemma}\label{lem-reso}
If $ \lambda|x-y| \ll 1$, then
\begin{equation}\label{free-R0pmlambda4}
   \begin{split}
	  R^\pm_0(\lambda^4)(x,y)=&\frac{b_\pm}{\lambda^2}
                       + g_0^\pm(\lambda) G_{-1}(x,y) + G_0(x,y) + c_\pm\lambda^2G_1(x,y)
             +\tilde{g}_2^\pm(\lambda) \lambda^4G_2(x,y)\\
             & + \lambda^4G_3(x,y)+ O\big(\lambda^{6-\epsilon}|x-y|^{8-\epsilon}\big),
		\end{split}
	\end{equation}
where
$ \displaystyle b_\pm = \pm \frac{i}{8}, c_\pm= \pm\frac{i}{512}$ ,
$\displaystyle g_0^\pm(\lambda)= a_0 \ln\lambda+ \alpha_\pm$,\ \
$\displaystyle \tilde{g}_2^\pm(\lambda)= a_2 \ln\lambda+\beta_\pm, $
 and $\alpha_\pm ,\beta_\pm \in \mathbb{C}$.
\end{lemma}
\begin{proof}
Let $r=|x-y|$.
By using (\ref{reslolent-R} )-(\ref{id-Y0}), we have
\begin{equation*}
  \begin{split}
&R^\pm(-\Delta;\lambda^2)(x,y)=\frac{i}{4}H^\pm_0(\lambda r)\\
=& \Big(- \frac{1}{2\pi}\ln \lambda \pm \frac{i}{4} + \frac{\ln 2}{2\pi} - \frac{\gamma}{2\pi}\Big) -\frac{1}{2\pi}\ln r
+\Big(\frac{1}{8\pi}\ln\lambda \mp \frac{i}{16}- \frac{\ln2 - \gamma +1}{8\pi}\Big)\lambda^2r^2\\
&+ \frac{1}{8\pi}\lambda^2r^2\ln r
+ \Big(-\frac{1}{128\pi} \ln\lambda \pm \frac{i}{256}
  +\frac{2\ln2 -2\gamma +3 }{256\pi} \Big)\lambda^4r^4
  - \frac{1}{128\pi}\lambda^4r^4\ln r\\
&+ \Big(\frac{1}{4608\pi}\ln\lambda \mp \frac{i}{9216}
-\frac{\ln2}{4608\pi} +\frac{\gamma }{4608\pi}-\frac{11}{27648\pi}  \Big)\lambda^6r^6
+ \frac{1}{4608\pi}\lambda^6 r^6\ln r
+O\big((\lambda r)^{8-\epsilon}\big),
  \end{split}
\end{equation*}
\begin{equation*}
  \begin{split}
&R^+(-\Delta;-\lambda^2)(x,y)=\frac{i}{4}H^+_0(i\lambda r)\\
=& \Big(-\frac{1}{2\pi}\ln \lambda + \frac{\ln 2-\gamma}{2\pi} \Big) -\frac{1}{2\pi}\ln r
+\Big(-\frac{1}{8\pi}\ln\lambda + \frac{\ln2 - \gamma +1}{8\pi}\Big)\lambda^2 r^2
-\frac{1}{8\pi}\lambda^2 r^2\ln r \\
&+ \Big(-\frac{1}{128\pi} \ln\lambda
 + \frac{2\ln2 -2\gamma +3 }{256\pi} \Big)\lambda^4r^4
  -\frac{1}{128\pi}\lambda^4r^4\ln r
+ \Big(- \frac{1}{4608\pi}\ln\lambda+\frac{\ln2 - \gamma }{4608\pi}\\
&+\frac{11}{27648\pi} \Big)\lambda^6r^6
- \frac{1}{4608\pi}\lambda^6r^6\ln r
+O\big((\lambda r)^{8-\epsilon}\big).
  \end{split}
\end{equation*}
By using the identity,
\begin{equation*}
   \begin{split}
 R^\pm_0(\lambda^4)=\frac{1}{2\lambda^2}\big( R^\pm(-\Delta; \lambda^2)-R(-\Delta; -\lambda^2) \big),
  \, \lambda>0,
  \end{split}
\end{equation*}
then the desired conclusion \eqref{def-G0-G3} holds.
\end{proof}
\begin{lemma}\label{lem-reso-small+big}
 We obtain the following expansions of $R^\pm_0(\lambda^4)(x,y)$:
\begin{equation}
   \begin{split}
	  R^\pm_0(\lambda^4)(x,y)=& \frac{b_\pm }{\lambda^2}I(x,y)
                       + g_0^\pm(\lambda) G_{-1}(x,y) + G_0(x,y) +E^\pm_0(\lambda)(x,y),
		\end{split}
	\end{equation}
where the error term satisfies $ E^\pm_0(\lambda)(x,y)= O_1(\lambda^{\ell}|x-y|^{2+\ell})$ for
$0<\ell \leq 2$.
\end{lemma}
\begin{proof}
We write that
\begin{equation}\label{label-hl-1}
   \begin{split}
R^\pm_0(\lambda^4)(x,y)= \chi(\lambda|x-y|)R^\pm_0(\lambda^4)(x,y)
          +\widetilde{\chi}(\lambda|x-y|) R^\pm_0(\lambda^4)(x,y).
   \end{split}
\end{equation}
Let $r=|x-y|$. If $\lambda r \ll 1$, then the statement is clear by (\ref{free-R0pmlambda4}).
If $ \lambda r \gtrsim 1 $, noting that
$$  g_0^\pm(\lambda) G_{-1}(x,y) + G_0(x,y)
=  |x-y|^2\Big(a_0\ln(\lambda|x-y|)+\alpha_\pm \Big), $$
by using (\ref{reso-big}) and (\ref{free-R0pmlambda4}), then it follows that for $\ell>0$,
\begin{equation}
   \begin{split}
&\Big| R^+_0(\lambda^4)(x,y)- \Big(  b_+\frac{1}{\lambda^2}+ g_0^+(\lambda) G_{-1}(x,y)
       + G_0(x,y) \Big)\Big|\\
= & \Big|\frac{1}{\lambda^2}\Big(\frac{i}{8}e^{i\lambda r}w_+(\lambda r )
   - \frac{i}{8}e^{-\lambda r}w_+(i\lambda r )-b_+\Big)-
       \Big( g_0^+(\lambda) G_{-1}(x,y)+ G_0(x,y) \Big)  \Big|\\
  \lesssim & \frac{1}{\lambda^2}\Big(|w_+(\lambda r )| +|w_+(i\lambda r )|+ c\Big)
     +\lambda^\ell  r^{2+\ell}\\
   \lesssim & \frac{1}{\lambda^2(1+\lambda r)^{\frac{1}{2}}}+ \frac{1}{\lambda^2}
     + \lambda^\ell  r^{2+\ell}
    \lesssim  \lambda^\ell r^{2+\ell}.
	\end{split}
\end{equation}
 Similarly, by the identity (\ref{label-hl-1}) we obtain that for $ 0<\ell \leq 2$,
 \begin{equation*}
   \begin{split}
\big| \partial_\lambda \big( R^+_0(\lambda^4)(x,y)\big)\big|
 \lesssim & \big|\partial_\lambda \big(R^+_0(\lambda^4)(x,y)\big)\chi(\lambda|x-y|) \big|
    + \big| \partial_\lambda \big(R^+_0(\lambda^4)(x,y)\big) \widetilde{\chi}(\lambda|x-y|) \big|\\
   \lesssim & \lambda^{-1}r^2 + \lambda r^4 \lesssim \lambda^{\ell-1}r^{2+\ell}.
	\end{split}
\end{equation*}
Thus the proof of this lemma is completed.
\end{proof}
\subsection{The resolvent expansions near zero }
In this subsection, we study the asymptotic expansions of the perturbed resolvent in a neighborhood of zero threshold.

Let $U(x)=\hbox{sign}\big(V(x)\big)$ and $v(x)=|V(x)|^{1/2}$, then we have $ V=Uv^2$ and
the following symmetric resolvent identity:
 \begin{equation}\label{id-RV}
R^\pm_V(\lambda^4) = R_0^\pm(\lambda^4) -R^\pm_0(\lambda^4)v(M^\pm(\lambda))^{-1}vR^\pm_0(\lambda^4),
\end{equation}
where $M^{\pm}(\lambda)=U+ vR^\pm_0(\lambda^4)v$. So, we will need to establish the expansions for
$(M^\pm(\lambda))^{-1}$.

Let $T_0= U+vG_0v$ and $P= \|V\|^{-1}_{L^1} v\langle v, \cdot \rangle$ denote the orthogonal projection onto the span space by $v$. Then we have the following expansions of $M^\pm(\lambda)$.
\begin{lemma}\label{lem-M}
Let $\displaystyle a_\pm= \pm\frac{i}{8} \|V  \|_1$ and  $ M^\pm(\lambda)= U+vR^\pm_0(\lambda^4)v$.
If $|V(x)|\lesssim (1+|x|)^{-\beta}$ with some $\beta > 0$, then the following identities of $M^\pm(\lambda)$ hold on $\mathbb{B}(L^2,L^2)$ for $\lambda>0$:

 (i) If $\beta > 10$, then we have
 \begin{equation}\label{Mpm-1}
   \begin{split}
   M^\pm(\lambda)= \frac{a_\pm}{\lambda^2}P +g_0^\pm(\lambda)v G_{-1} v +T_0
                  +E^\pm_0(\lambda),
     \end{split}
 \end{equation}
 where the error term satisfies $E_0^\pm(\lambda)=O_1\big(\lambda^2 v(x)|x-y|^4v(y)\big)$.

  (ii) If $\beta > 14$, then we have
 \begin{equation}\label{Mpm-2}
   \begin{split}
   M^\pm(\lambda)=& \frac{a_\pm}{\lambda^2}P +g_0^\pm(\lambda)v G_{-1} v +T_0+ c_\pm\lambda^2 vG_1 v
                 +E_1^\pm(\lambda),
     \end{split}
 \end{equation}
  where the error term satisfies $E_1^\pm(\lambda)=O_1\big(\lambda^{4-\epsilon}v(x)|x-y|^{6-\epsilon}v(y)\big)$.

  (iii) If $\beta > 18$, then we have
 \begin{equation}\label{Mpm3}
   \begin{split}
   M^\pm(\lambda)=& \frac{a_\pm}{\lambda^2}P +g_0^\pm(\lambda)v G_{-1} v +T_0+ c_\pm\lambda^2 vG_1 v
                   + \tilde{g}_2^\pm(\lambda)\lambda^4vG_2v+  \lambda^4 vG_3v
     +E_2^\pm(\lambda),
     \end{split}
 \end{equation}
 where the error term satisfies $E_2^\pm(\lambda)=O_1\big(\lambda^{6-\epsilon}v(x)|x-y|^{8-\epsilon}v(y)\big)$.
\end{lemma}
\begin{proof}
We only prove (\ref{Mpm-1}) since the proofs of other conclusions are similar. Let $r=|x-y|$. If $\lambda r \ll 1$, then the statement is clear by (\ref{free-R0pmlambda4}).
If $ \lambda r \gtrsim 1 $, using (\ref{reso-big}) and (\ref{free-R0pmlambda4}), by the same arguments with
the proof of Lemma \ref{lem-reso-small+big}, we have
\begin{equation}
   \begin{split}
\big| M^\pm(\lambda) -\frac{a_\pm}{\lambda^2}P -g_0^\pm(\lambda)v G_{-1} v -T_0 \big) \big|
\lesssim \lambda^2 v(x)|x-y|^4 v(y),
	\end{split}
\end{equation}
\begin{equation}
   \begin{split}
\big| \partial_\lambda \big( M^\pm(\lambda) -\frac{a_\pm}{\lambda^2}P -g_0^\pm(\lambda)v G_{-1} v -T_0 \big) \big| \lesssim \lambda v(x)|x-y|^4 v(y).
	\end{split}
\end{equation}
Since the error term is a Hilbert-Schmidt operator if $|V(x)| \lesssim (1+|x|)^{-\beta}$ with $\beta>10$,  so we get (\ref{Mpm-1}) as bounded operators on $L^2(\mathbb{R}^2)$.
\end{proof}

Now we define the type of resonances that may occur at the zero threshold as follows:

Let $T_0= U+vG_0v$, $P= \|V\|^{-1}_{L^1} v\langle v, \cdot \rangle$, $ Q=I-P$ and $S_0$ be the orthogonal projection onto the kernel of $ QvG_{-1}vQ$
on $QL^2$.  By Lemma \ref{projiction-spaces-SjL2} (i) we have $S_0\neq 0$. Hence, $QvG_{-1}vQ$ is not invertible on $QL^2$. But $QvG_{-1}vQ+S_0$ is invertible on $QL^2$, in this case, we
define $D_0=(QvG_{-1}vQ+S_0)^{-1}$ as an operator on $QL^2$.
\begin{definition}\label{definition of resonance}
(i) If $S_0T_0S_0$ is invertible on $S_0L^2$, then we say that zero is a regular point of the spectrum of $H$. In this case, we define $D_1= (S_0T_0S_0)^{-1}$ as an operator on $S_0L^2$.

(ii) Assume that $S_0T_0S_0$ is  not invertible on $S_0L^2.$ Let $S_1$ be the Riesz projection onto the kernel of $S_0T_0S_0.$ Then $S_0T_0S_0+S_1$ is invertible on $S_0L^2$.  In this case, we define $D_1=\big(S_0T_0S_0+S_1\big)^{-1}$ as an operator on $S_0L^2$, which doesn't conflict with the previous definition since $S_1=0$ when zero is regular.
We say zero is a resonance of the first kind  if the operator
\begin{equation}\label{T_1}
T_1:= -S_1T_0(Q-S_0) \Big((Q-S_0)vG_{-1}v(Q-S_0)\Big)^{-1}(Q-S_0)T_0S_1
\end{equation}
is invertible on $S_1L^2$, where the inverse $\big((Q-S_0)vG_{-1}v(Q-S_0)\big)^{-1}$ is well-defined on the subspace $(Q-S_0)L^2$ by the the \eqref{positive-matrix} of Lemma \ref{lem-d-inver}.

(iii)~Assume that $T_1$ is not invertible on $S_1L^2.$ Let $S_2$ be the Riesz projection onto the kernel of $T_1.$ Then $T_1+S_2$ is invertible on $S_1L^2$.
 In this case, we define $D_2=\big(T_1+S_2\big)^{-1}$ as an operator on $S_1L^2$.
Let $T_2:= S_2vG_{-1}vvG_{-1}vS_2.$
 then $T_2$ is not invertible on $S_2L^2$ by Lemma \ref{projiction-spaces-SjL2}(iv).
 Let $S_3$ be the Riesz projection onto the kernel of $T_2$, then $T_2+S_3$ is invertible on $S_2L^2$, in this case, we define $D_3=(T_2+S_3)^{-1}$ as an operator on $S_2L^2$.
We say zero is a resonance of the second kind  if
\begin{equation}\label{T_3}
T_3:=S_3T_0^2S_3-c_\pm a_\pm S_3vG_1vS_3
\end{equation}
is invertible on $S_3L^2$.

(iv)~Assume that $T_3$ is not invertible on $S_3L^2$.  Let $S_4$ be the Riesz projection onto the kernel of $T_3.$ Then $T_3+S_4$ is invertible on $S_3L^2$.
In this case, we define $D_4=\big(T_3+S_4\big)^{-1}$ as an operator on $S_3L^2$.
We say zero is a resonance of  the third kind  if
\begin{equation}\label{T_4}
T_4:=S_4vG_2vS_4
\end{equation}
is invertible on $S_4L^2$.

(v)~ Finally,  if $T_4$ is not invertible on $S_4L^2$,
we say there is a resonance of the fourth kind at zero. In this case, the operator $T_5:=S_4vG_3vS_4$ is always invertible on $S_5L^2$ where $S_5$ be the Riesz projection onto the kernel of $T_4,$ see Lemma \ref{lemma-S5vG3vS5}. We define $D_5=(T_4+S_5)^{-1}$ as an operator on $S_4L^2$.
\end{definition}

From the definition above, we have $ S_0L^2 \supseteq S_1L^2 \supseteq S_2L^2
\supseteq S_3L^2 \supseteq S_4L^2 \supseteq S_5L^2$, which describe the zero energy resonance types of $H$ as follows:
\begin{itemize}
\item ~zero is a regular point of $H$ if and only if $S_1L^2(\mathbb{R}^2)=\{0\};$
\vskip0.2cm
\item~ zero is a first kind resonance of $H$ if and only if $S_1L^2(\mathbb{R}^2)\neq\{0\}$ and $S_2L^2(\mathbb{R}^2)=\{0\};$
\vskip0.2cm
\item~zero is a second kind resonance of $H$ if and only if $S_2L^2(\mathbb{R}^2)\neq\{0\}$ and $S_4L^2(\mathbb{R}^2)=\{0\};$
\vskip0.2cm
\item~zero is a third kind resonance of $H$ if and only if $S_4L^2(\mathbb{R}^2)\neq\{0\}$ and $S_5L^2(\mathbb{R}^2)=\{0\};$
\vskip0.2cm
\item~zero is an eigenvalue of $H$ ( i.e. the fourth kind resonance ) if and only if $S_5L^2(\mathbb{R}^2)\neq\{0\}.$
\end{itemize}
 In view of the statements above, we notice that the space $S_0L^2$ and $S_3L^2$  do not directly correspond to the resonance types, but as transient paces, they  actually play important roles in establishing the expansions of $M(\lambda)^{-1}$ if zero is the regular point and second resonance type ( see e.g.  Section \ref{the proof of inverse Pro} for the more details ).

Furthermore, since $vG_0v$ is a Hilbert-Schmidt operator, and $T_0$ is a compact perturbation of $U$ (see e.g. \cite{Erdogan-Goldberg-Green-14} and \cite{Green-Toprak19}). Hence $S_1$ is a finite-rank projection by the Fredholm alternative theorem. Notice that $S_5\leq S_4\leq S_3\leq S_2\leq S_1\leq S_0\leq Q$,  then all $S_j(j=1,\cdots,5)$ are finite-rank operators. Moreover, by the definitions of $S_j(j=0,1,\cdots,5)$, we have that
$S_iD_j=D_jS_i=S_i (i\geq j)$ and $S_iD_j=D_jS_i=D_j (i< j)$.

\begin{definition}
We say an operator $T:\,L^2(\mathbb{R}^2)\rightarrow L^2(\mathbb{R}^2)$ with kernel $T(\cdot, \cdot)$
 is absolutely bounded if the operator with the kernel $|T(\cdot, \cdot)|$ is bounded from $L^2(\mathbb{R}^2)$ into itsself.
\end{definition}

We remark that Hilbert-Schmidt and finite-rank operators are absolutely bounded operators. Moreover, we have the following proposition, which is similar to Lemma 8 in  \cite{Schlag-CMP} or Lemma 4.3 in\cite{Erdogan-Green-Toprak}.
\begin{proposition}\label{Pro-absulu-oper}
Let $|V(x)|\leq (1+|x|)^{-\beta}$ with $\beta>10$, then $S_0D_1S_0$ is absolutely bounded.
\end{proposition}
\begin{proof}
Recall that $D_1=(S_0T_0S_0+S_1)^{-1}$. Here we only prove the case $S_1=0$. Let $f\in S_0L^2$ and $S_0Uf=0$. Then $Uf=c_0v+c_1x_1v+c_2x_2v$, where $c_0, c_1$ and $c_2$ are some constants. Since $U^2=1$, then $f=c_0Uv+c_1Ux_1v+c_2Ux_2v$. Note that $f\in S_0L^2$ implies
$$ \langle f, v \rangle = \langle f, x_1v \rangle=\langle f, x_2v\rangle=0.$$
Hence, we have
\begin{equation}\label{equs-1}
\begin{cases}
c_0\langle Uv, v  \rangle+ c_1\langle Ux_1v,v \rangle+c_2\langle Ux_2v, v  \rangle=0,\\
c_0\langle Uv, x_1v  \rangle+ c_1\langle Ux_1v,x_1v \rangle+c_2\langle Ux_2v, x_1v  \rangle=0,\\
c_0\langle Uv, x_2v  \rangle+ c_1\langle Ux_1v,x_2v \rangle+c_2\langle Ux_2v, x_2v  \rangle=0.
\end{cases}
\end{equation}
Let $c=(c_0, c_1, c_2)^T$ and
$$A=
\begin{pmatrix}
\int V(x)dx & \int x_1V(x)dx &\int x_2V(x)dx\\
\int x_1V(x)dx & \int x_1^2V(x)dx &\int x_1x_2V(x)dx\\
\int x_2V(x)dx & \int x_1x_2V(x)dx &\int x_2^2V(x)dx
 \end{pmatrix}.
 $$
Notice that $Uv^2=V$, then $Ac=0$. Hence
 ker$_{S_0L^2}(S_0US_0)=\{0 \}$ if and only if $\det(A)\neq 0$.

\textbf{Case 1.} \  If $\det(A) \neq 0$, then we claim that $S_0US_0$ is invertible on $S_0L^2$. More precisely, by using the Cramer's Rule
we obtain that the solution of equation $S_0US_0f=g$ for any $g\in S_0L^2$, is expressed as
$$f=Ug+c_0Uv+c_1Ux_1v+c_2Ux_2v$$
with
$$c_0= \frac{\det(A_1)}{\det(A)}, \, c_1= \frac{\det(A_2)}{\det(A)}, \, c_3= \frac{\det(A_3)}{\det(A)},$$
where $b=\big( -\langle Ug,v\rangle, -\langle Ug,x_1v\rangle, -\langle Ug,x_2v\rangle\big)^T$, and
$A_i(i=1,2,3)$ are the matrixes  by replacing the $i$ column of matrix $A$ with $b$.

It is easy to check from the explicit formula above that $ S_0(S_0US_0)^{-1}S_0$ is absolutely bounded. Moreover,
we have the following identity on $S_0L^2$,
\begin{equation}\label{ab-id1}
\big(S_0(U+vG_0v)S_0\big)^{-1}=\big(S_0US_0\big)^{-1}\big(S_0+S_0vG_0v (S_0US_0)^{-1}S_0\big)^{-1}.
\end{equation}
Notice that $vG_0v$ is a Hilbert-Schmidt operator when $|V(x)|\lesssim (1+|x|)^{-\beta}$ with $\beta>10$,
then $W:=S_0vG_0v(S_0US_0)^{-1}S_0$ is also Hilbert-Schmidt operator. Finally, notice that the following identity holds on $S_0L^2$:
\begin{equation*}
\big(S_0+S_0vG_0v (S_0US_0)^{-1}S_0\big)^{-1}-S_0= - \big(S_0+ W  \big)^{-1}W.
\end{equation*}
Since the right hand side of the identity above is a Hilbert-Schmidt operator, hence by the identity (\ref{ab-id1}) we know that
$\big(S_0(U+ vG_0v)S_0\big)^{-1}$ is the composition of an absolutely bounded operator with the sum of an absolutely bounded operator
and a Hilbert-Schmidt operator. Thus, $S_0D_1S_0$ is absolutely bounded.

\textbf{Case 2.} \  If $\det(A) = 0$, that is, $S_0US_0$ is not invertible on $S_0L^2$, then $S_0US_0 +\pi_0$ is invertible on $S_0L^2$ where $\pi_0$ is the Riesz projection onto the kernel of $S_0US_0$. By the same argument as the case 1 we obtain that $\big(S_0US_0 +\pi_0\big)^{-1}$ is absolutely
 bounded on $S_0L^2$. Finally, by the identity
 $$ \big(S_0(U+ vG_0v)S_0\big)^{-1}= \big( S_0US_0+\pi_0 + S_0vG_0vS_0 - \pi_0 \big)^{-1} $$
on $S_0L^2$, we still obtain the desired statement.
\end{proof}


In the following, we further give the specific characterizations
of projection spaces $S_jL^2$.
\begin{lemma}\label{projiction-spaces-SjL2}
Let $S_j(j=0,1,\cdots, 5)$ be the projection operators given in Definition \ref{definition of resonance}. Then

(i) $S_0L^2=  \{ f \in QL^2 \big| \langle x_jv, f\rangle=0,j=1,2 \}$.

(ii) $S_1L^2=\{ f\in S_0L^2 \big| S_0T_0f=0 \}$.

(iii) $S_2L^2=\{ f\in S_1L^2 \big| QT_0f=0 \}$.

(iv) $S_3L^2 = \{ f\in S_2L^2 \big| \langle  |x|^2v, f\rangle =0  \}.$

(v)  $S_4L^2= \{ f\in S_3L^2 \big| \langle x_ix_jv, f\rangle =0, i, j=1,2, \text{and}\ PT_0f =0 \}.$

(vi) $S_5L^2= \{ f\in S_4L^2 \big| \langle x_ix_jx_kv, f\rangle =0, i,j,k=1,2\}.$
\end{lemma}

\begin{remark}\label{rem-dim1}
We remark that these spaces $(S_1-S_2)L^2$, $(S_2-S_4)L^2$, $(S_4-S_5)L^2$ and $S_5L^2$  correspond to  each zero resonance type, respectively. In particular, we have that  the dimensions of $S_4L^2$ and $S_5L^2$ are finite from Remark \ref{rem5-2} below. The dimension
of $(S_1-S_4)L^2$ is $3$ by  Propositions \ref{lemma-spectral-S1L2} and  \ref{lemma-spectral-S3L2}), the dimension of
$(S_2-S_4)L^2$ is $1$ by Propositions \ref{lemma-spectral-S2L2} and  \ref{lemma-spectral-S3L2}, hence
$(S_1-S_2)L^2$ is  $2$.
The dimension of $(S_4-S_5)L^2$ is  $4$ by (v) and (vi) of Lemma \ref{projiction-spaces-SjL2}. As a consequence, $S_i(i=1,\cdots, 5)$ are finite rank operators.
\end{remark}
\begin{proof}
(i) Let $f\in S_0L^2$. Then we have
$$0=\langle QvG_{-1}vQ f, f\rangle = \langle vG_{-1}vf, f\rangle.$$
By using (\ref{def-G0-G3}) and the orthogonality $\langle v, f\rangle=0$,
one has
\begin{equation*}
	\begin{split}
0=&\langle v G_{-1}vf, f\rangle
= \int_{\mathbb{R}^2}v(x)(G_{-1}vf)(x)\overline{ f(x)}dx\\
=& \int_{\mathbb{R}^2}v(x)\overline{f(x)}\int_{\mathbb{R}^2}|x-y|^2v(y)f(y)dydx\\
=&\int_{\mathbb{R}^2}v(x)\overline{f(x)}\int_{\mathbb{R}^2}\big(|x|^2-2(x_1y_1+x_2y_2)+|y|^2\big)
         v(y)f(y)dydx\\
=& -2\Big(  \int_{\mathbb{R}^2}x_1v(x)f(x) dx\Big)^2
  - 2\Big(  \int_{\mathbb{R}^2}x_2v(x)f(x) dx\Big)^2,
	\end{split}
\end{equation*}
which leads to $\langle x_jv, f\rangle=0, j=1,2 $.

On the other hand, if $\langle x_jv, f\rangle=0( j=1,2) $ and $f\in QL^2$, we can obtain
$QvG_{-1}vQf=0$.

(ii) Let $f\in S_1L^2$. Then  $f\in S_0L^2$ and we have $S_0T_0f=0$ by the definition of $S_1$.

(iii) Let $f\in S_2L^2$, then $f\in S_1L^2$, therefore $S_0T_0f=0$.
Note that  $$ QT_0f  = QT_0f -S_0T_0f =(Q-S_0)T_0f,$$ we only need to prove $(Q-S_0)T_0 f =0 $.
Indeed, if $f\in S_2L^2=\hbox{ker}(T_1)$, then we have
\begin{equation}
  \begin{split}
0=&\langle -T_1f,f\rangle
=  \Big\langle S_1T_0(Q-S_0)\big((Q-S_0)vG_{-1}v(Q-S_0)\big)^{-1}(Q-S_0)T_0S_1f, \ f  \Big\rangle \\
=&\Big \langle\big((Q-S_0)vG_{-1}v(Q-S_0)\big)^{-1}(Q-S_0)T_0f,\,   (Q-S_0)T_0f  \Big\rangle.
\end{split}
\end{equation}
Since $ \big((Q-S_0)vG_{-1}v(Q-S_0)\big)^{-1}<0 $ on the space $(Q-S_0)L^2$ by the \eqref{positive-matrix} of Lemma \ref{lem-d-inver}, so we must have $ (Q-S_0)T_0f = 0$.

On the other hand ,  if $f\in QL^2$ and $\langle x_jv, f\rangle=0(j=1,2) $ and $QT_0 f = 0 $,  then it is easy to check that $\langle T_1f, f\rangle=0 $.

(iv) Let $f\in S_3L^2:= \hbox{ker}(T_2) =\hbox{ker}(S_2vG_{-1}vvG_{-1}vS_2)$.
Then we have
\begin{align*}
0=&\langle S_2vG_{-1}vvG_{-1}vS_2 f, f\rangle=\langle vG_{-1}vf, vG_{-1}vf\rangle.
\end{align*}
Note that the orthogonality $\langle v, f\rangle=0$ and $\langle x_jv, f\rangle=0(j=1,2)$, by (\ref{def-G0-G3}) we have
\begin{equation*}
	\begin{split}
0=&\langle v G_{-1}vf, v G_{-1}vf\rangle=\int_{\mathbb{R}^2}\Big| v(x)(G_{-1}vf)(x) \Big|^2 \, dx\\
=& \int_{\mathbb{R}^2}\Big| v(x)\Big(\int_{\mathbb{R}^2}|x-y|^2v(y)f(y)\,dy\Big)\Big|^2 \,dx\\
 =&\int_{\mathbb{R}^2} |v(x)|^2 \Big(\int_{\mathbb{R}^2}|y|^2 v(y)f(y) dy \Big)^2dx,\\
 =&\langle |x|^2v, f \rangle^2\|V\|_{L^1},
	\end{split}
\end{equation*}
which leads to $ \langle |x|^2v, f \rangle=0$.

On the other hand, if $f\in S_2L^2$ and $ \langle |x|^2v, f \rangle=0$, we can obtain $ T_2f=0 $.

(v)
Let $f\in S_4L^2$. Note that $QT_0f=0$ and $P=\frac{\langle v, \cdot \rangle}{\|V\|_{L^2}}v$, we have
$$S_3T_0^2S_3f(x)= S_3T_0(PT_0f)(x)= \frac{\langle v, T_0f \rangle}{\|V\|_{L^1}} S_3(T_0 v)(x), $$
which leads to
\begin{equation}\label{S3T0S3}
  \begin{split} \langle S_3T_0^2S_3f, f\rangle
  =&\frac{\langle v, T_0f \rangle}{\|V\|_{L^1}}  \langle S_3(T_0 v), f\rangle
=\frac{\langle v, T_0f \rangle}{\|V\|_{L^1}}  \langle  v, T_0S_3f\rangle
=\frac{\langle v, T_0f \rangle^2}{\|V\|_{L^1}}.
\end{split}
\end{equation}
Since
$$|x-y|^4 = |x|^4 -4|x|^2(x_1y_1 +x_2y_2)-4|y|^2(x_1y_1+x_2y_2) + 4(x_1^2y_1^2+ x_2^2y_2^2+2x_1x_2y_1y_2)
+2|x|^2|y|^2 +|y|^4,$$
and (\ref{def-G0-G3}), one has
\begin{equation}\label{S3vG1vS3}
  \begin{split}
\langle S_3vG_1vS_3f, f \rangle = & \langle vG_1vf, f  \rangle\\
=&\int_{\mathbb{R}^2}\Big(\int_{\mathbb{R}^2}|x-y|^4v(y)f(y)dy\Big) v(x)\overline{f(x)}dx\\
 =& 4\langle x_1^2v, f \rangle^2 +4\langle x_2^2v, f \rangle^2+8\langle x_1x_2v, f \rangle^2.
   \end{split}
\end{equation}
Combing (\ref{S3T0S3}) and (\ref{S3vG1vS3}), we have
\begin{equation*}
  \begin{split}
0=\langle T_3f, f \rangle
=&\frac{\langle v, T_0f \rangle^2}{\|V\|_{L^1}} + \frac{\|V\|_{L^1}}{1024}\langle x_1^2v, f \rangle^2
+\frac{\|V\|_{L^1}}{1024}\langle x_2^2v, f \rangle^2+\frac{\|V\|_{L^1}}{512}\langle x_1x_2v, f \rangle^2 \geq 0,
   \end{split}
\end{equation*}
which implies that  $\langle v, T_0f \rangle=0 $ and
$ \langle x_ix_jv, f \rangle= 0(i,j=1,2 )$. Note that,
$ \langle v, T_0f \rangle= 0 $ means $ PT_0f=0$ by $P= \frac{\langle v, \cdot \rangle}{\|V\|_{L^1}} v$,
we obtain that $ T_0f =PT_0f +QT_0f=0$.

On the other hand, if  $f\in S_3L^2$ and $PT_0f=0$,  it is easy to check that
$f\in S_4L^2$.

(vi) Let $f\in S_5L^2$. We only need to prove $ \langle x_ix_jx_kv, f\rangle =0,\
i,j,k=1,2  $. Indeed, notice that
\begin{equation*}
   \begin{split}
|x-y|^6=& |x|^6+3|x|^4|y|^2-6|x|^4x\cdot y+12|x|^2(x\cdot y)^2 -12|x|^2|y|^2x\cdot y -8(x\cdot y)^3\\
&+12|y|^2(x\cdot y)^2-6|y|^4x\cdot y +3|x|^2|y|^4 +|y|^6,
\end{split}
\end{equation*}
by using  $ \langle v, f\rangle =\langle x_iv, f \rangle
= \langle x_ix_jv, f\rangle=0, 1 \leq i,j \leq 2$, we have
\begin{equation*}
   \begin{split}
0=&\langle S_4vG_2vS_4 f,  f \rangle= \langle vG_2vf,  f \rangle\\
=&-12 \int_{\mathbb{R}^2}v(x)\int_{\mathbb{R}^2} |x|^2|y|^2(x\cdot y )v(y)f(y)dy  \overline{f(x)}dx
- 8\int_{\mathbb{R}^2}v(x)\int_{\mathbb{R}^2} (x\cdot y)^3 v(y)f(y)dy  \overline{f(x)}dx.
\end{split}
\end{equation*}
Note that $ (x\cdot y)^2= x_1^2y_1^2+ 2x_1x_2y_1y_2+ x_2^2y_2^2$,
$ (x\cdot y)^3= x_1^3y_1^3+ 3x_1^2y_1^2x_2y_2+ 3x_1y_1x_2^2y_2^2+x_2^3y_2^3$,
then
\begin{equation*}
   \begin{split}
0=&-12 \sum_{j=1}^2 \Big|\int_{\mathbb{R}^2} |x|^2x_j v(x)f(x)dx\Big|^2
-8\sum_{j=1}^2\Big|\int_{\mathbb{R}^2} x_j^3 v(x)f(x)dx\Big|^2\\
&-24\Big| \int_{\mathbb{R}^2} x_1^2x_2 v(x)f(x)dx\Big|^2
-24\Big|\int_{\mathbb{R}^2} x_1x_2^2v(x)f(x)dx\Big|^2,
\end{split}
\end{equation*}
which leads to $ \langle x_ix_jx_kv, f\rangle=0,\, 1\leq i,j,k\leq 2$.

On the other hand, if $ f\in S_4L^2$ and $ \langle x_ix_jx_kv, f\rangle=0(1 \leq i,j,k \leq 2)$,
it is easy to check that $ S_4vG_2vS_4f=0$.

Thus the proof of this lemma is completed.
\end{proof}

Throughout the paper, $\Gamma^k_{i,j}=\Gamma^k_{i,j}(\lambda)$ denote the absolutely bounded operators on $L^2(\mathbb{R}^2)$ which depend on $\lambda$ and satisfy
\begin{equation*}
\big\|\Gamma^k_{i,j}(\lambda)\big\|_{L^2\rightarrow L^2}=O_1(1),\, 0<\lambda \ll 1.
\end{equation*}
For the asymptotic expansions of $\big(M^\pm(\lambda)\big)^{-1}$
near $\lambda=0$, we have the following theorem.
\begin{theorem}\label{thm-main-inver-M}
Let $g_0^\pm(\lambda)$ and $\widetilde{g}_2^\pm(\lambda)$ be functions defined in Lemma \ref{lem-reso}.  Suppose that $|V(x)| \lesssim (1+|x|)^{-\beta}$ with some $\beta>0$. Then for $0<\lambda \ll 1$,  the following
conclusions hold:

(i) If zero is a regular point and $\beta > 10$, then
\begin{equation}\label{thm-regularinver-M0 }
	\begin{split}
\big(M^\pm(\lambda)\big)^{-1}=& S_0D_1S_0+ g_0^\pm(\lambda)^{-1}Q\Gamma^0_{0,1}Q
  + g_0^\pm(\lambda)^2\lambda^2S_0\Gamma^0_{2,1}S_0\\
&+ g_0^\pm(\lambda)\lambda^2\Big( S_0\Gamma^0_{2,2}Q+Q\Gamma^0_{2,3}S_0\Big)
 + O_1(\lambda^2).
\end{split}
\end{equation}

(ii) If zero is a resonance of the first kind and $\beta > 14$, then
\begin{equation}\label{thm-resoinver-M1 }
	\begin{split}
\big(M^\pm (\lambda)\big)^{-1}
 = &g^\pm_0(\lambda) S_1D_2 S_1+\Big( S_0\Gamma_{0,1}^1 S_0 + S_1\Gamma_{0,2}^1 Q
+ Q\Gamma_{0,3}^1 S_1\Big)+g^\pm_0(\lambda)^{-1}Q\Gamma_{0,4}^1Q\\
&+g^\pm_0(\lambda)^4\lambda^2S_1\Gamma_{2,1}^1S_1
+g^\pm_0(\lambda)^3\lambda^2\Big(S_1\Gamma_{2,2}^1Q
+ Q\Gamma_{2,3}^1S_1 \Big)
+g^\pm_0(\lambda)^2\lambda^2\\
&\Big(Q \Gamma_{2,4}^1Q + S_1\Gamma_{2,5}^1 +\Gamma_{2,6}^1S_1 \Big)
+g^\pm_0(\lambda)\lambda^2\Big( Q\Gamma_{2,7}^1 +\Gamma_{2,8}^1Q  \Big)+ O_1(\lambda^2).
\end{split}
\end{equation}

(iii) If zero is a resonance of the second kind and $\beta > 18$, then
\begin{equation}\label{thm-resoinver-M2 }
  \begin{split}
\big(M^\pm(\lambda)\big)^{-1}
=&\frac{h^\pm(\lambda)}{\lambda^2} S_2\Gamma^2_{-2,1}S_2
+g^\pm_0(\lambda)^5\Big( S_2\Gamma^2_{0,1} +\Gamma^2_{0,2}S_2 +
S_1\Gamma^2_{0,3}Q + Q\Gamma^2_{0,4}S_1 \Big)\\
&+g^\pm_0(\lambda)^{10}\lambda^2 \Big( S_2\Gamma^2_{2,1}
+\Gamma^2_{2,2}S_2 +S_1\Gamma^2_{2,3}Q + Q\Gamma^2_{2,4}S_1 \Big)\\
&+g^\pm_0(\lambda)\lambda^2\Big( S_2\Gamma^2_{2,5} +\Gamma^2_{2,6}S_2
 +S_1\Gamma^2_{2,7}Q + Q\Gamma^2_{2,8}S_1 \Big)+  O_1(\lambda^2),
\end{split}
\end{equation}
where $h^\pm(\lambda)= \Big(c_1+ g^\pm_0(\lambda)^{-1}c_2 + g^\pm_0(\lambda)^{-2}c_3\Big)^{-1}$,
$ c_1$ and $c_2$  are constants defined in Lemma \ref{lem-d1-inver}.
Moreover, $ \displaystyle  h^\pm(\lambda)=O_1(1)$ when $\lambda$ is small enough.

(iv) If zero is a resonance of the third kind and $\beta > 18$, then
\begin{equation}\label{thm-resoinver-M3}
   \begin{split}
\big(M^\pm(\lambda) \big)^{-1}
=& \frac{ \widetilde{g}_2^\pm(\lambda)^{-1} }{\lambda^4}S_4\Gamma_{-4,1}^3S_4 -\frac{\widetilde{g}_2^\pm(\lambda)^{-2}}{\lambda^4}S_4\Gamma_{-4,2}^3S_4
 +O_1\big(\lambda^{-4}(\ln\lambda)^{-3}\big).
\end{split}
\end{equation}

(v) If zero is a resonance of the fourth kind ( i.e. zero eigenvalue )and $\beta > 18$, then
\begin{equation}\label{thm-resoinver-M4}
   \begin{split}
\big(M^\pm(\lambda)\big)^{-1}
= &\frac{1}{\lambda^4}S_5D_6S_5 +\frac{\widetilde{g}_2^\pm(\lambda)^{-1}}{\lambda^4} \Big(S_4\Gamma_{-4,1}^4S_4\Big)
 +\frac{\widetilde{g}_2^\pm(\lambda)^{-2}}{\lambda^4}S_4\Gamma_{-4,2}^4S_4\\
&+ O_1\big(\lambda^{-4}(\ln\lambda)^{-3}\big).
\end{split}
\end{equation}
Here $S_j(j=0,1,\cdots,5)$ and $D_j(j=1,\cdots)$ are the operators given in Definition \ref{definition of resonance},
$\Gamma^k_{i,j}$ denote the absolutely bounded operators on $L^2$ depending on $\lambda$ and satisfying
\begin{equation*}
\big\|\Gamma^k_{i,j}(\lambda)\big\|_{L^2\rightarrow L^2}=O_1(1),\, 0<\lambda \ll 1.
\end{equation*}
\end{theorem}
As we will see, since the degenerate term $g_0^\pm(\lambda) G_{-1}(x,y)$ with the logarithm factor occurs in the expansions of the free resolvent $R^\pm_0(\lambda^4)(x,y)$ in $\mathbb{R}^2$, it leads to quite complicate calculations for the inversion expansion of $\big(M^\pm(\lambda)\big)^{-1}$. More seriously, such logarithm factor also lead that all expansions terms essentially have the highest singularity $O(\lambda^{-4})$ with only logarithm order differences in the cases of  the third kind and fourth kind resonances. Note that  the orthogonality of projections $S_j(j=1,2,3,4, 5)$ can not be used to smooth the singularity of the remainder terms in the expansions \eqref{thm-resoinver-M3} and \eqref{thm-resoinver-M4}, so this explains why we need a higher regular term $H^{3/2}$ in the lower energy decay estimates \eqref{thm-low-3} and \eqref{thm-low-4} of  Section \ref{proof of main results} below.  We will give the proof of this theorem in Section \ref{the proof of inverse Pro}.

\section{The proof of Theorem \ref{thm-main results}} \label{proof of main results}
In this section, we are devoted to proving Theorem \ref{thm-main results}.
Choosing a fixed even function $ \varphi \in C^\infty_c(\mathbb{R})$ such that
 $\varphi(s)=1$ for
$ |s|\leq \frac{1}{2}$ and $ \varphi(s)=0$ for $ |s| \geq 1$.
Let $\varphi_N(s)=\varphi(2^{-N}s)- \varphi(2^{-N+1}s),\ N\in \mathbb{Z}$. Then
$\varphi_N(s)=\varphi_0(2^{-N}s)$,
$\hbox{supp}\varphi_0 \subset [ \frac{1}{4}, 1]$, and
\begin{equation}\label{unit-deco}
\sum_{N=-\infty}^{\infty}\varphi_0(2^{-N}s)=1,\  s\in \mathbb{R}\setminus \{0\}.
\end{equation}
By using  Stone's formula, one has
\begin{equation}\label{Stone-Paley}
   \begin{split}
H^\alpha e^{-itH}P_{ac}(H)f=& \frac{2}{ \pi i } \int_0^\infty e^{-it\lambda^4}\lambda^{3+4\alpha}
[R_V^+(\lambda^4) -R_V^-(\lambda^4)]f d\lambda\\
= &\frac{2}{ \pi i } \int_0^\infty e^{-it\lambda^4}\lambda^{3+4\alpha}
\sum_{N=-\infty}^{\infty}\varphi_0(2^{-N}\lambda)[R_V^+(\lambda^4) -R_V^-(\lambda^4)]f d\lambda\\
=&\frac{2}{ \pi i }\int_0^\infty e^{-it\lambda^4}\lambda^{3+4\alpha}\chi(\lambda)[R_V^+(\lambda^4) -R_V^-(\lambda^4)]f d\lambda\\
&+\frac{2}{ \pi i } \int_0^\infty e^{-it\lambda^4}\lambda^{3+4\alpha}\widetilde{\chi}(\lambda)[R_V^+(\lambda^4) -R_V^-(\lambda^4)]f d\lambda,
\end{split}
\end{equation}
where $\chi(\lambda)=\sum\limits_{N=-\infty}^{N'}\varphi_0(2^{-N}\lambda)$  and
$\widetilde{\chi}(\lambda)=\sum\limits_{N=N'+1}^{+\infty}\varphi_0(2^{-N}\lambda)$ for some $ N'<0$. We remark that the choice of the constant $N'$
 depends on a sufficiently small neighborhood of $\lambda=0$ in which the expansions of all resonance types in Theorem \ref{thm-main-inver-M} hold.

In order to prove Theorem \ref{thm-main results}, it suffices to show the following two theorems.

\begin{theorem}\label{thm-low}(\textbf{Low energy decay estimates})\\
Let $|V(x)|\lesssim (1+|x|)^{-\beta}$ for $x\in \mathbb{R}^2$ and some $\beta>0$. Assume that $H=\Delta^2+V$ has no positive embedded eigenvalues and $P_{ac}(H)$ is the projection onto absolutely continuous spectrum
space of $H$. Then the following statements hold:

(i)~If zero is a regular point and $\beta>10$, then
\begin{equation}\label{thm-low-1}
\|e^{-itH}P_{ac}(H)\chi(H)\|_{L^1\rightarrow L^\infty}\lesssim |t|^{-\frac{1}{2}}.
\end{equation}

(ii)~If zero is a resonance of the first kind and $\beta>14$, then
\begin{equation}\label{thm-low-2}
\| e^{-itH}P_{ac}(H)\chi(H)\|_{L^1\rightarrow L^\infty}\lesssim |t|^{-\frac{1}{2}}.
\end{equation}

(iii)~If zero is a resonance of the second kind and $\beta>18$, then
\begin{equation}\label{thm-low-3}
\| H^{\frac{1}{2}}e^{-itH}P_{ac}(H)\chi(H)\|_{L^1\rightarrow L^\infty}\lesssim |t|^{-\frac{1}{2}}.
\end{equation}

(iv) If zero is a resonance of the third and fourth kinds,  and $\beta>18$, then
\begin{equation}\label{thm-low-4}
\|H^\frac{3}{2}e^{-itH}P_{ac}(H)\chi(H)\|_{L^1\rightarrow L^\infty}\lesssim |t|^{-\frac{1}{2}}.
\end{equation}
\end{theorem}
\begin{theorem}\label{thm-high}(\textbf{ High energy decay estimates})\\
Let $|V(x)|\lesssim (1+ |x|)^{-\beta}$ for $x\in \mathbb{R}^2$ and $\beta>5$. Assume that $H$ has no positive eigenvalues and $P_{ac}(H)$
is the projection onto absolutely continuous spectrum space of $H$. Then
\begin{equation}\label{thm-high-1}
\|e^{-itH}P_{ac}(H)\widetilde{\chi}(\lambda)\|_{L^1\rightarrow L^\infty}\lesssim |t|^{-\frac{1}{2}}.
\end{equation}
\end{theorem}

The proofs of the low and high energy parts will be given in the sequent subsections. To proceed,
 let's first establish the decay estimates of the free group $e^{-it\Delta^2}$ by resolvent
methods, which indicates some key ideas on how to prove the potential cases.

\subsection{ The decay estimates for the free case}
 We first give a lemma which will be used frequently later.
\begin{lemma}\label{lem-LWP}
Let $A$ be some subset of $\mathbb{Z}$. Suppose that $\Psi(s,z)$ is a function on $\mathbb{R}\times\mathbb{R}^m$ which is smooth for the first variable $s$, and satisfies for any $s\in [1/4,1],\ z\in \mathbb{R}^m$,
\begin{equation*}\label{lem-LWP-condition}
   \begin{split}
 | \partial_s^k \Psi(2^Ns,z)| \lesssim 1, \,k=0,1,\, N  \in A  \subset  \mathbb{Z} .
\end{split}
\end{equation*}
 Suppose that $\varphi_0(s)$ is a smoothing function of $\mathbb{R}$ defined in \eqref{unit-deco}, $\Phi(z)$ is a real value function on $\mathbb{R}^m$ and
$ N_0 =\big[\frac{1}{3}\log_2\frac{|\Phi(z)|}{|t|}\big]$. Then for each $z\in \mathbb{R}^m$ and $t\neq0$,
\begin{equation}\label{lem-LWP-con1}
\Big|\int_0^\infty e^{-it2^{4N}s^4}
e^{\pm i2^Ns\Phi(z)} \Psi(2^Ns,z)\varphi_0(s) ds \Big| \le C
\begin{cases}
(1+|t| 2^{4N})^{-\frac{1}{2}},  & \hbox{if}\  |N-N_0|\leq 2,\\
(1+|t|2^{4N})^{-1}, & \hbox{if}\  |N-N_0|> 2,
\end{cases}
\end{equation}
where $C$ is a constant independent of $t$ and $z$.  In particular,  we have
\begin{equation}\label{lem-LWP-con2}
\sup\limits_{z\in \mathbb{R}^m}\Big|\int_0^\infty e^{-it2^{4N}s^4}
e^{\pm i2^Ns\Phi(z)} \Psi(2^Ns,z)\varphi_0(s) ds \Big| \le C' (1+|t| 2^{4N})^{-\frac{1}{2}},\  \ N\in A,
\end{equation}
where $C'$ is a constant independent of $t$.
\end{lemma}
\begin{proof}
 Without loss of generality, we assume that $\Phi>0$. We set $t\geq0$, similar to get the conclusions as desired for $t<0$. For each $N\in A$, we write
 \begin{equation*}
   \begin{split}
K_N^\pm(t,z)=&\int_0^\infty e^{-it2^{4N}s^4}
e^{\pm i2^Ns\Phi(z)} \Psi(2^Ns,z)\varphi_0(s) ds.
\end{split}
\end{equation*}
Now we shall divide into  $2^N \Phi(z) <1$ and  $2^N \Phi(z)\geq 1$ two cases to estimate the integral above.

\textbf{Case 1.} If $2^N \Phi(z)< 1$, then we divide into $|t| 2^{4N} <1$ and $|t| 2^{4N} \geq 1$ two cases again. As for $|t|2^{4N} <1$. Note that $ s \in \hbox{supp}\varphi_0 \subset[\frac{1}{4}, 1]$, one has
\begin{equation}\label{lem-Lp1}
   \begin{split}
|K_N^\pm(t,z)| \lesssim \int_{\frac{1}{4}}^1 \Big|  \Psi(2^Ns,z)\varphi_0(s) \Big|ds
\lesssim (1+|t|\cdot2^{4N})^{-1}.
\end{split}
\end{equation}
As for $|t|2^{4N} \geq 1$. By integration by parts,  one has
$$K_N^\pm(t,z)= \frac{1}{4it2^{4N}}\int_{\frac{1}{4}}^1 e^{-it2^{4N}s^4}
\partial_s\big( s^{-3}\varphi_0(s)e^{\pm i2^Ns\Phi(z)}\Psi(2^Ns,z)\big) ds. $$
Notice that for $k=0,1$,
$$\big|\partial_s^k\big( e^{\pm i2^Ns\Phi(z)}\big)\big|  \lesssim 1,\,\,
  \big|\partial_s^k \Psi(2^Ns,z)\big|\lesssim 1. $$
Thus, we obtain that $|K_N^\pm(t,z)| $ is controlled by $\big(1+|t|2^{4N}\big)^{-1}$ uniformly in $z$.

\textbf{Case 2.} If $2^N \Phi(z)\geq  1$, then we consider $K_N^-(t,z)$ and $K_N^+(t,z)$ case by case.
We begin with estimating $K_N^-(t,z)$. If $|t| 2^{4N} < 1$,  then we can obtain that $|K_N^-(t,z)| $ is bounded by $\big(1+|t|2^{4N}\big)^{-1}$ uniformly in $z$.
If $|t| 2^{4N} \geq 1$, then we write
 \begin{equation*}
   \begin{split}
K_N^-(t,z)=&\int_{\frac{1}{4}}^1e^{-iU_-(t,s,z,N)}V(s,z,N)ds.
\end{split}
\end{equation*}
where $U_-(t,s,z,N)= t2^{4N}s^4+2^Ns\Phi(z)$ and
$V(s,z,N)=\Psi(2^Ns,z)\varphi_0(s)$.
Since
$$ |\partial_sU_-(t,s,z,N)| = | 4t2^{4N}s^3 + 2^N\Phi(z)| \gtrsim  1+|t|2^{4N},$$
so $U_-(t,s,z,N)$ doesn't have critical point. By integration by parts we obtain that $|K_N^-(t,z)|$ is controlled by $\big(1+|t|\cdot2^{4N}\big)^{-1}$ uniformly in $z$.

Next we turn to the estimate of $K_N^+(t,z)$. If $|t| 2^{4N} < 1$,  then it is clear that
$|K_N^+(t,z)|$ is controlled by $\big(1+|t|\cdot2^{4N}\big)^{-1}$ uniformly in $z$.
If $|t| 2^{4N} \geq 1$, set
 $U_+(t,s,z,N)= t2^{4N}s^4-2^Ns\Phi(z)$ and $V(s,z,N)=\Psi(2^Ns,z)\varphi_0(s)$,
then
 \begin{equation*}
   \begin{split}
K_N^+(t,x,y)=&\int_\frac{1}{4}^1 e^{-iU_+(t,s,z,N)}V(s,z,N)ds.
\end{split}
\end{equation*}
Note that $ \partial_sU_+(t,s,z,N)= 4t2^{4N}s^3-2^N\Phi(z) $, hence $U_+(t,s,z,N)$ may have a critical point
$s_0$ located at $[\frac{1}{4}, 1]$ and satisfying  $2^N\Phi(z)=4t2^{4N}s_0^3$.

If $ |N-N_0|<2$, then $ -2+N < \frac{1}{3}\log_2\frac{\Phi(z)}{t}< 3+N$
( i.e. $\frac{1}{2^6}t2^{4N} < 2^N\Phi(z)< 2^9t2^{4N}$), which implies the critical point $s_0$ exists.
Since
$$|\partial_s^2 u_+(t,s,z,N)|  \gtrsim |t|2^{4N}, \ s \in [1/4, 1],$$
hence by Van der Corput lemma (see \cite{Stein}, p.334), we obtain that
$$|K_N^+(t,z)| \lesssim \big( |t|2^{4N}\big)^{-\frac{1}{2}}\Big(V(1,z,N)+ \int_\frac{1}{4}^1
\big|\partial_s V(s,z,N)\big|ds \Big).$$
Note that $V(1,z,N)=0 $ and $\big|\partial_s V(s,z,N)\big|\lesssim 1$, we have
$$|K_N^+(t,x,y)| \lesssim \big( |t|2^{4N}\big)^{-\frac{1}{2}}\lesssim
\big(1+ |t|2^{4N}\big)^{-\frac{1}{2}}.$$

If $|N-N_0|\geq 2 $, then $2^N\Phi(z)\leq \frac{1}{2^6}t2^{4N}$ or $2^N\Phi(z)\geq 2^9 t2^{4N}$, thus the critical point $s_0$ doesn't exist. Using integration by parts again, we obtain that $|K_N^+(t,z)|$ is controlled by $ \big(1+|t|2^{4N}\big)^{-1}$.

Combining \textbf{Case 1} with \textbf{Case 2}, we immediately obtain that (\ref{lem-LWP-con1}) holds.

Finally, it immediately follows from (\ref{lem-LWP-con1}) that (\ref{lem-LWP-con2})
is bounded by $(1+|t| 2^{4N})^{-\frac{1}{2}}$ uniformly in $z$.
\end{proof}

Next we turn to the time decay estimates of $e^{-it\Delta^2}$. From Stone's formula one has
\begin{equation*}\label{Stone-Paley}
   \begin{split}
e^{-it\Delta^2}f=& \frac{2}{ \pi i } \int_0^\infty e^{-it\lambda^4}\lambda^3
[R_0^+(\lambda^4) -R_0^-(\lambda^4)]f d\lambda.
\end{split}
\end{equation*}
We need to split the free resolvent kernel $R_0^\pm(\lambda^4)(x,y)$ into
 low and high parts based on the size of $\lambda|x-y|$ as follows:
 \begin{equation*}\label{label-freeresonance-twoparts}
   \begin{split}
R^\pm_0(\lambda^4)(x,y)= \chi(\lambda|x-y|)R^\pm_0(\lambda^4)(x,y)
          +\widetilde{\chi}(\lambda|x-y|) R^\pm_0(\lambda^4)(x,y).
   \end{split}
\end{equation*}
Set
$$R^\pm_0(\lambda^4)(x,y)=\frac{e^{\pm i\lambda|x-y|}}{\lambda^2}\widetilde{R}^\pm(\lambda|x-y|).$$
Then we have by (\ref{reso-big}) and Lemma \ref{lem-reso},
\begin{equation}\label{freeres-lagrepart-Rtubap}
   \begin{split}
\widetilde{R}^\pm(p)= &e^{\mp ip}\Big[ b_\pm+ a_0 p^2\ln p+ \alpha_\pm p^2
  + c_\pm p^4  + a_2p^6\ln p+ \beta_\pm  p^6 +O_1\big(p^{8-\epsilon} \big)\Big]\chi(p)\\
 &+ \frac{1}{8}\Big( \pm i w_\pm(p)
        -ie^{\mp i p } e^{- p}w_+(  ip)\Big)\widetilde{\chi}(p).
  \end{split}
\end{equation}

\begin{proposition}\label{prop-free estimates}
Let $ N_0=\big[ \frac{1}{3}\log_2\frac{|x-y|}{|t|}  \big]$ and $N \in \mathbb{Z}$. Then for $x\neq y$ and $t\neq0$,
\begin{equation*}
\Big|\int_0^\infty e^{-it\lambda^4}\lambda^3\varphi_0(2^{-N}\lambda) R_0^\pm(\lambda^4)(x,y) d\lambda \Big|
 \lesssim 2^{2N}
\begin{cases}
(1+|t| 2^{4N})^{-\frac{1}{2}},  & \hbox{if}\, |N-N_0|\leq 2\\
(1+|t| 2^{4N})^{-1}, & \hbox{if}\,  |N-N_0|>2.
\end{cases}
\end{equation*}
As a consequence,
\begin{equation}\label{pro-k3}
\sup\limits_{x,y\in \mathbb{R}^2}\Big|\int_0^\infty e^{-it\lambda^4}\lambda^3
R_0^\pm(\lambda^4)(x,y) d\lambda \Big|
 \lesssim|t|^{-\frac{1}{2}},
\end{equation}
which immediately implies that
$\big\|e^{-it\Delta^2}\big\|_{L^1\rightarrow L^\infty} \lesssim |t|^{-\frac{1}{2}}.$
\end{proposition}
\begin{proof}
We write
$$ K_{0,N}^\pm(t,x,y)= \int_0^\infty e^{-i t\lambda^4}
                \lambda^3\varphi_0(2^{-N}\lambda)R^\pm_0(\lambda^4)(x,y)d\lambda.$$
Set $$ R^\pm_0(\lambda^4)(x,y)=\frac{e^{\pm i\lambda|x-y|}}{\lambda^2}\widetilde{R}(\lambda|x-y|).$$
Let $\lambda =2^Ns$, then we have
\begin{equation*}
   \begin{split}
 K_{0,N}^\pm(t,x,y)=&\int_0^\infty e^{-i t\lambda^3}\lambda \varphi_0(2^{-N}\lambda)
  e^{\pm i\lambda|x-y|}\widetilde{R}^\pm(\lambda|x-y|)d\lambda\\
  =& 2^{2N}\int_0^\infty e^{-i t 2^{4N}s^4}e^{\pm i2^N s |x-y|}\ s\,\varphi_0(s)\widetilde{R}^\pm(2^Ns|x-y|)ds.
	\end{split}
\end{equation*}
Note that $ s\in \hbox{supp}\varphi_0 \subset [\frac{1}{4}, 1] $, by (\ref{freeres-lagrepart-Rtubap}) we obtain that for any $x, y$
$$ \big|\partial_s^k \widetilde{R}^\pm(2^Ns|x-y|)\big|\lesssim 1, \ k=0,1. $$
By Lemma \ref{lem-LWP} with $z=(x,y)$, $\Phi(z)=|x-y|$ and
$$\Psi(2^Ns;z)=\widetilde{R}^\pm(2^Ns|x-y|),$$
we have $|K_{0,N}^\pm(t,x,y)|$ is dominated by $2^{2N}(1+|t|2^{4N})^{-1/2}$
when $ |N-N_0|\leq 2$, bounded by  $2^{2N}(1+|t|2^{4N})^{-1}$ when $ |N -N_0|>2$.

Note that $\sum\limits_{N=-\infty}^{\infty}\varphi_0(2^{-N}s)=1, \ s>0$, then we have
\begin{equation*}
   \begin{split}
K_0^\pm(t,x,y):= \int_0^\infty e^{-it\lambda^4}\lambda^3 R_0^\pm(\lambda^4)(x,y) d\lambda
   = \sum_{N=-\infty}^{+\infty}K_{0,N}^\pm(t,x,y).
  \end{split}
\end{equation*}
Since for $t\neq 0$, there exists a constant $N_0' \in \mathbb{Z}$ such that $|t|2^{4N_0'}\sim 1$.
Hence we have for $x\neq y,$
\begin{equation*}
   \begin{split}
 |K_0^\pm(t,x,y)|\leq &\sum_{|N-N_0|\leq 2} 2^{2 N}( 1+|t|2^{4N})^{-\frac{1}{2}}
+\sum_{N=-\infty}^{+\infty} 2^{2 N}( 1+|t|2^{4N})^{-1}\\
\lesssim &\sum_{|N-N_0|\leq 2}|t|^{-\frac{1}{2}}+ \sum_{N=-\infty}^{N_0'}2^{2 N}( 1+|t|2^{4N})^{-1}
+\sum_{N=N_0'+1}^{+\infty}2^{2 N}( 1+|t|2^{4N})^{-1}\\
\lesssim &|t|^{-\frac{1}{2}} +\sum_{N=-\infty}^{N_0'}2^{2N}
+ |t|^{-1}\sum_{N=N_0'+1}^{+\infty}2^{-2N}\\
\lesssim &|t|^{-\frac{1}{2}}.
  \end{split}
\end{equation*}
Thus the proof of this lemma is completed.
\end{proof}

\subsection{Low energy decay estimates  }
In this subsection, we are devoted to proving low energy decay estimates (\ref{thm-low-1})-(\ref{thm-low-4})
case by case. The following lemma plays an important role in making use of cancellations
 of projection operators $S_j$ in the asymptotic expansions of resolvent $R_V^\pm(z)$ near zero.
\begin{lemma}\label{Taylor-low}
Assume that $\lambda>0$, $F(p)\in C^\infty(\mathbb{R})$,  $\bar{y}=(-y_2,y_1) \in \mathbb{R}^2$, \ $\displaystyle w \equiv w(x)=\frac{x}{|x|}$ for $x\neq 0$ and $w(x)=0$ for $x=0$.
Let $\theta\in [0,1]$ and $\displaystyle |y| \cos\alpha= \langle y, w(x-\theta y)\rangle$, where $\alpha\equiv \alpha(x, y, \theta )$ is the angle between the vectors $y$ and $x-\theta y $. Then

$(i)$ We have
\begin{equation}\label{lem-Taylor1}
   \begin{split}
F(\lambda|x-y|)= F(\lambda|x|)-\lambda |y|\int_0^1F'(\lambda|x-\theta y|)\cos\alpha d\theta.
\end{split}
\end{equation}

$(ii)$ If $F'(0)=0$, then
\begin{align*}
F(\lambda|x-y|)
=&F(\lambda|x|)-\lambda  F'(\lambda|x|)\big\langle w(x), y \big\rangle\\
&+\lambda^2|y|^2\int_0^1(1-\theta)\Big(
F''(\lambda|x-\theta y|) \cos^2\alpha
+\frac{F'(\lambda|x-\theta y|)}{\lambda|x-\theta y|} \sin^2\alpha \Big) d\theta.
\end{align*}

$(iii)$ If $ F'(0)=F''(0)=0$, then
\begin{equation*}
   \begin{split}
 F(\lambda|x-y|)=
 & F(\lambda|x|) -\lambda F'(\lambda|x|)\big\langle w(x) , y \big\rangle\\
&+\frac{\lambda^2}{2} \Big(\frac{F'(\lambda|x|)}{\lambda|x|}
\big\langle w(x), \bar{y}\big\rangle^2 +
F''(\lambda|x|)\big\langle w(x), y \big\rangle^2\Big)
+\frac{\lambda^3|y|^3}{2}\int_0^1(1-\theta)^2\\
&\Big[ \Big( \frac{ F'(\lambda|x-\theta y|)}{\lambda^2|x-\theta y|^2}
- \frac{ F''(\lambda|x-\theta y|)}{\lambda|x-\theta y|} \Big)
3\cos\alpha \sin^2\alpha
-F'''(\lambda|x-\theta y|)\cos^3\alpha \Big]d\theta.
\end{split}
\end{equation*}
Here $ \langle w(x),\bar{y} \rangle^2=|y|^2- \langle w(x), y\rangle^2$.

$(iv)$ If $F^{(k)}(0)=0, k=1,2,3$, then
\begin{equation*}
   \begin{split}
F(\lambda||x-y|)=&F(\lambda |x|)-\lambda F'(\lambda|x|)\langle w(x), y\rangle
+\frac{\lambda^2}{2}\Big( \frac{F'(\lambda|x|)}{\lambda|x|}\langle w(x), \bar{y}\rangle^2
+F''(\lambda|x|)\langle w(x),y\rangle^2   \Big)\\
&+\frac{\lambda^3 }{3!}\Big[ \Big( \frac{F'(\lambda|x|)}{\lambda^2|x|^2}
-\frac{F''(\lambda|x|)}{\lambda|x|}\Big)3\langle w(x),y\rangle  \langle w(x),\bar{y}\rangle^2     - F'''(\lambda|x|)\langle w(x),y\rangle^3\Big]\\
&+\frac{\lambda^4|y|^4}{3!}\int_0^1(1-\theta)^3\Big[\Big(\frac{F'(\lambda|x-\theta y|)}
{\lambda^3|x-\theta y|^3} - \frac{F''(\lambda|x-\theta y|)}{\lambda^2|x-\theta y|^2}\Big)
(15\cos^2\alpha \sin^2\alpha\\
& -3\sin^2\alpha)
+ \frac{F'''(\lambda|x-\theta y|)}{\lambda|x-\theta y|} 6\cos^2\alpha\sin^2\alpha
+F^{(4)}(\lambda|x-\theta y|) \cos^4\alpha \Big]d\theta.
\end{split}
\end{equation*}
\end{lemma}
\begin{proof}
Let $G_\varepsilon(y)= F(\lambda\sqrt{\varepsilon^2+|x-y|^2}),\, \varepsilon \neq 0$.
Then $G_\varepsilon(y)\in C^{k+1}(\mathbb{R}^2)$  for $\varepsilon \neq 0$ and
$F(\lambda|x-y|)=\lim_{\varepsilon\rightarrow0}G_\varepsilon(y)$.
By Taylor's expansions, we have
\begin{equation}\label{lem-Taylorformula}
   \begin{split}
G_\varepsilon(y)=\sum_{|\alpha|<k}\frac{\partial^\alpha G_\varepsilon(0)}{\alpha !}y^\alpha
+k\int_0^1 (1-\theta)^{k-1}\sum_{|\alpha|=k}\frac{(\partial^\alpha
G_\varepsilon )(\theta y)y^\alpha}{\alpha !}d\theta.
\end{split}
\end{equation}

$(i)$ Observe that
$$\partial_{y_j}G_\varepsilon(y)
=\frac{-\lambda(x_j-y_j)}{(\varepsilon^2+|x-y|^2)^{\frac{1}{2}}}G'(\lambda\sqrt{\varepsilon^2+|x-y|^2}),
\ j=1,2.$$
Since there exists a constants $C=C(\lambda,x,y)$ such that $|(\partial_{y_i}G_\varepsilon)(\theta y)|
 \leq C (i=1,2)$ for $0\leq \theta\leq 1$ and $0<\varepsilon \leq 1$,
then by Lebesgue's dominated convergence theorem, we have for $x-\theta y \neq 0$,
\begin{equation*}
   \begin{split}
\lim_{\varepsilon\rightarrow 0}\int_0^1(\partial_{y_i}G_\varepsilon)(\theta y)d\theta
=& \int_0^1 \frac{-\lambda(x_j-\theta y_j)}{|x-\theta y|}F'(\lambda|x-\theta y|)d\theta, \ j=1,2,
\end{split}
\end{equation*}
and $\lim_{\varepsilon\rightarrow 0}\int_0^1(\partial_{y_i}G_\varepsilon)(\theta y)d\theta=0, j=1,2 $ for
$x-\theta y=0 $.
From Taylor expansions (\ref{lem-Taylorformula}) with $k=1$, we obtain that
\begin{equation*}
   \begin{split}
F(\lambda|x-y|)= F(\lambda|x|)-\lambda |y|\int_0^1F'(\lambda|x-\theta y|)\cos\alpha d\theta.
\end{split}
\end{equation*}

$(ii)$ Note that
\begin{equation*}
   \begin{split}
\partial_{y_j^2}^2G_\varepsilon(y)=& \frac{\lambda^2\varepsilon^2
+\lambda^2\big(|x-y|^2-(x_j-y_j)^2\big)}{(\varepsilon^2+|x-y|^2)}
\frac{F'(\lambda\sqrt{\varepsilon^2+|x-y|^2})}{\lambda(\varepsilon^2+|x-y|^2)^{\frac{1}{2}}}\\
&+\frac{\lambda^2(x_j-y_j)^2}{\varepsilon^2+|x-y|^2}F''(\lambda\sqrt{\varepsilon^2+|x-y|^2}), j=1,2,
\end{split}
\end{equation*}
\begin{equation*}
   \begin{split}
\partial_{y_1y_2}^2G_\varepsilon(y)
=&\frac{\lambda^2(x_1-y_1)(x_2-y_2)}{\varepsilon^2+|x-y|^2}
\Big( F''(\lambda\sqrt{\varepsilon^2+|x-y|^2})
-\frac{F'(\lambda\sqrt{\varepsilon^2+|x-y|^2})}{\lambda(\varepsilon^2+|x-y|^2)^\frac{1}{2}}  \Big).
\end{split}
\end{equation*}
Hence, we have for $x\neq \theta y$,
$$\lim_{\varepsilon\rightarrow0}\partial_{y_j^2}^2G_\varepsilon(\theta y)=
\lambda^2\Big(1- \frac{(x_j-\theta y_j)^2}{|x-\theta y|^2}\Big)\frac{F'(\lambda|x-\theta y|)}{\lambda|x-\theta y|}
 +\frac{\lambda^2(x_j-\theta y_j)^2}{|x-\theta y|^2}F''(\lambda|x-\theta y|), \ j=1,2.$$
Note that $F'(0)=0$, by Hospital's formula we have for $x=\theta y$,
$$\lim_{\varepsilon\rightarrow0}(\partial_{y_j^2}^2G_\varepsilon)(\theta y)=
 \lambda^2 \lim_{\varepsilon\rightarrow0}\frac{F'(\lambda\varepsilon)}{\lambda\varepsilon}
  = \lambda^2 \lim_{\varepsilon\rightarrow0}F''(\lambda \varepsilon ) = \lambda^2F''(0), \ j=1,2.$$
Similarly, we obtain that
$$
\lim_{\varepsilon\rightarrow0}(\partial_{y_1y_2}^2G_\varepsilon)(\theta y)=
\lambda^2w(x_1-\theta y_1)w(x_2-\theta y_2)\Big(F''(\lambda|x-\theta y|)
-\frac{F'(\lambda|x-\theta y|)}{\lambda|x-\theta y|}  \Big).
$$
It is easy to check that $$ \lim_{\varepsilon\rightarrow0} \partial_{y_j}G_\varepsilon(0)
=-\lambda w(x_j)F'(\lambda|x|),\ j=1,2.$$
Note that $ \sin^2\alpha = 1-|\langle w(y), w(x-\theta y)\rangle|^2 = |\langle w(\bar{y}), w(x-\theta y))\rangle |^2$, by Taylor's expansions and Lebesgue's dominated convergence  theorem,
we obtain that
\begin{align*}
F(\lambda|x-y|)
=&F(\lambda|x|)-\lambda F'(\lambda|x|)\langle w(x), y \rangle\\
&+\lambda^2|y|^2\int_0^1(1-\theta)\Big(\frac{F'(\lambda|x-\theta y|)}{\lambda|x-\theta y|} \sin^2 \alpha+
 F''(\lambda|x-\theta y|) \cos^2\alpha\Big) d\theta.
\end{align*}

Similarly, we can prove that $(iii)$ and $(iv)$ hold, these details are left to the interested readers.
\end{proof}
\begin{remark}\label{remark-beha-Rpm}
To use Taylor's expansion above, we will use the following notations and facts related to $R^\pm_0(\lambda)$ later.

$(i)$ Set
 $R^\pm_0(\lambda^4)(x,y)=\frac{1}{\lambda^2}F^\pm(\lambda|x-y|).$
Then by (\ref{reso-big}) and Lemma \ref{lem-reso} we have
\begin{equation}\label{freeres-lagrepart-Fp}
   \begin{split}
F^\pm(p)=& \big(b_\pm+ a_0 p^2\ln p+ \alpha_\pm p^2
  + c_\pm p^4  + a_2p^6\ln p+ \beta_\pm  p^6 +O_1(p^{8-\epsilon})\big)\chi(p)\\
  &   +\big(\pm \frac{i}{8} e^{\pm ip} w_\pm(p)
        -\frac{i}{8} e^{- p}w_+(  ip)\big)\widetilde{\chi}(p).
  \end{split}
\end{equation}

$(ii)$ By (\ref{freeres-lagrepart-Fp}) it follows that
 $$\big[R_0^+(\lambda^4)-R^-_0(\lambda^4)\big](x,y)
 = \frac{1}{\lambda^2}(F^+(\lambda|x-y|)- F^-(\lambda|x-y|))
 := \frac{1}{\lambda^2}\bar{F}(\lambda|x-y|),$$
 where
\begin{equation}\label{reso-R+RF-big}
   \begin{split}
\bar{F}(p)=&\frac{i}{4}\big(1-\frac{1}{4}p^2+\frac{1}{64}p^4-\frac{1}{2304}p^6+O(p^8)\big)\chi(p)
+\frac{i}{8}\big( e^{ip}w_+(p)+ e^{-ip}w_-(p) \big)\widetilde{\chi}(p).
\end{split}
\end{equation}
Moreover, $\bar{F}(p)\in C^\infty(\mathbb{R})$, $\bar{F}'(0)=0$ and $ \bar{F}''(0)\neq 0$.

$(iii)$ Set $ \widetilde{F}(p)= \bar{F}(p)-\frac{i}{16}p^2$. Then we have
\begin{equation}\label{reso-R+R-Ftuba-big}
   \begin{split}
\widetilde{F}(p)=&\frac{i}{4}\big(1+\frac{1}{64}p^4-\frac{1}{234}p^6+O(p^8)\big)\chi(p)
+\frac{i}{8}\big( e^{ip}w_+(p)+ e^{-ip}w_-(p) -\frac{p^2}{2} \big)\widetilde{\chi}(p),
\end{split}
\end{equation}
which satisfies that $\widetilde{F}(p)\in C^\infty(\mathbb{R})$, $\widetilde{F}^{(k)}(0)=0(k=1,2,3)$,
and $\widetilde{F}^{(4)}(0)\neq 0$.
\end{remark}

\subsubsection{\textbf{Regular case} }
Let's begin with showing  Theorem \ref{thm-low} in regular case. Recall that
Stone's formula
\begin{equation}\label{stone-formula-low}
   \begin{split}
H^\alpha e^{-itH}P_{ac}(H)\chi(H)f=& \frac{2}{\pi i}\int_0^\infty e^{-it\lambda^4}\chi(\lambda)\lambda^{3+4\alpha}
[R_V^+(\lambda^4)- R_V^-(\lambda^4)]fd\lambda\\
=&\sum_{N=-\infty}^{N'}\sum_{\pm}\frac{\pm 2}{\pi i}
\int_0^\infty e^{-it\lambda^4}\varphi_0(2^{-N}\lambda)\lambda^{3+4\alpha}R^\pm_V(\lambda^4)fd\lambda.
\end{split}
\end{equation}
If zero is a regular point, then using (\ref{id-RV}) and (\ref{thm-regularinver-M0 })
we have
\begin{equation}\label{RV-regular}
   \begin{split}
R_V^\pm(\lambda^4)
= &R^\pm_0(\lambda^4)- R^\pm_0(\lambda^4)v\Big(S_0D_1S_0 \Big)vR^\pm_0(\lambda^4)
-R^\pm_0(\lambda^4)v\Big(g_0^\pm(\lambda)^{-1}Q\Gamma^0_{0,1}Q \Big)\\
&\times vR^\pm_0(\lambda^4)
-R^\pm_0(\lambda^4)v\Big(g_0^\pm(\lambda)^2\lambda^2S_0\Gamma^0_{2,1}S_0 \Big) vR^\pm_0(\lambda^4)
-R^\pm_0(\lambda^4)v\\
&\times \Big(g_0^\pm(\lambda)\lambda^2\big( S_0\Gamma^0_{2,2}Q
+Q\Gamma^0_{2,3}S_0\big) \Big) vR^\pm_0(\lambda^4)
-R^\pm_0(\lambda^4)\Big(vO_1(\lambda^2) v\Big) R^\pm_0(\lambda^4).
  \end{split}
\end{equation}
Hence, to prove Theorem \ref{thm-low} in regular case, combining with the Proposition \ref{prop-free estimates}, by (\ref{stone-formula-low}) and
(\ref{RV-regular}) it suffices to prove the following
Propositions \ref{prop-S0D1S0}--\ref{prop-reg-freeterms}.
\begin{proposition}\label{prop-S0D1S0}
Assume that $|V(x)|\lesssim (1+|x|)^{-\beta}$ with $\beta > 10$. If $N\in\mathbb{Z}$ and $N\leq N'$,
then
\begin{equation*}
   \begin{split}
\sup\limits_{x,y\in \mathbb{R}^2}\Big|\int_0^\infty e^{-it\lambda^4}\lambda^{3}\varphi_0(2^{-N}\lambda)\big[R_0^\pm(\lambda^4)v(S_0D_1S_0)vR_0^\pm(\lambda^4)\big](x,y)
 d\lambda\Big|\lesssim 2^{2N}(1+|t| 2^{4N})^{-1/2}.
\end{split}
\end{equation*}
\end{proposition}
\begin{proof}
We write
$$K_{1,N}^{0,\pm}(t;x,y)= \int_0^\infty e^{-it\lambda^4}\lambda^3\varphi_0(2^{-N}\lambda)\big[R_0^\pm(\lambda^4)vS_0D_1S_0vR_0^\pm(\lambda^4)\big](x,y)
 d\lambda.$$
Let $\displaystyle R_0^\pm(\lambda^4)(x,y)=\frac{1}{\lambda^2}F^\pm(\lambda|x-y|)$.
 Using the orthogonality $S_0v=0$ and Lemma \ref{Taylor-low}(i) we have
\begin{equation*}
   \begin{split}
&[R_0^\pm(\lambda^4) vD_1S_0v R_0^\pm(\lambda^4)](x,y)\\
=& \frac{1}{\lambda^4}\int_{\mathbb{R}^4}F^\pm(\lambda|x-u_2|)
     [vS_0D_1S_0v](u_2,u_1)F^\pm(\lambda |y-u_1|)du_1du_2\\
=&\frac{1}{\lambda^2}\int_{\mathbb{R}^4}\int_0^1\int_0^1(F^\pm)'(\lambda|x-\theta_2u_2|)
(F^\pm)'(\lambda|y-\theta_1u_1|)\cos\alpha_2\cos\alpha_1d\theta_1d\theta_2 \\
&|u_1||u_2|[vS_0D_1S_0v](u_2,u_1)du_1du_2,
  \end{split}
\end{equation*}
where $\cos\alpha_1= \cos\alpha(y,u_1,\theta_1) $ and $\cos\alpha_2= \cos\alpha(x,u_2,\theta_2)$.

Thus, we have
\begin{equation*}
   \begin{split}
K_{1,N}^{0,\pm}(t;x,y)=&\int_{\mathbb{R}^4}\Big(\int_0^1\int_0^1\Big(\int_0^\infty e^{-it\lambda^4}\lambda
\varphi_0(2^{-N}\lambda)(F^\pm)'(\lambda|x-\theta_2u_2|)(F^\pm)'(\lambda|y-\theta_1u_1|)d\lambda\Big)\\
&\cos\alpha_2\cos\alpha_1d\theta_1d\theta_2\Big)
\, |u_1|\, |u_2|\, [vS_0D_1S_0v](u_2,u_1)\,du_1du_2.
 \end{split}
\end{equation*}
Let
\begin{equation*}
   \begin{split}
E^{0,\pm}_{1,N}(t;x,y,\theta_1,\theta_2,u_1,u_2)= \int_0^\infty e^{-it\lambda^4}\lambda
\varphi_0(2^{-N}\lambda)(F^\pm)'(\lambda|x-\theta_2u_2|)(F^\pm)'(\lambda|y-\theta_1u_1|)d\lambda.
  \end{split}
\end{equation*}
Then
\begin{equation}\label{prop-esti-k1N}
   \begin{split}
|K_{1,N}^{0,\pm}(t;x,y)|\lesssim & \int_{\mathbb{R}^4}\Big(\int_0^1\int_0^1|E^{0,\pm}_{1,N}(t;x,y,\theta_1,\theta_2,u_1,u_2)|d\theta_1d\theta_2\Big)\times\\
&\ \ \ \ \ |u_1||u_2|\,|[vS_0D_1S_0v](u_2,u_1)|du_1du_2.
  \end{split}
\end{equation}
Now we begin to estimate $E^{0,\pm}_{1,N}(t;x,y,\theta_1,\theta_2,u_1,u_2)$. In fact, set
\begin{equation}\label{freereso-Fp-daoshu}
   \begin{split}
(F^\pm)'(p)=e^{\pm ip}F^\pm_1(p),\ p\geq 0.
  \end{split}
\end{equation}
Then from (\ref{freeres-lagrepart-Fp}) we have
\begin{equation}\label{freereso-Fp-daoshu-big}
   \begin{split}
F^\pm_1(p)=\frac{1}{8}\Big(-w_\pm(p)\pm i w_\pm'(p)
           +ie^{\mp ip}e^{-p}\big(w_+(ip)-w'_+(ip)\big) \Big), \  p \gg 1,
  \end{split}
\end{equation}
\begin{equation}\label{freereso-Fp-daoshu-small}
   \begin{split}
F_1^\pm(p)=&  e^{\mp ip}\Big(2a_0p\ln p+(a_0+2\alpha_\pm)p +4c_\pm p^3
+6a_2p^5\ln p+(a_2+6\beta_\pm)p^5 \\
&+O(p^{7-\epsilon}) \Big),\  p\ll1.
  \end{split}
\end{equation}
Let $\lambda=2^Ns$, then we have
\begin{equation*}
   \begin{split}
&E^{0,\pm}_{1,N}(t;x,y,\theta_1,\theta_2,u_1,u_2)\\
=& \int_0^\infty e^{-it\lambda^4}\lambda
\varphi_0(2^{-N}\lambda)e^{\pm i\lambda|x-\theta_2u_2|}e^{\pm i\lambda|y-\theta_1u_1|}
F^\pm_1(\lambda|x-\theta_2u_2|)F^\pm_1(\lambda|y-\theta_1u_1|)d\lambda\\
=& 2^{2N}\int_0^\infty e^{-it2^{4N}s^4}s\varphi_0(s)
e^{\pm i2^Ns(|x-\theta_2u_2| +|y-\theta_1u_1| )}
F^\pm_1(2^Ns|x-\theta_2u_2|)F^\pm_1(2^Ns|y-\theta_1u_1|)ds.
  \end{split}
\end{equation*}
Note that $s\in[1/4,1]$ and $N\in\mathbb{Z}$, by (\ref{freereso-Fp-daoshu-big})
and (\ref{freereso-Fp-daoshu-small})
it is easy to check that for $k=0,1$,
$$ |\partial_s^kF^\pm_1(2^Ns|x-\theta_2u_2|)| \lesssim 1\  \hbox {and}\
|\partial_s^kF^\pm_1(2^Ns|y-\theta_1u_1|)| \lesssim 1. $$
Let $z= (x,y,\theta_1,\theta_1,u_1,u_2)\in \mathbb{R}^{10}$, $\Phi(z)=|x-\theta_2u_2| +|y-\theta_1u_1| $ and
 $$ \Psi(2^Ns, z)=F^\pm_1(2^Ns|x-\theta_2u_2|)F^\pm_1(2^Ns|y-\theta_1u_1|).$$
By Lemma \ref{lem-LWP}, we immediately obtain that for any $z$,
$$|E^{0,\pm}_{1,N}(t;x,y,\theta_1,\theta_2,u_1,u_2)|\lesssim 2^{2N}(1+|t|2^{4N})^{-1/2}.$$
Finally, from (\ref{prop-esti-k1N})and  H\"{o}lder inequality  we obtain that for any $x,y$,
\begin{equation*}
   \begin{split}
|K_{1,N}^{0,\pm}(t;x,y)|
\lesssim & 2^{2N}(1+|t|2^{4N})^{-1/2}\big\|u_1v(u_1)\big\|_{L^2_{u_1}} \big\|u_2v(u_2)\big\|_{L^2_{u_2}}
\big\|S_0D_1S_0\big\|_{L^2\rightarrow L^2}\\
\lesssim & 2^{2N}(1+|t|2^{4N})^{-1/2}.
  \end{split}
\end{equation*}
Thus the proof of this proposition is completed.
\end{proof}
\label{prop-QD1Q}
\begin{proposition}\label{prop-QlambdaQ}
Assume that $|V(x)|\lesssim (1+|x|)^{-\beta}$ with $\beta >10$. If $N\in \mathbb{Z}$ and
 $N \leq N'$, then
\begin{equation*}
\sup\limits_{x,y}\Big|\int_0^\infty e^{-it\lambda^4}\lambda^3\varphi_0(2^{-N}\lambda)\big[R_0^\pm(\lambda^4)v
\big(g_0^\pm(\lambda)^{-1}Q\Gamma_{0,1}^0(\lambda) Q\big) vR_0^\pm(\lambda^4)\big](x,y)
 d\lambda\Big|\lesssim \frac{2^{2N}}{(1+|t|2^{4N})^{\frac{1}{2}}}.
\end{equation*}
\end{proposition}
\begin{proof}
Denote
\begin{equation*}
   \begin{split}
K_{2,N}^{0,\pm}(t;x,y)=\int_0^\infty e^{-it\lambda^4}\lambda^3\varphi_0(2^{-N}\lambda)\big[R_0^\pm(\lambda^4)v
\big(g_0^\pm(\lambda)^{-1}Q\Gamma_{0,1}^0(\lambda) Q\big) vR_0^\pm(\lambda^4)\big](x,y)
 d\lambda.
  \end{split}
\end{equation*}
Let
$$R_0^\pm(\lambda^2)(x,y)=\frac{1}{\lambda^2}F^\pm(\lambda|x-y|).$$
Using the orthogonality $Qv=0$, by Lemma \ref{Taylor-low}(i) we have
\begin{equation*}
   \begin{split}
&\big[R_0^\pm(\lambda^4) v\big(g_0^\pm(\lambda)^{-1}Q\Gamma_{0,1}^0(\lambda) Q\big)v R_0^\pm(\lambda^4)\big](x,y)\\
=& \frac{g_0^\pm(\lambda)^{-1}}{\lambda^4}\int_{\mathbb{R}^4}F^\pm(\lambda|x-u_2|)
     [vQ\Gamma_{0,1}^0(\lambda) Qv](u_2,u_1)F^\pm(\lambda |y-u_1|)du_1du_2\\
=&\frac{g_0^\pm(\lambda)^{-1}}{\lambda^2}\int_{\mathbb{R}^4}\int_0^1\int_0^1(F^\pm)'(\lambda|x-\theta_2u_2|)
(F^\pm)'(\lambda|y-\theta_1u_1|)\cos\alpha_2\cos\alpha_1d\theta_1d\theta_2 \\
&|u_1||u_2|[vQ\Gamma_{0,1}^0(\lambda) Qv](u_2,u_1)du_1du_2.
  \end{split}
\end{equation*}
Thus, by Fubini's theorem we have
\begin{equation*}
   \begin{split}
&K_{2,N}^{0,\pm}(t;x,y)\\
=&\int_0^1\int_0^1\Big(\int_0^\infty e^{-it\lambda^4}\lambda
\varphi_0(2^{-N}\lambda) g^\pm_0(\lambda)^{-1}\int_{\mathbb{R}^4}(F^\pm)'(\lambda|x-\theta_2u_2|)
(F^\pm)'(\lambda|y-\theta_1u_1|)\\
&|u_1||u_2|v(u_1)v(u_2)[Q\Gamma_{0,1}^0(\lambda) Q](u_2,u_1)du_1du_2
d\lambda\Big)
\cos\alpha_2\cos\alpha_1d\theta_1d\theta_2.
 \end{split}
\end{equation*}
Let
\begin{equation*}
   \begin{split}
E^{0,\pm}_{2,N}(t;x,y,\theta_1,\theta_2)=& \int_0^\infty e^{-it\lambda^4}\lambda
\varphi_0(2^{-N}\lambda) g^\pm_0(\lambda)^{-1}\int_{\mathbb{R}^4}(F^\pm)'(\lambda|x-\theta_2u_2|)\\
&(F^\pm)'(\lambda|y-\theta_1u_1|)
|u_1||u_2|v(u_1)v(u_2)[Q\Gamma_{0,1}^0(\lambda) Q](u_2,u_1)du_1du_2
d\lambda.
  \end{split}
\end{equation*}
Thus, we have
\begin{equation}\label{esti-K2N0pm}
   \begin{split}
|K_{2,N}^{0,\pm}(t;x,y)|\lesssim \int_0^1\int_0^1|E^{0,\pm}_{2,N}(t;x,y,\theta_1,\theta_2)|d\theta_1d\theta_2
  \end{split}
\end{equation}
If for any $x,y\in \mathbb{R}^2$, and $0 \leq \theta_i \leq 1, i=1,2$, then
\begin{equation}\label{estu-E0pm2N}
   \begin{split}
|E^{0,\pm}_{2,N}(t;x,y,\theta_1,\theta_2)| \lesssim 2^{(2+\alpha)N}(1+|t|2^{4N})^{-1/2}.
  \end{split}
\end{equation}
By (\ref{esti-K2N0pm} ) we immediately obtain that $|K_{2,N}^{0,\pm}(t;x,y)|$
is bounded by $2^{2N}(1+|t|2^{4N})^{-1/2}$ uniformly in  $x,y$.

Finally, it remains the proof of (\ref{estu-E0pm2N}). Indeed, let
$$ (F^\pm)'(p)=e^{\pm ip}F_1^\pm(p),\ p\geq 0. $$
Set $\lambda =2^Ns$, we have
\begin{equation*}
   \begin{split}
&E^{0,\pm}_{2,N}(t;x,y,\theta_1,\theta_2)\\
=&2^{2N} \int_0^\infty e^{-it2^{2N}s^4} e^{\pm i2^Ns(|x|+|y|)}s\varphi_0(s)
\Big( g^\pm_0(2^Ns)^{-1}\int_{\mathbb{R}^4}e^{\pm i2^Ns(|x-\theta_2u_2|-|x|)}
e^{\pm i2^Ns(|y-\theta_1u_1|-|y|)}\\
&F_1^\pm(2^Ns|x-\theta_2u_2|)F_1^\pm(2^Ns|y-\theta_1u_1|)
|u_1||u_2|v(u_1)v(u_2)[Q\Gamma_{0,1}^0(2^Ns) Q](u_2,u_1)du_1du_2\Big) ds.
  \end{split}
\end{equation*}
Let $z=(x,y, \theta_1, \theta_1)$, $\Phi(z)=|x|+|y|$ and
\begin{equation*}
   \begin{split}
\Psi(2^Ns,z)=&g^\pm_0(2^Ns)^{-1}\int_{\mathbb{R}^4}e^{\pm i2^Ns(|x-\theta_2u_2|-|x|)}
e^{\pm i2^Ns(|y-\theta_1u_1|-|y|)}F_1^\pm(2^Ns|x-\theta_2u_2|)\\
&F_1^\pm(2^Ns|y-\theta_1u_1|)
|u_1||u_2|v(u_1)v(u_2)[Q\Gamma_{0,1}^0(2^Ns) Q](u_2,u_1)du_1du_2.
  \end{split}
\end{equation*}
By (\ref{freereso-Fp-daoshu-big})
and (\ref{freereso-Fp-daoshu-small})
it is easy to check that for $k=0,1$,
$$ |\partial_s^kF^\pm_1(2^Ns|x-\theta_2u_2|)| \lesssim 1,\ \
|\partial_s^kF^\pm_1(2^Ns|y-\theta_1u_1|)| \lesssim 1, \ |\partial_s^k{(g^\pm_0(2^Ns)^{-1})}| \lesssim 1.$$
$$\big|\partial_s^k e^{\pm i2^Ns(|x-\theta_2u_2|-|x|)}\big| \lesssim (2^N|u_2|)^k,\
 \  \big|\partial_s^k e^{\pm i2^Ns(|y-\theta_1u_1|-|y|)}\big| \lesssim (2^N|u_1|)^k. $$
Note that $\| \Gamma_{0,2}^0(\lambda)\|_{L^2\rightarrow L^2}= O_1(1)$,
by H\"{o}lder's inequality we have
\begin{equation*}
   \begin{split}
|\Psi(2^Ns,z)| \lesssim & \| u_1v(u_1)\|_{L^2_{u_1}}\| u_2v(u_2)\|_{L^2_{u_2}}.
  \end{split}
\end{equation*}
Observe that $2^N\cdot (2^Ns)^{-1} \lesssim 1$ by $ s\in \hbox{supp}\varphi_0\subset [\frac{1}{4},1]$, then
$$2^N \| Q\partial_s \Gamma_{0,1}^0(2^Ns)Q\|_{L^2\rightarrow L^2}
\lesssim 2^N \| \partial_s \Gamma_{0,1}^0(2^Ns)\|_{L^2\rightarrow L^2}
\lesssim 2^N \cdot (2^Ns)^{-1}\lesssim 1,  $$
 by H\"{o}lder's inequality, we have
\begin{equation*}
   \begin{split}
|\partial_s\Psi(2^Ns,z)
\lesssim & (\| u_1v(u_1)\|_{L^2_{u_1}}\| u_2v(u_2)\|_{L^2_{u_2}}
 +\| |u_1|^2v(u_1)\|_{L^2_{u_1}}\| u_2v(u_2)\|_{L^2_{u_2}} \\
  &+\| u_1v(u_1)\|_{L^2_{u_1}}\| |u_2|^2v(u_2)\|_{L^2_{u_2}}\lesssim 1.
  \end{split}
\end{equation*}
By Lemma \ref{lem-LWP}, we  obtain that the inequality (\ref{estu-E0pm2N}) holds.
\end{proof}

Next, since we use the same orthogonal relation $Qv=0$ as in the proof of Proposition \ref{prop-QlambdaQ}, so we immediately obtain the following proposition.
\begin{proposition}\label{prop-reg-otherterms}
Assume that $|V(x)|\lesssim (1+|x|)^{-\beta}$ with $\beta >10$. Let $\Omega^\pm(\lambda)$ be one of
the three operators:
$g_0^\pm(\lambda)^2\lambda^2S_0\Gamma^0_{2,1}(\lambda)S_0$, $g_0^\pm(\lambda)\lambda^2S_0\Gamma^0_{2,2}(\lambda)Q$
 and  $g_0^\pm(\lambda)\lambda^2Q\Gamma^0_{2,3}(\lambda)S_0$.
If $N\in \mathbb{Z}$ and $N \leq N'$, then
\begin{equation*}
\sup\limits_{x,y}\Big|\int_0^\infty e^{-it\lambda^4}\lambda^3\varphi_0(2^{-N}\lambda)\big[R_0^\pm(\lambda^4)v\Omega^\pm(\lambda) vR_0^\pm(\lambda^4)\big](x,y)d\lambda\Big|\lesssim 2^{2N}(1+|t|2^{4N})^{-1/2}.
\end{equation*}
\end{proposition}
\begin{proposition}\label{prop-reg-freeterms} Suppose that $|V(x)|\lesssim (1+|x|)^{-\beta}$ with $\beta >10$ and  $\Gamma^0(\lambda)=O_1(\lambda^2)$ be the error term of the regular expansions of $(M^\pm(\lambda))^{-1}$. If $N\in \mathbb{Z}$ and
$N \leq N'$, then
\begin{equation*}
\sup\limits_{x,y}\Big|\int_0^\infty e^{-it\lambda^4}\lambda^3\varphi_0(2^{-N}\lambda)\big[R_0^\pm(\lambda^4)v\Gamma^0(\lambda) vR_0^\pm(\lambda^4)\big](x,y)d\lambda\Big|\lesssim 2^{2N}(1+|t|2^{4N})^{-1/2}.
\end{equation*}
\end{proposition}
\begin{proof}
According to Riesz's theorem, it suffices to prove that for any $f,g\in L^1$,
\begin{equation*}
   \begin{split}
\Big|\int_0^\infty e^{-it\lambda^4}\lambda^3\varphi_0(2^{-N}\lambda)
\big\langle [v\Gamma^0(\lambda) v] (R_0^\pm(\lambda^4)f), ~\big(R_0^\pm(\lambda^4)\big)^*g  \big\rangle
d\lambda\Big|
  \end{split}
\end{equation*}
is bounded by $2^{2N}(1+|t|2^{4N})^{-1/2}\|f\|_{L^1}\|g\|_{L^1}$.
Hence it is enough to show that the following kernel
\begin{equation*}
   \begin{split}
K^{0,\pm}_{3,N}(t;x,y):= \int_0^\infty e^{-it\lambda^4}\lambda^3\varphi_0(2^{-N}\lambda)
\big\langle [v\Gamma^0(\lambda) v]\big(R_0^\pm(\lambda^4)(*, y)\big)(\cdot), ~(R_0^\pm)^*(\lambda^4)(x,\cdot)   \big\rangle d\lambda
  \end{split}
\end{equation*}
is bounded uniformly  in $x,y$ by $2^{2N}(1+|t|2^{4N})^{-1/2}$.

 Set $R_0^\pm(\lambda^4)(x,y)= \frac{ e^{\pm i\lambda|x-y|}}{\lambda^2} \widetilde{R}^\pm(\lambda|x-y|)$ where the kernel $\widetilde{R}^\pm(\lambda|x-y|)$ are defined in \eqref{freeres-lagrepart-Rtubap}.
Then we have
\begin{equation*}
   \begin{split}
&\Big\langle [v\Gamma^0(\lambda) v]\big(R_0^\pm(\lambda^4)(*,y)\big)(\cdot),~ R_0^\mp(\lambda)(x,\cdot)   \Big\rangle\\
=&\frac{1}{\lambda^4}\Big\langle[v\Gamma^0(\lambda) v] \big(e^{\pm i\lambda|*-y|}\widetilde{R}^\pm(\lambda|*-y|)\big)(\cdot),~
\big(e^{\mp i\lambda|x-\cdot|}\widetilde{R}^\mp(\lambda|x-\cdot|)\big)   \Big\rangle\\
=&\frac{1}{\lambda^4}e^{\pm i\lambda(|x|+|y|)}\Big\langle[v\Gamma^0(\lambda) v] \big(e^{\pm i\lambda(|*-y|-|y|)}\widetilde{R}^\pm(\lambda|*-y|)\big)(\cdot),~
\big(e^{\mp i\lambda(|x-\cdot|-|x|)}\widetilde{R}^\mp(\lambda|x-\cdot|)\big)   \Big\rangle\\
:=&\frac{1}{\lambda^4}e^{\pm i\lambda(|x|+|y|)}E^{0,\pm}_{3}(\lambda;x,y).
  \end{split}
\end{equation*}
Let $\lambda =2^Ns$, then
\begin{equation*}
   \begin{split}
K^{0,\pm}_{3,N}(t;x,y)= \int_0^\infty e^{-it2^{4N}s^4}s^{-1}\varphi_0(s)
e^{\pm i2^Ns(|x|+|y|)}
E^{0,\pm}_{3}(2^Ns;x,y) ds.
  \end{split}
\end{equation*}
By (\ref{freeres-lagrepart-Rtubap}) it is easy to check that for $k=0,1$,
$$ \big|\partial_s^k\big(e^{\pm i\lambda(|\cdot-y|-|y|)}\widetilde{R}^\pm(\lambda|\cdot-y|)\big)\big|
 \lesssim \langle \cdot \rangle^k\ \ {\rm and }\ \
  \big|\partial_s^k\big(e^{\pm i\lambda(|x-\cdot|-|x|)}\widetilde{R}^\pm(\lambda|x-\cdot|)\big)\big|
 \lesssim \langle \cdot \rangle^k.$$
Note that $\|\Gamma^0(\lambda)\|_{L^2\rightarrow L^2}= O_1(\lambda^2)$ ( Choosing some $N'\in \mathbb{Z}$
such that $\lambda\ll1$ and $\lambda<  2^{N'}$ ), then we have
$$\big\|\partial_s^k\big(\Gamma^0(2^Ns)\big)\big\|_{L^2\rightarrow L^2} \lesssim 2^{2N}s^{2-k},
 k=0,1.$$
Hence,
$$\big|\partial_s^k \big(E^{0,\pm}_{3}(2^Ns;x,y)\big)\big| \lesssim \big\|\partial_s^k\big(\Gamma^0(2^Ns)\big)\big\|_{L^2\rightarrow L^2} \lesssim 2^{2N}s^{2-k},
 k=0,1.$$
Let $ z=(x,y)$, $ \Phi(z)=|x|+|y|$, and $\Psi(2^Ns;z)= E^{0,\pm}_{3}(2^Ns;x,y)$.
 Using Lemma \ref{lem-LWP} again, we obtain that $ |K^{0,\pm}_{3,N}(t;x,y)|$ is
bounded by $2^{2N}(1+|t|2^{4N})^{-1/2}$ uniformly in  $x,y$.
\end{proof}

\subsubsection{\textbf{The first kind of resonance}}
If zero is a resonance of the first kind, then by using (\ref{id-RV}) and (\ref{thm-resoinver-M1 })
one has
\begin{equation}\label{RV-first resonance}
   \begin{split}
R_V^\pm(\lambda^4)
=& R^\pm_0(\lambda^4)- R^\pm_0(\lambda^4)v\Big(g^\pm_0(\lambda) S_1D_2 S_1\Big)vR^\pm_0(\lambda^4)
-R^\pm_0(\lambda^4)v\Big( S_0\Gamma_{0,1}^1 S_0\\
&+ S_1\Gamma_{0,2}^1 Q+ Q\Gamma_{0,3}^1 S_1\Big)vR^\pm_0(\lambda^4)
-R^\pm_0(\lambda^4)v\Big(g^\pm_0(\lambda)^{-1}Q\Gamma_{0,4}^1Q\Big)vR^\pm_0(\lambda^4)\\
&-R^\pm_0(\lambda^4)v\Big(g^\pm_0(\lambda)^4\lambda^2
S_1\Gamma_{2,1}^1S_1\Big)vR^\pm_0(\lambda^4)
-g^\pm_0(\lambda)^3\lambda^2R^\pm_0(\lambda^4)v\Big(S_1\Gamma_{2,2}^1Q\\
&+ Q\Gamma_{2,3}^1S_1 \Big)vR^\pm_0(\lambda^4)
-g^\pm_0(\lambda)^2\lambda^2R^\pm_0(\lambda^4)v\Big(Q \Gamma_{2,4}^1Q
+ S_1\Gamma_{2,5}^1 +\Gamma_{2,6}^1S_1 \Big)\\
&\times vR^\pm_0(\lambda^4)
-g^\pm_0(\lambda)\lambda^2R^\pm_0(\lambda^4)v\Big(
Q\Gamma_{2,7}^1 +\Gamma_{2,8}^1Q  \Big) vR^\pm_0(\lambda^4)\\
&-R^\pm_0(\lambda^4)\Big(vO_1(\lambda^2)v\Big)R^\pm_0(\lambda^4).
\end{split}
\end{equation}
In order to prove Theorem \ref{thm-low}(ii), we will make use of the resolvent difference $R_V^+(\lambda^4)-R_V^-(\lambda^4)$ in
Stone's formula
\begin{equation}\label{stoneformula-first+alpha}
   \begin{split}
 e^{-itH}P_{ac}(H)\chi(H)f=& \frac{2}{\pi i}\int_0^\infty e^{-it\lambda^4}\chi(\lambda)\lambda^3
[R_V^+(\lambda^4)- R_V^-(\lambda^4)]fd\lambda.
\end{split}
\end{equation}
In view of the analysis of regular case, it is suffice to consider the term
$$R^\pm_0(\lambda^4)v\Big(g^\pm_0(\lambda) S_1D_2S_1\Big)vR^\pm_0(\lambda^4):=\Omega^\pm_1(\lambda).$$
Note that
\begin{equation}\label{firstres-firstterm}
   \begin{split}
\Omega^+_1(\lambda)-\Omega^-_1(\lambda)
=&\big(R^+_0(\lambda^4)-R_0^-(\lambda^4)\big) v\Big(g^+_0(\lambda) S_1D_2 S_1\Big)vR^+_0(\lambda^4)\\
&+R^-_0(\lambda^4)v \Big(g^-_0(\lambda)S_1 D_2 S_1\Big) v\big(R^+_0(\lambda^4)-R_0^-(\lambda^4)\big)\\
&+ R^-_0(\lambda^4)v\Big(\big( g^+_0(\lambda)-g^-_0(\lambda)\big) S_1D_2 S_1\Big)vR^+_0(\lambda^4)\\
:=&\Omega_{1,1}(\lambda)+\Omega_{1,2}(\lambda)+\Omega_{1,3}(\lambda).
\end{split}
\end{equation}
 Hence we only need to prove the following proposition.
\begin{proposition}\label{prop-first-S1gamma101S1}
Assume that $|V(x)|\lesssim (1+|x|)^{-\beta}$ with $\beta >14$. Let $\Omega_{1,i}(\lambda)(i=1,2,3)$ be the operators defined in (\ref{firstres-firstterm}).
then for $N\in \mathbb{Z}$ and $N \leq N'$,
\begin{equation*}
   \begin{split}
&\sup\limits_{x,y}\Big|\int_0^\infty e^{-it\lambda^4}\lambda^3\varphi_0(2^{-N}\lambda)
\Omega_{1,i}(\lambda)(x,y)d\lambda\Big|\lesssim 2^{2N}(1+|t|2^{4N})^{-1/2},\ i=1,2,3.
\end{split}
\end{equation*}
\end{proposition}
\begin{proof}
We first prove the case with $\Omega_{1,1}(\lambda)$.  We write that
\begin{equation*}
   \begin{split}
K^{1,\pm}_{1,N}(t;x,y)=\int_0^\infty e^{-it\lambda^4}\lambda^3\varphi_0(2^{-N}\lambda)\big[\big(R^+_0(\lambda^4)
-R_0^-(\lambda^4)\big)
v\big(g^+_0(\lambda) S_1D_2 S_1\big)vR^+_0(\lambda^4)\big](x,y)d\lambda.
\end{split}
\end{equation*}
 $$R_0^+(\lambda^4)(x,y)=\frac{1}{\lambda^2}F^+(\lambda|x-y|) \  \hbox{and}\
\big[R_0^+(\lambda^4)-R^-_0(\lambda^4)\big](x,y)=\frac{1}{\lambda^2}\bar{F}(\lambda|x-y|).$$
From Remark \ref{remark-beha-Rpm} we know that $\bar{F}(p)\in C^3(\mathbb{R})$, $\bar{F}'(0)=0$.
Using the orthogonality $S_1v=0$ and $S_1(x_jv)=0(j=1,2)$,
it follows by Lemma \ref{Taylor-low}(i) and (ii) that
\begin{equation*}
   \begin{split}
&\Big[\big(R^+_0(\lambda^4)
-R_0^-(\lambda^4)\big)
v\big(g^+_0(\lambda) S_1D_2 S_1\big)vR^+_0(\lambda^4)\Big](x,y)\\
=& \frac{g^+_0(\lambda)}{\lambda^4}\int_{\mathbb{R}^4}\bar{F}(\lambda|x-u_2|)
[vS_1D_2 S_1v](u_2,u_1)F^+(\lambda |y-u_1|)du_1du_2\\
=&\frac{g^+_0(\lambda)}{\lambda}\int_{\mathbb{R}^4}\Big(\int_0^1\int_0^1
\Big(\bar{F}''(\lambda|x-\theta_2u_2|)\cos^2\alpha_2+
\frac{\bar{F}'(\lambda|x-\theta_2u_2|)}{\lambda|x-\theta_2u_2|}\sin^2\alpha_2\Big)
(F^+)'(\lambda|y-\theta_1u_1|)\\
&(1-\theta_2) \cos\alpha_1d\theta_1d\theta_2\Big)\
|u_1||u_2|^2[vS_1D_2 S_1v](u_2,u_1)du_1du_2.
  \end{split}
\end{equation*}
Thus,  we have
\begin{equation*}
   \begin{split}
K_{1,N}^{1,\pm}(t;x,y)=&\int_{\mathbb{R}^4}\int_0^1\int_0^1
\bigg(\int_0^\infty e^{-it\lambda^4}\varphi_0(2^{-N}\lambda)\lambda^2g^+_0(\lambda) \Big(\bar{F}''(\lambda|x-\theta_2u_2|)\cos^2\alpha_2\\
&+\frac{\bar{F}'(\lambda|x-\theta_2u_2|)}{\lambda|x-\theta_2u_2|}\sin^2\alpha_2\Big)
(F^+)'(\lambda|y-\theta_1u_1|) d\lambda\bigg)
(1-\theta_2) \cos\alpha_1d\theta_1d\theta_2\\
&\times|u_1||u_2|^2[vS_1D_2 S_1v](u_2,u_1)du_1du_2.
\end{split}
\end{equation*}
Set
\begin{equation*}
   \begin{split}
E^{1,\pm}_{1,N}(t;x,y,\theta_1,\theta_2,u_1,u_2)=& \int_0^\infty e^{-it\lambda^4}\varphi_0(2^{-N}\lambda)\lambda^2g^+_0(\lambda) \Big(\bar{F}''(\lambda|x-\theta_2u_2|)\cos^2\alpha_2\\
&+\frac{\bar{F}'(\lambda|x-\theta_2u_2|)}{\lambda|x-\theta_2u_2|}\sin^2\alpha_2\Big)
(F^+)'(\lambda|y-\theta_1u_1|) d\lambda.
  \end{split}
\end{equation*}
Then we have
\begin{equation}\label{prop-esti-sec-k11N}
   \begin{split}
|K_{1,N}^{1,\pm}(t;x,y)|\lesssim & \int_{\mathbb{R}^4}\Big(\int_0^1\int_0^1|E^{1,\pm}_{1,N}(t;x,y,\theta_1,\theta_2,u_1,u_2)||(1-\theta_2)|d\theta_1d\theta_2\Big)\\
&\ \ \ \times|u_1||u_2|^2|[vS_1D_2S_1v](u_2,u_1)|du_1du_2.
  \end{split}
\end{equation}

Now let's begin to estimate $E^{1,\pm}_{1,N}(t;x,y,\theta_1,\theta_2,u_1,u_2)$. In fact,
set
$$ (F^+)'(p)= e^{ ip}F^+_1(p)  \ \hbox{and}\
\frac{\bar{F}'(p)\sin^2\alpha_2}{p}+\bar{F}''(p)\cos^2\alpha_2=e^{ip}\bar{F}_{11}(p)+e^{-ip}\bar{F}_{12}(p),$$
where $F^+_1(p) $ is the same function as in \eqref{freereso-Fp-daoshu}, 
$\bar{F}_{11}(p)$ and $\bar{F}_{12}(p)$ satisfy that by (\ref{reso-R+RF-big})
\begin{equation}\label{Ftuba-onetwodaoshu-big}
   \begin{split}
\bar{F}_{11}(p)=&\frac{1}{8p}\Big( -w_+(p)+iw_+'(p)\Big)\sin^2\alpha_2
- \frac{i}{8}\Big(w_+(p)+w_-(p) \Big)\cos^2\alpha_2,\ p\gg1,\\
\bar{F}_{11}(p)=& e^{-ip}\Big[\frac{i\sin^2\alpha_2}{4}\Big(-\frac{1}{2}p+\frac{1}{16}p^2-\frac{1}{39}p^4\Big)
+\frac{i\cos^2\alpha_2}{4}\Big(-p+\frac{3}{16}p^2-\frac{5}{39}p^4\Big)\\
&\ \ \ \ \ \ \ \ \ \,+O(p^6)\Big],\ p\ll 1.
  \end{split}
\end{equation}
\begin{equation}\label{Ftuba-onetwodaoshu-small}
   \begin{split}
\bar{F}_{12}(p)=&\frac{1}{8p}\Big(w_-(p)+iw_-'(p) \Big)\sin^2\alpha_2
+  \frac{1}{4}\Big(w_+'(p)+w_-'(p)\Big)\cos^2\alpha_2,\ p\gg1, \\  
\bar{F}_{12}(p)=& e^{ip}\Big[\frac{i\sin^2\alpha_2}{4}\Big(-\frac{1}{2}p+\frac{1}{16}p^2-\frac{1}{39}p^4\Big)
+\frac{i\cos^2\alpha_2}{4}\Big(-p+\frac{3}{16}p^2-\frac{5}{39}p^4\Big)\\
&\ \ \ \ \ \ \ \ \ \,+O(p^6)\Big],\ p\ll 1.
  \end{split}
\end{equation}
Hence  we have
\begin{equation}\label{E11Ntxy}
   \begin{split}
&E^{1,\pm}_{1,N}(t;x,y,\theta_1,\theta_2,u_1,u_2)\\
=& \int_0^\infty e^{-it\lambda^4}\varphi_0(2^{-N}\lambda)\lambda^2g^+_0(\lambda)
e^{i\lambda|x-\theta_2u_2|}e^{i\lambda|y-\theta_1u_1|}\bar{F}_{11}(\lambda|x-\theta_2u_2|)
F^+_1(\lambda|y-\theta_1u_1|) d\lambda\\
&+\int_0^\infty e^{-it\lambda^4}\varphi_0(2^{-N}\lambda)\lambda^2g^+_0(\lambda)
e^{-i\lambda|x-\theta_2u_2|}e^{i\lambda|y-\theta_1u_1|}\bar{F}_{12}(\lambda|x-\theta_2u_2|)
F^+_1(\lambda|y-\theta_1u_1|) d\lambda\\
:=&E^{1,\pm}_{11,N}(t;x,y,\theta_1,\theta_2,u_1,u_2)+E^{1,\pm}_{12,N}(t;x,y,\theta_1,\theta_2,u_1,u_2).
 \end{split}
\end{equation}
For the first term $E^{1,\pm}_{11,N}(t;x,y,\theta_1,\theta_2,u_1,u_2)$ in \eqref{E11Ntxy}.
Let $\lambda =2^Ns$, we have
\begin{equation*}
   \begin{split}
&E^{1,\pm}_{11,N}(t;x,y,\theta_1,\theta_2,u_1,u_2)\\
=& \int_0^\infty e^{-it\lambda^4}\varphi_0(2^{-N}\lambda)\lambda^2g^+_0(\lambda)
e^{i\lambda|x-\theta_2u_2|}e^{i\lambda|y-\theta_1u_1|}\bar{F}_{11}(\lambda|x-\theta_2u_2|)
F^+_1(\lambda|y-\theta_1u_1|) d\lambda\\
=& 2^{2N}\int_0^\infty e^{-it2^{4N}s^4}s\varphi_0(s)
e^{i2^Ns (|x-\theta_2u_2|+|y-\theta_1u_1|)}(2^Ns)g^+_0(2^Ns)\bar{F}_{11}(2^Ns|x-\theta_2u_2|)\\
&\ \ \ \ \ \ \ \ F^+_{11}(2^Ns|y-\theta_1u_1|) ds.
  \end{split}
\end{equation*}
Let $z=(x,y,\theta_1,\theta_2,u_1,u_2)$, $\Phi(z)= |x-\theta_2u_2|+|y-\theta_1u_1|$, and
$$\Psi_1(2^Ns, z)=(2^Ns)g^+_0(2^Ns)\bar{F}_{11}(2^Ns|x-\theta_2u_2|)F^+_1(2^Ns|y-\theta_1u_1|) . $$
By (\ref{freereso-Fp-daoshu-big}), (\ref{freereso-Fp-daoshu-small}) and (\ref{Ftuba-onetwodaoshu-big}), 
it is easy check that for any $z$,
$$ |\partial_s^k \Psi_1(2^Ns, z)|\lesssim 1,\ k=0,1. $$
Hence by using Lemma \ref{lem-LWP} again, we immediately obtain that $|E^{1,\pm}_{11,N}(t;x,y,\theta_1,\theta_2,u_1,u_2)|$  is bounded by $2^{2N}(1+|t|2^{4N})^{-1/2}.$ 
Similarly, we obtain the same bounds for $|E^{1,\pm}_{12,N}(t;x,y,\theta_1,\theta_2,u_1,u_2)|$.
By \eqref{E11Ntxy}, we immediately get that  $|E^{1,\pm}_{1,N}(t;x,y,\theta_1,\theta_2,u_1,u_2)|$ is bounded by $2^{2N}(1+|t|2^{4N})^{-1/2}$.

Finally, by (\ref{prop-esti-sec-k11N}) and H\"{o}lder's inequality
we obtain that $|K_{1,N}^{1,\pm}(t;x,y)|$ is bounded by $2^{2N}(1+|t|2^{4N})^{-1/2}$
uniformly in $x,y$.

For the case with $\Omega_{1,2}(\lambda)$, we can obtain the same estimates by the similar proof to $\Omega_{1,1}(\lambda)$.
For the case with $\Omega_{1,3}(\lambda)$. Note that
$\Omega_{1,3}(\lambda)= -\frac{i}{16}R^-_0(\lambda^4)v\Big( S_1D_2 S_1\Big)vR^+_0(\lambda^4)$,
hence we can prove that the conclusion holds by the same argument as the proof of Proposition \ref{prop-S0D1S0}.
\end{proof}

\subsubsection{\textbf{The second kind of resonance}}
If zero is the second kind resonance  of $H$, then using (\ref{id-RV}) and
(\ref{thm-resoinver-M2 }) one has
\begin{equation}\label{secondresinance-RV}
  \begin{split}
R^\pm_V(\lambda^4)=&R^\pm_0(\lambda^4)
-R^\pm_0(\lambda^4)v\Big(\frac{h^\pm(\lambda)}{\lambda^2} S_2\Gamma^2_{-2,1}S_2\Big)vR^\pm_0(\lambda^4)
-R^\pm_0(\lambda^4)v\Big(g^\pm_0(\lambda)^5\big( S_2\Gamma^2_{0,1} \\
&+\Gamma^2_{0,2}S_2 +S_1\Gamma^2_{0,3}Q + Q\Gamma^2_{0,4}S_1 \big)\Big)vR^\pm_0(\lambda^4)
-R^\pm_0(\lambda^4)v\Big(g^\pm_0(\lambda)^{10}\lambda^2 \big( S_2\Gamma^2_{2,1}\\
& +\Gamma^2_{2,2}S_2 +S_1\Gamma^2_{2,3}Q + Q\Gamma^2_{2,4}S_1\big) \Big)vR^\pm_0(\lambda^4)
-R^\pm_0(\lambda^4)v\Big(g^\pm_0(\lambda)\lambda^2\big( S_2\Gamma^2_{2,5}\\
& +\Gamma^2_{2,6}S_2+S_1\Gamma^2_{2,7}Q + Q\Gamma^2_{2,8}S_1\big) \Big)vR^\pm_0(\lambda^4) -R^\pm_0(\lambda^4)\Big(vO_1(\lambda^2)v\Big)R^\pm_0(\lambda^4).
\end{split}
\end{equation}
In order to prove Theorem \ref{thm-low}(iii),
 comparing with the cases of the regular zero point and the first kind zero resonance,
 it suffices to study the following term in \eqref{secondresinance-RV},
$$\Omega^\pm_2(\lambda):=R^\pm_0(\lambda^4)v\Big( \frac{h^\pm(\lambda)}{\lambda^2}S_2\Gamma_{-2,1}^2S_2 \Big)vR^\pm_0(\lambda^4).$$
Note that
\begin{equation}\label{secondres-firstterm}
   \begin{split}
\Omega^+_2(\lambda)-\Omega^-_2(\lambda)
&=\big(R^+_0(\lambda^4)-R_0^-(\lambda^4)\big) v\Big(\frac{h^+(\lambda)}{\lambda^2}S_2\Gamma_{-2,1}^2(\lambda)S_2 \Big)vR^+_0(\lambda^4)\\
&+R^-_0(\lambda^4)v\Big(\frac{h^-(\lambda)}{\lambda^2}S_2\Gamma_{-2,1}^2(\lambda)S_2 \Big)v
\big(R^+_0(\lambda^4)-R_0^-(\lambda^4)\big)\\
&+R_0^-(\lambda^4)v\Big(\frac{h^+(\lambda)-h^-(\lambda)}{\lambda^2}
S_2\Gamma_{-2,1}^2(\lambda)S_2   \Big)vR^+_0(\lambda^4)\\
:=&\Omega_{2,1}(\lambda)+\Omega_{2,3}(\lambda)+\Omega_{2,3}(\lambda).
\end{split}
\end{equation}
Hence, it suffices to prove the following proposition.
\begin{proposition}\label{prop-sencond-S22-01-S2}
Assume that $|V(x)|\lesssim (1+|x|)^{-\beta}$ with some $\beta >18$. If $\Omega_{2,i}(\lambda)\ ( i=1,2,3)$ be the operators defined in (\ref{secondres-firstterm}), then for $N\in \mathbb{Z}$ and $N \leq N'$,
\begin{equation*}
   \begin{split}
\sup\limits_{x,y}\Big|\int_0^\infty e^{-it\lambda^4}\lambda^{3+4\alpha}\varphi_0(2^{-N}\lambda)\Omega_{2,i}(\lambda)(x,y)d\lambda\Big|
\lesssim  2^{(1+4\alpha)N}(1+|t|2^{4N})^{-1/2}, \ i=1,2;
\end{split}
\end{equation*}
\begin{equation*}
   \begin{split}
\sup\limits_{x,y}\Big|\int_0^\infty e^{-it\lambda^4}\lambda^{3+4\alpha}\varphi_0(2^{-N}\lambda)\Omega_{2,3}(\lambda)(x,y)d\lambda\Big|
\lesssim  2^{4\alpha N}(1+|t|2^{4N})^{-1/2}.
\end{split}
\end{equation*}
In particular, for $\alpha =\frac{1}{2} $,
\begin{equation*}
   \begin{split}
\sup\limits_{x,y}\Big|\int_0^\infty e^{-it\lambda^4}\lambda^{3+4\alpha}\ \Omega_{2,i}(\lambda)(x,y)d\lambda\Big|
\lesssim |t|^{-\frac{1}{2}},\, i=1,2,3.
\end{split}
\end{equation*}
\end{proposition}
\begin{proof}
For the case with $ \Omega_{2,1}(\lambda)$, we first use the orthogonality  $S_2v=S_2(x_1v)=S_2(x_2v)=0$ of $S_2$ for  on the left hand of $\Omega_{2,1}(\lambda)$,  and use the orthogonality $S_2v=0$ of $S_2$ on the right hand of
$\Omega_{2,1}(\lambda)$. Then by the same argument with the proof of Proposition \ref{prop-first-S1gamma101S1}, it is easy to check that the statement holds. For the case of $\Omega_{2,2}(\lambda)$, the proof proceeds similarly as the case $ \Omega_{2,1}(\lambda)$.

Next, we turn to the case of $ \Omega_{2,3}(\lambda)$.
We write that
\begin{equation*}
   \begin{split}
&K_{1,N}^{2,\pm}(t;x,y)\\
=&\frac{h^+(\lambda)-h^-(\lambda)}{\lambda^2}\int_0^\infty  e^{-it\lambda^4}\lambda^{3+4\alpha}\varphi_0(2^{-N}\lambda)\big[R_0^-(\lambda^4)v
S_2\Gamma_{-2,1}^2(\lambda) S_2 vR_0^+(\lambda^4)\big](x,y)
 d\lambda.
  \end{split}
\end{equation*}
Let $R_0^\pm(\lambda^2)(x,y)=\frac{1}{\lambda^2}F^\pm(\lambda|x-y|).$
Then by using the orthogonality $S_2v=0$ and Lemma \ref{Taylor-low}(i),
\begin{equation}\label{S2-proof}
   \begin{split}
&\big[R_0^-(\lambda^4)v
S_2\Gamma_{-2,1}^2(\lambda) S_2 vR_0^+(\lambda^4)\big](x,y)\\
=& \frac{1}{\lambda^4}\int_{\mathbb{R}^4}F^-(\lambda|x-u_2|)
     [vS_2\Gamma_{-2,1}^2(\lambda)S_2v](u_2,u_1)F^+(\lambda |y-u_1|)du_1du_2\\
=&\frac{1}{\lambda^2}\int_{\mathbb{R}^4}\Big(\int_0^1\int_0^1(F^-)'(\lambda|x-\theta_2u_2|)
(F^+)'(\lambda|y-\theta_1u_1|)\cos\alpha_2\cos\alpha_1d\theta_1d\theta_2\Big) \\
&\ \ \ \ \ \ \ \ \ \times|u_1||u_2|[vS_2\Gamma_{-2,1}^2(\lambda) S_2v](u_2,u_1)du_1du_2.
  \end{split}
\end{equation}
Hence we have
\begin{equation*}
   \begin{split}
&K_{1,N}^{2,\pm}(t;x,y)\\
=&\int_0^1\int_0^1\Big(\int_0^\infty e^{-it\lambda^4}\lambda^{4\alpha}
\varphi_0(2^{-N}\lambda)\frac{h^+(\lambda)-h^-(\lambda)}{\lambda}
\int_{\mathbb{R}^4}(F^-)'(\lambda|x-\theta_2u_2|)
(F^+)'(\lambda|y-\theta_1u_1|)\\
& \ \ \ \ \ \ \ \ \ \times|u_1||u_2|v(u_1)v(u_2)[S_2\Gamma_{-2,1}^2(\lambda) S_2](u_2,u_1)du_1du_2
d\lambda\Big)
\cos\alpha_2\cos\alpha_1d\theta_1d\theta_2.
 \end{split}
\end{equation*}
Let
\begin{equation*}
   \begin{split}
E^{2,\pm}_{1,N}(t;x,y,\theta_1,\theta_2)= & \int_0^\infty e^{-it\lambda^4}\lambda^{4\alpha}
\varphi_0(2^{-N}\lambda)\frac{h^+(\lambda)-h^-(\lambda)}{\lambda}\Big(\int_{\mathbb{R}^4}
(F^-)'(\lambda|x-\theta_2u_2|)\\
& \ \ \ \times(F^+)'(\lambda|y-\theta_1u_1|)
|u_1||u_2|v(u_1)v(u_2)[S_2\Gamma_{-2,1}^2(\lambda) S_2](u_2,u_1)du_1du_2\Big)d\lambda.
  \end{split}
\end{equation*}
Then we have
\begin{equation}\label{esti-second-K2N0pm}
   \begin{split}
|K_{1,N}^{2,\pm}(t;x,y)|\lesssim \int_0^1\int_0^1|E^{2,\pm}_{1,N}(t;x,y,\theta_1,\theta_2)|d\theta_1d\theta_2.
  \end{split}
\end{equation}
Notice that $h^+(\lambda)-h^-(\lambda)= O_1\big((\ln\lambda)^{-1}\big)$, by using similar arguments as in the proof of Proposition \ref{prop-QlambdaQ}, we have for any $x,y\in \mathbb{R}^2$ and $0 \leq \theta_i \leq 1, i=1,2$,
\begin{equation}\label{estu-second-E0pm2N}
   \begin{split}
|E^{0,\pm}_{2,N}(t;x,y,\theta_1,\theta_2)| \lesssim 2^{4\alpha N}(1+|t|2^{4N})^{-1/2}.
  \end{split}
\end{equation}
Therefore from \eqref{esti-second-K2N0pm}, we immediately obtain that $|K_{2,N}^{0,\pm}(t;x,y)|$
is boudned by $2^{4\alpha N}(1+|t|2^{4N})^{-1/2}$ uniformly in $x,y$.
\end{proof}

\begin{remark}\label{Remark of second kind}
We notice that the projection $S_2$ actually
has the three orthogonal relations $S_2v=S_2x_1v=S_2x_2v=0$. However, for the kernel $\big[R_0^-(\lambda^4)vS_2\Gamma_{-2,1}^2(\lambda) S_2 vR_0^+(\lambda^4)\big](x,y)$ in \eqref{S2-proof} above, the only orthogonal relation $S_2v=0$ can be used since the derivatives
$(F^\pm)^{(k)}(p)$ of $F^{\pm}(p)$ is bounded uniformly in $p$ as $k=0,1$, but not uniformly bounded for $k\ge2$. This leads to need an extra regular term $H^{1/2}$ to modify the singularity of $\lambda$ near zero in order to obtain the decay estimates $O(|t|^{-1/2})$ of $ e^{-itH}P_{ac}(H)$.
 In principle, if all orthogonal relations of $S_2$ can be used in \eqref{S2-proof}, then the desired time decay estimate $O(|t|^{-1/2})$ should be obtained without such regular factor $H^{1/2}$. Moreover, we remark that that some regular terms $H^{s}$ are also indispensable  based on such similar reasons and the existences of worse error terms in the following third and fourth resonance cases.
\end{remark}

\subsubsection{\textbf{The third and fourth kind resonances}}
If zero is the  third kind resonance of $H$, then using (\ref{id-RV}) and
(\ref{thm-resoinver-M3}) one has
\begin{equation}\label{Rv-thirdreso}
   \begin{split}
R^\pm_V(\lambda^4)=&R^\pm_0(\lambda^4)
- R^\pm_0(\lambda^4)v\Big( \frac{ \widetilde{g}_2^\pm(\lambda)^{-1} }{\lambda^4}S_4\Gamma_{-4,1}^3S_4   \Big)vR^\pm_0(\lambda^4)
-R^\pm_0(\lambda^4)v\\
&\times \Big( \frac{\widetilde{g}_2^\pm(\lambda)^{-2}}{\lambda^4}S_4\Gamma_{-4,1}^3S_4    \Big)vR^\pm_0(\lambda^4)
-R^\pm_0(\lambda^4)\Big(vO_1\big(\lambda^{-4}(\ln\lambda)^{-3}\big)v\Big)R^\pm_0(\lambda^4).
  \end{split}
 \end{equation}
If zero is the fourth kind resonance of $H$, then using (\ref{id-RV}) and
(\ref{thm-resoinver-M4}) one has
\begin{equation}\label{Rv-fourthreso}
   \begin{split}
&R^\pm_V(\lambda^4)\\
=&R^\pm_0(\lambda^4)
- R^\pm_0(\lambda^4)v\Big(\frac{1}{\lambda^4}S_5D_6S_5 \Big)vR^\pm_0(\lambda^4)
-R^\pm_0(\lambda^4)v\Big( \frac{\widetilde{g}_2^+(\lambda)^{-1}}{\lambda^4} S_4\Gamma_{-4,1}^4S_4\Big)
vR^\pm_0(\lambda^4)\\
&-R^\pm_0(\lambda^4)v\Big( \frac{\widetilde{g}_2^+(\lambda)^{-2}}{\lambda^4}S_4\Gamma_{-4,2}^4S_4    \Big)vR^\pm_0(\lambda^4)
-R^\pm_0(\lambda^4)\Big(vO_1\big(\lambda^{-4}(\ln\lambda)^{-3}\big)v\Big)R^\pm_0(\lambda^4).
  \end{split}
 \end{equation}

In order to prove Theorem \ref{thm-low}(iv), we plug (\ref{Rv-thirdreso}) and (\ref{Rv-fourthreso}) into Stone's formula (\ref{Stone-Paley}), respectively.
Then combining with Proposition \ref{prop-free estimates},
it suffices to show the following Propositions \ref{prop-third-S43-01-S4} and \ref{prop-third-errorterm}.

\begin{proposition}\label{prop-third-S43-01-S4}
Assume that $V(x)\lesssim (1+|x|)^{-\beta}$ with $\beta>18$. Let $ \Omega_{3,1}^\pm(\lambda):= \frac{ \widetilde{g}_2^\pm(\lambda)^{-1} }{\lambda^4}S_4\Gamma_{-4,1}^3S_4 $ and $ \Omega_{4,1}^\pm(\lambda):=\frac{1}{\lambda^4}S_5D_6S_5$.
Then for $i=3,4$ and $N \leq N_0'$
\begin{equation*}
   \begin{split}
&\sup\limits_{x,y}\Big|\int_0^\infty e^{-it\lambda^4}\lambda^{3+4\alpha}\varphi_0(2^{-N}\lambda)
\big[R^\pm_0(\lambda^4)v\Omega_{i,1}^\pm(\lambda)vR^\pm_0(\lambda^4)\big](x,y)d\lambda\Big| \lesssim
2^{(-2+4\alpha)N}(1+|t|2^{4N})^{-1/2}.
\end{split}
\end{equation*}
In particular, for $\alpha=1$,
\begin{equation*}
   \begin{split}
&\sup\limits_{x,y}\Big|\int_0^\infty e^{-it\lambda^4}\lambda^{3+4\alpha}
\big[R^\pm_0(\lambda^4)v\Omega_{i,1}^\pm(\lambda)vR^\pm_0(\lambda^4)\big](x,y)d\lambda\Big| \lesssim
|t|^{-1/2},\, i=3,4.
\end{split}
\end{equation*}
\end{proposition}

\begin{proof}
We begin with proving the case with $\Omega^\pm_{3,1}(\lambda)$. First note that
\begin{equation}\label{third-Fourth kind}
   \begin{split}
&R^+_0(\lambda^4)v\Omega^+_{3,1}(\lambda)vR^+_0(\lambda^4)
-R^-_0(\lambda^4)v\Omega^-_{3,1}(\lambda)vR^-_0(\lambda^4)\\
=&\big(R^+_0(\lambda^4)-R_0^-(\lambda^4)\big) v\Big(\frac{ \widetilde{g}_2^+(\lambda)^{-1} }{\lambda^4}S_4\Gamma_{-4,1}^3S_4 \Big)vR^+_0(\lambda^4)
+R^-_0(\lambda^4)v\Big(\frac{ \widetilde{g}_2^-(\lambda)^{-1} }{\lambda^4}S_4\Gamma_{-4,1}^3S_4 \Big)\\
&\times v\big(R^+_0(\lambda^4)-R_0^-(\lambda^4)\big)
+R^-_0(\lambda^4)v\Big(\frac{\widetilde{g}_2^+(\lambda)^{-1}- \widetilde{g}_2^-(\lambda)^{-1} }{\lambda^4}S_4\Gamma_{-4,1}^3S_4 \Big)vR^+_0(\lambda^4)\\
&:=I+II+III.
\end{split}
\end{equation}
For the first term $I$, we write that
\begin{equation*}
K^{3,\pm}_{1,N}(t;x,y)
=\int_0^\infty e^{-it\lambda^4}\lambda^{3+4\alpha}\varphi_0(2^{-N}\lambda)\frac{\widetilde{g}^+_2(\lambda)^{-1}}{\lambda^4}
\big[\big(R^+_0(\lambda^4)
-R_0^-(\lambda^4)\big)
vS_4\Gamma_{-4,1}^3 S_4vR^+_0(\lambda^4)\big](x,y)d\lambda.
\end{equation*}
We below will use the orthogonal relations $S_4v=0, S_4(x_i v)=0, S_4(x_ix_j v)=0( i,j=1,2)$ for $S_4$ on the left side of the first term $I$, and the orthogonality $S_4v=0$ for $S_4$ on the right side of the first term $I$.

Let  $$R_0^+(\lambda^4)(x,y)=\frac{1}{\lambda^2}F^+(\lambda|x-y|),\ \hbox{ and} \
\big[R_0^+(\lambda^4)-R^-_0(\lambda^4)\big](x,y)=\frac{1}{\lambda^2}\bar{F}(\lambda|x-y|).$$
Then by using the orthogonality of $S_4$ we first obtain
 \begin{equation*}
   \begin{split}
&\big[\big(R^+_0(\lambda^4)-R_0^-(\lambda^4)\big)
v S_4\Gamma_{-4,1}^3 S_4vR^+_0(\lambda^4)\big](x,y)\\
=& \frac{1}{\lambda^4}\int_{\mathbb{R}^4}\bar{F}(\lambda|x-u_2|)
[vS_4\Gamma_{-4,1}^3 S_4v](u_2,u_1)F^+(\lambda |y-u_1|)du_1du_2\\
=&\frac{1}{\lambda^4}\int_{\mathbb{R}^4}
\big( \bar{F}(\lambda|x-u_2|)-\frac{i}{16}\lambda^2|x-u_2|^2\big)
[vS_4\Gamma_{-4,1}^3 S_4v](u_2,u_1)F^+(\lambda |y-u_1|)du_1du_2\\
=&\frac{1}{\lambda^4}\int_{\mathbb{R}^4}\widetilde{F}(\lambda|x-u_2|)
[vS_4\Gamma_{-4,1}^3 S_4v](u_2,u_1)F^+(\lambda |y-u_1|)du_1du_2,
  \end{split}
\end{equation*}
 where $ \widetilde{F}(p)= \bar{F}(p)-\frac{i}{16}p^2,\ p\in \mathbb{R}$.

Next, note that $\widetilde{F}(p)\in C^4(\mathbb{R})$ and $\widetilde{F}^{(k)}(0)=0 (k=1,2,3) $ by Remark \ref{remark-beha-Rpm}.
Hence by using Lemma \ref{Taylor-low}(i) and (iii), and the orthogonal relations $S_4v(x)=0, S_4(x_j v)=0$ and $S_4(x_jx_k v)=0 (j,k=1,2)$,
we have
 \begin{equation*}
   \begin{split}
&\big[\big(R^+_0(\lambda^4)-R_0^-(\lambda^4)\big)
v S_4\Gamma_{-4,1}^3S_4vR^+_0(\lambda^4)\big](x,y)\\
=&\frac{1}{2}\int_{\mathbb{R}^4}\Big(\int_0^1\int_0^1
\Big[\Big(\frac{\widetilde{F}'(\lambda|x-\theta_2u_2|)}{\lambda^2|x-\theta_2u_2|^2}
-\frac{\widetilde{F}''(\lambda|x-\theta_2u_2|)}{\lambda|x-\theta_2u_2|}\Big)3\cos\alpha_2\sin^2\alpha_2\\
&-\widetilde{F}'''(\lambda|x-\theta_2u_2|)\cos^3\alpha_2\Big]
(F^+)'(\lambda|y-\theta_1u_1|)(1-\theta_2)^2 \cos\alpha_1d\theta_1d\theta_2\Big)\\
&\ \ \ \ \ \times|u_1||u_2|^3[vS_4\Gamma_{-4,1}^3 S_4v](u_2,u_1)du_1du_2.
  \end{split}
\end{equation*}
Thus, by Fubini's theorem  we write
\begin{equation*}
   \begin{split}
&K_{1,N}^{3,\pm}(t;x,y)\\
=& \frac{1}{2}\int_{\mathbb{R}^4}\bigg(\int_0^1\int_0^1
\Big(\int_0^\infty e^{-it\lambda^4}\varphi_0(2^{-N}\lambda)
\lambda^{-1+4\alpha}\widetilde{g}^+_2(\lambda)^{-1}
\Big[\Big(\frac{\widetilde{F}'(\lambda|x-\theta_2u_2|)}{\lambda^2|x-\theta_2u_2|^2}
-\frac{\widetilde{F}''(\lambda|x-\theta_2u_2|)}{\lambda|x-\theta_2u_2|}\Big)\\
&\times 3\cos\alpha_2\sin^2\alpha_2
-\widetilde{F}'''(\lambda|x-\theta_2u_2|)\cos^3\alpha_2\Big]
(F^+)'(\lambda|y-\theta_1u_1|) d\lambda\Big)
(1-\theta_2)^2 \cos\alpha_1d\theta_1d\theta_2\bigg)\\
&\ \ \ \ \ \ \times|u_1||u_2|^3[v S_4\Gamma_{-4,1}^3 S_4v](u_2,u_1)du_1du_2.
\end{split}
\end{equation*}
 Let
\begin{equation*}
   \begin{split}
&E^{3,\pm}_{1,N}(t;x,y,\theta_1,\theta_2,u_1,u_2)\\
=&\int_0^\infty e^{-it\lambda^4}\varphi_0(2^{-N}\lambda)
\lambda^{-1+4\alpha}\widetilde{g}^+_2(\lambda)^{-1}
\Big[\Big(\frac{\widetilde{F}'(\lambda|x-\theta_2u_2|)}{\lambda^2|x-\theta_2u_2|^2}
-\frac{\widetilde{F}''(\lambda|x-\theta_2u_2|)}{\lambda|x-\theta_2u_2|}\Big)
3\cos\alpha_2\sin^2\alpha_2\\
&\ \ \ \ \ \ \ \ \ -\widetilde{F}'''(\lambda|x-\theta_2u_2|)\cos^3\alpha_2\Big]
(F^+)'(\lambda|y-\theta_1u_1|) d\lambda.
  \end{split}
\end{equation*}
 Then we have
 \begin{equation}\label{prop-esti-third-k13N}
   \begin{split}
|K_{1,N}^{3,\pm}(t;x,y)|\lesssim & \int_{\mathbb{R}^4}\Big(\int_0^1\int_0^1|E^{3,\pm}_{1,N}(t;x,y,\theta_1,\theta_2,u_1,u_2)|
|(1-\theta_2)^2|d\theta_1d\theta_2\Big)\\
&\ \ \ \ \ \ \ \times|u_1||u_2|^3|[v S_4\Gamma_{-4,1}^3 S_4v](u_2,u_1)|du_1du_2.
  \end{split}
\end{equation}

Now we begin to estimate $E^{3,\pm}_{1,N}(t;x,y,\theta_1,\theta_2,u_1,u_2)$.
In fact, we write that
\begin{equation}\label{Ftuba-onetwosandaoshu-small}
  \begin{split}
&\ \ \ \ \ \ \ \ \ \ \ \ \ \ \ \ \ \ \ \ \ \ (F^+)'(p)=e^{ip}F_1^+(p),\\
& \Big(\frac{\widetilde{F}'(p)}{p^2}-\frac{\widetilde{F}''(p)}{p}\Big)3\cos\alpha_2\sin^2\alpha_2
-\widetilde{F}'''(p)\cos^3\alpha_2 = e^{ip}\widetilde{F}_{11}(p)+e^{-ip}\widetilde{F}_{12}(p).
  \end{split}
\end{equation}
By \eqref{Ftuba-onetwosandaoshu-small}, we have 
\begin{equation}\label{E31Nx}
   \begin{split}
&E^{3,\pm}_{1,N}(t;x,y,\theta_1,\theta_2,u_1,u_2)\\
=&2^{4\alpha N} \int_0^\infty e^{-it2^{4N}s^4}\varphi_0(s)s^{-1+4\alpha}
\big(\widetilde{g}_0^\pm(2^Ns)\big)^{-1}
e^{i2^Ns (|x-\theta_2u_2|+|y-\theta_1u_1|)} \widetilde{F}_{11}(2^Ns|x-\theta_2u_2|)\\
&\times F^+_1(2^Ns|y-\theta_1u_1|) ds
+2^{4\alpha N} \int_0^\infty e^{-it2^{4N}s^4}\varphi_0(s)s^{-1+4\alpha}
\big(\widetilde{g}_0^\pm(2^Ns)\big)^{-1}\\
&\times e^{-i2^Ns (|x-\theta_2u_2|-|y-\theta_1u_1|)}
\widetilde{F}_{12}(2^Ns|x-\theta_2u_2|)F^+_1(2^Ns|y-\theta_1u_1|) ds\\
:=&E^{3,\pm}_{11,N}(t;x,y,\theta_1,\theta_2,u_1,u_2)+E^{3,\pm}_{12,N}(t;x,y,\theta_1,\theta_2,u_1,u_2).
  \end{split}
\end{equation}
For the first term $E^{3,\pm}_{11,N}(t;x,y,\theta_1,\theta_2,u_1,u_2)$ in \eqref{E31Nx}.
Let $\lambda =2^Ns$, then we get that
\begin{equation*}
   \begin{split}
E^{3,\pm}_{11,N}(t;x,y,\theta_1,\theta_2,u_1,u_2)
&=2^{4\alpha N} \int_0^\infty e^{-it2^{4N}s^4}\varphi_0(s)s^{-1+4\alpha}
\big(\widetilde{g}_0^\pm(2^Ns)\big)^{-1}
e^{i2^Ns (|x-\theta_2u_2|+|y-\theta_1u_1|)}\\
&\ \ \ \ \ \ \ \ \widetilde{F}_{11}(2^Ns|x-\theta_2u_2|)F^+_1(2^Ns|y-\theta_1u_1|) ds.
  \end{split}
\end{equation*}
Furthermore, set $z=(x,y,\theta_1,\theta_2,u_1,u_2)$, $\Phi(z)= |x-\theta_2u_2|+|y-\theta_1u_1|$ and
$$\Psi_1(2^Ns, z)= \big(\widetilde{g}_0^+(2^Ns)\big)^{-1}
\widetilde{F}_{11}(2^Ns|x-\theta_2u_2|)F^+_1(2^Ns|y-\theta_1u_1|).$$
By using  (\ref{reso-R+R-Ftuba-big}), (\ref{freereso-Fp-daoshu-big}), (\ref{freereso-Fp-daoshu-small}) and   (\ref{Ftuba-onetwosandaoshu-small}), one can get that for any $z$,
$$ |\partial_s^k \Psi_1(2^Ns, z)|\lesssim 1,\ k=0,1.$$
Hence it immediately follows from Lemma \ref{lem-LWP} that
\begin{equation}\label{3.43-a}
	|E^{3,\pm}_{11,N}(t;x,y,\theta_1,\theta_2,u_1,u_2)|\lesssim  2^{4\alpha N}(1+|t|2^{4N})^{-1/2}.\end{equation}
Similarly, we get the same bound for  $|E^{3,\pm}_{12,N}(t;x,y,\theta_1,\theta_2,u_1,u_2)| $  as the \eqref{3.43-a}. 
By \eqref{E31Nx}, we immediately get that $|E^{3,\pm}_{1,N}(t;x,y,\theta_1,\theta_2,u_1,u_2)| $ is bounded 
by $2^{4\alpha N}(1+|t|2^{4N})^{-1/2}$. 
Thus by \eqref{prop-esti-third-k13N} and H\"{o}lder's inequality,  we obtain that $|K_{1,N}^{3,\pm}(t;x,y)|$ is bounded by
 $ 2^{4\alpha N}(1+|t|2^{4N})^{-1/2}$ uniformly in $x,y$.

For the second term $II$ in \eqref{third-Fourth kind},  the proof proceeds identically with the first term $I$.

For the last term $III$ in \eqref{third-Fourth kind}, we write that
\begin{equation*}
   \begin{split}
&K_{2,N}^{3,\pm}(t;x,y)\\
=&\frac{\widetilde{g}_2^+(\lambda)^{-1}- \widetilde{g}_2^-(\lambda)^{-1} }{\lambda^4}\int_0^\infty  e^{-it\lambda^4}\lambda^{3+4\alpha}\varphi_0(2^{-N}\lambda)\big[R_0^-(\lambda^4)v
S_4\Gamma_{-4,1}^3S_4  vR_0^+(\lambda^4)\big](x,y)
 d\lambda\\
 =&-\frac{i}{9216 \lambda^4} \int_0^\infty  e^{-it\lambda^4}\lambda^{3+4\alpha}\varphi_0(2^{-N}\lambda)\big[R_0^-(\lambda^4)v
S_4\Gamma_{-4,1}^3S_4  vR_0^+(\lambda^4)\big](x,y) d\lambda.
  \end{split}
\end{equation*}
Here we use the orthogonality $S_4v=0$ and the same arguments as the proof of
Proposition \ref{prop-QlambdaQ}, we can conclude that $|K_{2,N}^{3,\pm}(t;x,y)|$ is bounded by
$2^{(-2 +\alpha)N}(1+|t|2^{4N})^{-1/2} $ uniformly in $x,y$.

Finally,  we come to prove the case with $\Omega^\pm_{4,1}(\lambda)$. Notice that
\begin{equation}\label{fourthreson-firstterm}
   \begin{split}
&R^+_0(\lambda^4)v\Big(\frac{1}{\lambda^4}S_5D_6S_5 \Big)vR^+_0(\lambda^4)
-R^-_0(\lambda^4)v\Big(\frac{1}{\lambda^4}S_5D_6S_5\Big)vR^-_0(\lambda^4)\\
&=\big(R^+_0(\lambda^4)-R_0^-(\lambda^4)\big) v\Big(\frac{1}{\lambda^4}S_5D_6S_5\Big)vR^+_0(\lambda^4)\\
&+R^-_0(\lambda^4)v\Big(\frac{1}{\lambda^4}S_5D_6S_5 \Big)v
\big(R^+_0(\lambda^4)-R_0^-(\lambda^4)\big),
\end{split}
\end{equation}
and the following orthogonal relations of $S_5$:
$$S_5v(x)=0, S_5(x_i v)=S_5(x_ix_j v)=S_5(x_ix_jx_kv)=0, i,j,k=1,2,$$
hence by the same arguments as the case with $\Omega^\pm_{3,1}(\lambda)$ above,  we get that the desired estimates hold.
\end{proof}

\begin{proposition}\label{prop-third-errorterm}
Assume that $|V(x)|\lesssim (1+|x|)^{-\beta}$ with $\beta >18$. Let $\Gamma^3(\lambda)= O_1(\lambda^{-4}\big(\ln\lambda)^{-3}\big)$ be the error term of the third or fourth kind resonance in the expansions of $\big(M^\pm(\lambda)\big)^{-1}$.
Then for $N\in \mathbb{Z}$ and $N \leq N'$,
\begin{equation*}
   \begin{split}
&\sup\limits_{x,y}\Big|\int_0^\infty e^{-it\lambda^4}\lambda^{3+4\alpha}\varphi_0(2^{-N}\lambda)
\big[R^\pm_0(\lambda^4)v\Gamma^3(\lambda)vR^\pm_0(\lambda^4)\big](x,y)d\lambda\Big| \lesssim
 2^{(-4+4\alpha) N}(1+|t|2^{4N})^{-1/2}.
\end{split}
\end{equation*}
In particular, for $\alpha=\frac{3}{2}$,
\begin{equation*}
   \begin{split}
&\sup\limits_{x,y}\Big|\int_0^\infty e^{-it\lambda^4}\lambda^{3+4\alpha}
\big[R^\pm_0(\lambda^4)v\Gamma^3(\lambda)vR^\pm_0(\lambda^4)\big](x,y)d\lambda\Big| \lesssim
 |t|^{-1/2}.
\end{split}
\end{equation*}
\end{proposition}
\begin{proof}
By the same arguments as in the proof of Proposition \ref{prop-reg-freeterms}, we can prove that this proposition holds.
\end{proof}

\subsection{Large energy decay estimates }
In this subsection, we will show Theorem \ref{thm-high-1}. To complete the proof, we use the formula,
\begin{equation}\label{highenergy-formula}
  e^{-itH}\widetilde{\chi}(H)f(x) =\frac{2}{\pi i}
\int_{\mathbb{R}^2}\Big(\int_0^\infty e^{-it\lambda^4}\lambda^3\widetilde{\chi}(\lambda)
\big[R^+_V(\lambda^4)-R^-_V(\lambda)\big](x,y)d\lambda \Big)f(y)dy,
\end{equation}
and the resolvent identity,
\begin{equation}\label{highenergy-reso}
   \begin{split}
R^\pm_V(\lambda^4)=R^\pm_0(\lambda^4)-R^\pm_0(\lambda^4)VR^\pm_0(\lambda^4)
+R^\pm_0(\lambda^4)VR^\pm_V(\lambda^4)VR^\pm_0(\lambda^4).
  \end{split}
\end{equation}
Combining with Proposition \ref{prop-free estimates}, it suffices  to prove
the following Propositions \ref{largeenergy-firstterm} and \ref{largeenergy-secondterm}.
\begin{proposition}\label{largeenergy-firstterm}
 Assume that $|V(x)|\lesssim (1+|x|)^{-\beta}$ with $\beta>3$. Then
\begin{equation}\label{prop-high-1}
\Big\|\int_0^\infty e^{-it\lambda^4}\lambda^3\widetilde{\chi}(\lambda)
R^\pm_0(\lambda^4)VR^\pm_0(\lambda^4)d\lambda \Big\|_{L^1\rightarrow L^\infty}
\lesssim |t|^{-\frac{1}{2}}.
\end{equation}
\end{proposition}
\begin{proof}
Set
\begin{equation*}
   \begin{split}
L^{\pm}_{1,N}(t;x,y)=\int_0^\infty e^{-it\lambda^4}\lambda^3\varphi_0(2^{-N}\lambda)
\big[ R^\pm_0(\lambda^4)VR^\pm_0(\lambda^4) \big](x,y)d\lambda.
\end{split}
\end{equation*}
Then
$$\int_0^\infty e^{-it\lambda^4}\lambda^3\widetilde{\chi}(\lambda)
\big[ R^\pm_0(\lambda^4)VR^\pm_0(\lambda^4) \big](x,y)d\lambda
   = \sum_{N=N'+1}^{+\infty}L^{\pm}_{1,N}(t;x,y).$$
Let $R_0^\pm(\lambda^4)(x,y):= \frac{e^{\pm i\lambda |x-y|}}{\lambda^2}\widetilde{R}^\pm(\lambda|x-y|)$ and
$\lambda =2^Ns$. Then
\begin{equation*}
   \begin{split}
L^{\pm}_{1,N}(t;x,y)
= &\int_{\mathbb{R}^2} \int_0^\infty e^{-it\lambda^4}\lambda^3\varphi_0(2^{-N}\lambda)
 R^\pm_0(\lambda^4)(x,u_1)V(u_1)R^\pm_0(\lambda^4)(u_1,y)d\lambda du_1\\
=& \int_0^\infty \int_{\mathbb{R}^2} e^{-it2^{4N}s^4}e^{\pm i2^Ns(|x|+|y|)} s^{-1}\varphi_0(s)
\Big( e^{\pm i2^Ns(|x-u_1|-|x|)}e^{\pm i2^Ns(|y-u_1|-|y|)}\\
& \ \ \ \ \ \ \ \widetilde{R}^\pm(2^Ns|x-u_1|)V(u_1)\widetilde{R}^\pm(2^Ns|y-u_1|)\Big) du_1 ds.
\end{split}
\end{equation*}
Let $z=(x,y)$, $\Phi(z)= |x|+|y|$ and
$$ \Psi(2^Ns, z)=\int_{\mathbb{R}^2}e^{\pm i2^Ns(|x-u_1|-|x|)}e^{\pm i2^Ns(|y-u_1|-|y|)}
\widetilde{R}^\pm(2^Ns|x-u_1|)V(u_1)\widetilde{R}^\pm(2^Ns|y-u_1|)du_1.$$
Then
\begin{equation*}
   \begin{split}
L^{\pm}_{1,N}(t, z)
=  \int_0^\infty e^{-it2^{4N}s^4}e^{\pm i2^Ns\Phi(z)} s^{-1}\varphi_0(s)
 \Psi(2^Ns,z)ds.
\end{split}
\end{equation*}
By (\ref{freeres-lagrepart-Rtubap}) and $s\in \hbox{supp} \varphi_0 \subset [\frac{1}{4}, 1]$,
one can check that for $k=0,1$,
$$\big|\partial_s^ke^{\pm i2^Ns(|x-u_1|-|x|)}\big| \lesssim (2^N|u_1|)^k,\,\,\,\,
\big|\partial_s^ke^{\pm i2^Ns(|y-u_1|-|y|)}\big| \lesssim (2^N|u_1|)^k,$$
$$\Big|\partial_s^k\Big(\widetilde{R}^\pm(2^Ns|x-u_1|)V(u_1)\widetilde{R}^\pm(2^Ns|y-u_1|)\Big)\Big|
\lesssim |V(u_1)|.$$
Hence for any $z$, we have
$$\partial_s^k|\Psi(2^Ns, z)|\lesssim 2^{kN}\int_{\mathbb{R}^2}(1+|u_1|)^k|V(u_1)|du_1\lesssim 2^{kN}, \ \ k=0,1.$$
Let $N_0=\big[\frac{1}{3}\log_2\frac{|x|+|y|}{|t|}\big]$ and $N>N'$, Then by using Lemma \ref{lem-LWP}, we immediately obtain that $|L^{\pm}_{1,N}(t;x,y)| $
is bounded by $2^{N}(1+|t| 2^{4N})^{-\frac{1}{2}}$ if $|N- N_0| \leq 2$, and bounded by
$2^{N}(1+|t|  2^{4N})^{-1}$ if $|N- N_0| > 2$.

Therefore it immediately follows that for  any $x,y$
\begin{equation*}
   \begin{split}
 |L^{\pm}_1(t;x,y)|\le & \sum_{N=N'+1}^{+\infty}|L^{\pm}_{1,N}(t;x,y)|\\
 \lesssim &\sum_{|N-N_0|\leq 2,N>N'} 2^{N}( 1+|t|2^{4N})^{-\frac{1}{2}}
+\sum_{|N-N_0|>2,N>N'} 2^{N}( 1+|t|2^{4N})^{-1}\\
\lesssim &\sum_{|N-N_0|\leq 2}|t|^{-\frac{1}{2}}+ \sum_{N=-\infty}^{+\infty}2^{2 N}( 1+|t|2^{4N})^{-1}\\
\lesssim &|t|^{-\frac{1}{2}},
  \end{split}
\end{equation*}
which implies the conclusion (\ref{prop-high-1}).
\end{proof}

Before we deal with the term $ R^\pm_0(\lambda^4)VR^\pm_V(\lambda^4)VR^\pm_0(\lambda^4)$, we state a lemma
as follows, see \cite{FSY}.
\begin{lemma}\label{lem-largeenergy}
Let $k\geq 0$ and $|V(x)| \lesssim (1+|x|)^{-k-1-}$ such that
$H=\Delta^2 +V$ has no embedded positive eigenvalues. Then for any $\sigma > k+\frac{1}{2}$,
 $R^\pm_V(\lambda)\in \mathcal{B}\big( L^2_\sigma(\mathbb{R}^d), L^2_{-\sigma}(\mathbb{R}^d)\big)$
are $C^k$-continuous for all $\lambda>0$. Furthermore,
$$ \big\|\partial_\lambda R^\pm_V(\lambda)\big\|_{L^2_\sigma(\mathbb{R}^2)\rightarrow  L^2_{-\sigma}(\mathbb{R}^2)} = O\big(|\lambda|^{\frac{-3(k+1)}{4}}\big), k=0,1,\  \hbox{as}\ \lambda \rightarrow +\infty. $$
\end{lemma}
\begin{proposition}\label{largeenergy-secondterm}
 Assume that $|V(x)|\lesssim (1+|x|)^{-\beta}$ with $\beta>5$. Then
\begin{equation}\label{prop-high-2}
\Big\|\int_0^\infty e^{-it\lambda^4}\lambda^3\widetilde{\chi}(\lambda)
R^\pm_0(\lambda^4)VR^\pm_V(\lambda^4)VR^\pm_0(\lambda^4)
d\lambda \Big\|_{L^1\rightarrow L^\infty}\lesssim |t|^{-\frac{1}{2}}.
\end{equation}
\end{proposition}
\begin{proof}
The proof is similar to  Proposition \ref{prop-reg-freeterms}. To prove (\ref{prop-high-2}),
it suffices to prove that for any $f, g \in L^1$,
\begin{equation*}
   \begin{split}
\Big|\int_0^\infty e^{-it\lambda^4}\lambda^3\widetilde{\chi}(\lambda)
\big\langle V R^\pm_V(\lambda^4)V(R_0^\pm(\lambda^4)f, ~(R_0^\pm)^*(\lambda^4)g  \big\rangle
d\lambda\Big|
  \end{split}
\end{equation*}
is bounded $|t|^{-\frac{1}{2}}\|f\|_{L^1}\|g\|_{L^1}$.
By Littlewood-Paley decomposition as used in Proposition \ref{largeenergy-firstterm} above, it suffices to show for each $N>N'$,
\begin{equation*}
   \begin{split}
L^\pm_{2,N}(t;x,y):= \int_0^\infty e^{-it\lambda^4}\lambda^3\varphi_0(2^{-N}\lambda)
\big\langle V R^\pm_V(\lambda^4)V\big(R_0^\pm(\lambda^4)(*,y)\big)(\cdot),
~(R_0^\pm\lambda^4))^*(x,\cdot) \big\rangle d\lambda
  \end{split}
\end{equation*}
is bounded  by $2^{2N}(1+|t|2^{4N})^{-1}$ if $|N-N_0|>2$, and
bounded by $2^{2N}(1+|t|2^{4N})^{-1/2}$  if $|N-N_0|\leq2$,
where $N_0=\big[\frac{1}{3}\log_2\frac{|x|+|y|}{|t|}\big]$.

Let $R_0^\pm(\lambda^4)(x,y)= \frac{ e^{\pm i\lambda|x-y|}}{\lambda^2} \widetilde{R}^\pm(\lambda|x-y|)$.
Then
\begin{equation*}
   \begin{split}
&\Big\langle  V R^\pm_V(\lambda^4)V(R_0^\pm(\lambda^4)(*,y))(\cdot),~ R_0^\mp(\lambda)(x,\cdot)   \Big\rangle\\
=&\frac{1}{\lambda^4}\Big\langle V R^\pm_V(\lambda^4)V \big(e^{\pm i\lambda|*-y|}\widetilde{R}^\pm(\lambda|*-y|)\big)(\cdot),~
\big(e^{\mp i\lambda|x-\cdot|}\widetilde{R}^\mp(\lambda|x-\cdot|)\big)   \Big\rangle\\
=&\frac{1}{\lambda^4}e^{\pm i\lambda(|x|+|y|)}\Big\langle V R^\pm_V(\lambda^4)V \big(e^{\pm i\lambda(|*-y|-|y|)}\widetilde{R}^\pm(\lambda|*-y|)\big)(\cdot),~
\big(e^{\mp i\lambda(|x-\cdot|-|x|)}\widetilde{R}^\mp(\lambda|x-\cdot|)\big)   \Big\rangle\\
:=&\frac{1}{\lambda^4}e^{\pm i\lambda(|x|+|y|)}E^{\pm}_{2}(\lambda;x,y).
  \end{split}
\end{equation*}
Let $\lambda =2^Ns$, then
\begin{equation*}
   \begin{split}
L^{\pm}_{2,N}(t;x,y)= \int_0^\infty e^{-it2^{4N}s^4}s^{-1}\varphi_0(s)e^{\pm i2^Ns(|x|+|y|)}
E^{\pm}_{2}(2^Ns;x,y) ds.
  \end{split}
\end{equation*}
By Lemma \ref{lem-largeenergy}, for $\sigma>k+\frac{1}{2}$ we have
$$  \big\|\partial_s^k \big[R^\pm_V(2^{4N}s^4)\big]\big\|_{L^2_\sigma\rightarrow L^2_{-\sigma}} \lesssim 2^{kN}\big(2^Ns\big)^{-3},\ k=0,1.
$$
Moreover, by (\ref{freeres-lagrepart-Rtubap}) we have for $k=0,1 $,
$$ \big|\partial_s^k\big(e^{\pm i\lambda(|\cdot-y|-|y|)}\widetilde{R}^\pm(\lambda|\cdot-y|)\big)\big|
 \lesssim \langle \cdot \rangle^k,\
 \big|\partial_s^k\big(e^{\pm i\lambda(|x-\cdot|-|x|)}\widetilde{R}^\pm(\lambda|x-\cdot|)\big)\big|
 \lesssim \langle \cdot \rangle^k.$$
Hence it follows that for $\sigma >k+\frac{1}{2}$ and $ k=0,1$,
$$   \big\|\partial_s^kE^{\pm}_{2}(2^Ns;x,y)\big\| \lesssim
\sum_{j=0}^k\big\|V(\cdot)\langle \cdot\rangle^{\sigma+1-j }\big\|^2_{L^2}
\big\|\partial_s^j R^\pm_V(2^{4N}s^4)\big\|_{L^2_\sigma\rightarrow  L^2_{-\sigma}} \lesssim 2^{kN}\big(2^Ns\big)^{-3},
$$
where  we use  $|V(x)| \lesssim (1+|x|)^{-5-}$.

Let $N_0=\big[\frac{1}{3}\log_2\frac{|x|+|y|}{|t|}\big]$ and $N>N'$. Using Lemma \ref{lem-LWP} with $ z=(x,y)$,
 $ \Phi(z)=|x|+|y|$ and  $\Psi(2^Ns;z)= E^{0,\pm}_{3}(2^Ns,x,y),$
we immediately obtain that when $|N-N_0| >2$,
$$ |L^\pm_{2,N}(t;x,y)| \lesssim 2^{(k-3)N}(1+|t|2^{4N})^{-1} \lesssim
2^{2N}(1+|t|2^{4N})^{-1},$$
and $|L^\pm_{2,N}(t;x,y)|$ is bounded by $2^{2N}(1+|t|2^{4N})^{-1/2}$ when $|N-N_0|\leq 2$.

Hence we have for $x,y$,
\begin{equation*}
   \begin{split}
 &|L^{\pm}_2(t;x,y)|\leq \sum_{N=N'+1}^{+\infty}|L^{\pm}_{2,N}(t;x,y)|\\ &\lesssim\sum_{|N-N_0|\leq 2} 2^{2 N}( 1+|t|2^{4N})^{-\frac{1}{2}}
+\sum_{|N-N_0|>2} 2^{2 N}( 1+|t|2^{4N})^{-1}
\lesssim |t|^{-\frac{1}{2}},
  \end{split}
\end{equation*}
which implies that the estimate (\ref{prop-high-2}) holds.
\end{proof}

\bigskip

\section{The proof of Theorem \ref{thm-main-inver-M}} \label{the proof of inverse Pro}
In this section, we are  devoted to showing Theorem \ref{thm-main-inver-M}, i.e. computing the expansions
of $\big(M^\pm(\lambda)\big)^{-1}$  for $\lambda$ near zero case by case, the proof is quite complicated due to logarithm terms in the expansions of the free resolvent $R^\pm_0(\lambda^4).$

Before computing the expansions of $\big(M^\pm(\lambda)\big)^{-1}$ as $\lambda\rightarrow 0$,
we first state the following lemmas which are used frequently.
\begin{lemma}\label{lemma-JN}(\cite[Lemma 2.1]{JN})
Let $A$ be a closed operator and $S$ be a projection. Suppose $A+S$ has a bounded inverse. Then $A$ has
a bounded inverse if and only if
\begin{equation}		
a:= S-S(A+S)^{-1}S
\end{equation}
has a bounded inverse in $SH$, and in this case
\begin{equation}		
A^{-1}= (A+S)^{-1} + (A+S)^{-1}S a^{-1} S(A+S)^{-1}.
\end{equation}	
\end{lemma}

\begin{lemma}\label{lemma-JN-matrix}(\cite[Lemma 2.3]{JN})
Let $A$ be an operator matrix on $ \mathcal{H}=\mathcal{H}_1  \bigoplus\mathcal{ H}_2$ :
$$A=
\begin{pmatrix}
a_{11}&a_{12}\\ a_{21}&a_{22}
 \end{pmatrix}
, \quad a_{ij}: \mathcal{H}_j \rightarrow \mathcal{H}_i, 1 \leq i, j \leq 2,
 $$
where  $a_{11}, a_{22} $ are closed and $ a_{12}, a_{21}$ are bounded. Suppose $a_{22}$ has a bounded inverse. Then $A$ has a bounded inverse if and only if
$ d \equiv (a_{11} -a_{12}a_{22}^{-1}a_{21})^{-1}$ exist and is bounded. Furthermore, we have
$$A^{-1}=
\begin{pmatrix}
d & -d a_{12}a_{22}^{-1}\\ -a_{22}^{-1}a_{21}d & a_{22}^{-1}a_{21}d a_{12}a_{22}^{-1}+a_{22}^{-1}
 \end{pmatrix}
. $$
\end{lemma}

Now we turn to the proof of Theorem \ref{thm-main-inver-M} case by case, i.e. compute the asymptotic expansions of $\big(M^\pm(\lambda)\big)^{-1}$ for
$\lambda\rightarrow 0$.

Let $ \displaystyle M^\pm(\lambda)= \frac{a_\pm}{\lambda^2} \widetilde{M}^\pm(\lambda)$. Then we just need to compute  the asymptotic expansion of  $\big(\widetilde{M}^\pm(\lambda)\big)^{-1}$
for
$\lambda\rightarrow 0$. Recall that from the (iii) of Lemma \ref{lem-M}
\begin{equation}\label{id-M-tuta}
   \begin{split}
\widetilde{M}^\pm(\lambda)
 = & P+\frac{g_0^\pm(\lambda)}{a_\pm}\lambda^2 vG_{-1}v
+ \frac{1}{a_\pm} \lambda^2 T_0
+\frac{c_\pm}{a_\pm}\lambda^4 v G_1 v\\
  & +\frac{ \tilde{g}_2^\pm(\lambda)}{a_\pm}\lambda^6 v G_2v
   + \frac{\lambda^6}{a_\pm} v G_3v+O\big(\lambda^8\big).
   \end{split}
\end{equation}
Note that projection $P$ is not invertible on $L^2$, but let $Q=I-P$, then $\widetilde{M}^\pm(\lambda)+Q$ is invertible on $L^2$ for small $\lambda>0$. Hence by Lemma \ref{lemma-JN} we know that $\widetilde{M}^\pm(\lambda)$ is invertible on $L^2$ if and only if
 $$M_1^\pm(\lambda):= Q-Q\big( \widetilde{M}^\pm(\lambda)+Q\big)^{-1}Q,$$
is invertible on $QL^2$. In order to get the invertibility of $M_1^\pm(\lambda)$, we first need to compute $\big( \widetilde{M}^\pm(\lambda)+Q\big)^{-1}$ on $L^2$,
which can be given by Neumann series
\begin{equation}\label{id-MQ}
   \begin{split}
&\big(\widetilde{M}^\pm(\lambda)+Q\big)^{-1}\\
  =& I- g^\pm_0(\lambda)\lambda^2 B_1^\pm-\lambda^2 B_2^\pm-g^\pm_0(\lambda)^2\lambda^4 B_3^\pm
       -g^\pm_0(\lambda)\lambda^4 B_4^\pm-\lambda^4 B_5^\pm
  -g_0^\pm(\lambda)^3\lambda^6 B_6^\pm\\
  &-g_0^\pm(\lambda)^2\lambda^6 B_7^\pm
   - g^\pm_0(\lambda)\lambda^6 B_8^\pm - \tilde{g}^\pm_2(\lambda)\lambda^6 B_9^\pm
   -\lambda^6 B_{10}^\pm  +O\big(\lambda^{8-\epsilon} \big),
   \end{split}
\end{equation}
where $B_j^\pm(1\leq j\leq 10)$  are bounded operators in $L^2$ as follows:\\
$ \displaystyle B_1^\pm= \frac{1}{a_\pm} vG_{-1}v$, \,
$\displaystyle  B_2^\pm = \frac{1}{a_\pm}T_0$,\,
$\displaystyle  B_3^\pm= -\frac{1}{a_\pm^2} (vG_{-1}v)^2$,\,
$\displaystyle B_4^\pm= -\frac{1}{a_\pm^2}(vG_{-1}vT_0 +T_0vG_{-1}v ) $,\\
$\displaystyle B_5^\pm= \frac{c_\pm}{a_\pm}vG_1v - \frac{1}{a_\pm^2}T_0^2  $ ,\,
$ \displaystyle B_6^\pm = \frac{1}{a_\pm^3}(vG_{-1}v)^3$,\,
$\displaystyle B_7^\pm = \frac{1}{a_\pm^3} \big[ (vG_{-1}v)^2 T_0 + T_0(vG_{-1}v)^2\big]$,\\
$\displaystyle B_8^\pm= \frac{1}{a_\pm^3} ( vG_{-1}v T_0^2 + T_0^2 vG_{-1}v)
   - \frac{c_\pm}{a_\pm^2}( vG_{-1}v\cdot vG_1v + vG_1v \cdot  vG_{-1}v ) $, \\
$ \displaystyle B_9^\pm=\frac{1}{a_\pm}vG_2v $,\,
$\displaystyle B_{10}^\pm
      = \frac{1}{a_\pm}vG_3v  + \frac{1}{a_\pm^3}T_0^3 - \frac{c_\pm}{a_\pm^2}(T_0 vG_1v +vG_1vT_0 )$.\\

Before we study the invertibility of $M_1^\pm(\lambda)$, let's list the following orthogonality relations of various operators and projections from Definition \ref{definition of resonance}, which are
used frequently later.
\begin{equation}\label{orthog-relation-1}
   \begin{split}
QD_0&=D_0Q=D_0;\ \ Q\ge S_i\ge S_j, \ 0\le i\le j\le 5;\\
 S_iD_j& = D_jS_i =S_i, i\geq j,\  S_iD_j = D_jS_i =D_i, i < j, \ D_iD_j = D_jD_i =D_i, i >  j;
   \end{split}
\end{equation}
\begin{equation}\label{orthog-relation-1}
   \begin{split}
&QvG_{-1}vS_0= S_0QvG_{-1}vS_0=0, \ S_1T_0S_0=S_0T_0S_1=0, \ S_2T_0Q=QT_0S_2=0;\\
&S_3vG_{-1}v=vG_{-1}vS_3=0,\ S_4T_0P=PT_0S_4=0,\ S_4T_0=T_0S_4=0;\\
&S_4vG_1vS_0=S_0vG_1vS_4=0,\ S_5vG_1vQ=QvG_1vS_5=0.
\end{split}
\end{equation}
Moreover, we can obtain
\begin{equation}\label{or-Bjpm}
   \begin{split}
QB_1^\pm S_0=S_0B_1^\pm Q = S_3B_1^\pm= B_1^\pm S_3 = S_3B_3^\pm= B_3^\pm S_3 = S_3B_4^\pm S_3= 0,
\, S_4B_8^\pm S_4=0 ;\\
S_4B_i^\pm = B_i^\pm S_4 =0 (i=1,2,3,4,6,7),
\ S_4B_5^\pm S_0 =S_0B^\pm_5S_4= S_5B_5^\pm Q =QB^\pm_5S_5= 0.
\end{split}
\end{equation}

Now we turn to the proof of the invertibility of $M_1^\pm(\lambda)$,
We only prove the assertion  for the case $+$ sign,
the case $-$ sign  proceeds identically. By Lemma \ref{lemma-JN}, we know that $\widetilde{M}^+(\lambda)$ is invertible on $L^2$ if and only if
 $$M_1^+(\lambda):= Q-Q\big( \widetilde{M}^+(\lambda)+Q\big)^{-1}Q $$
is invertible on $QL^2$. Moreover, by (\ref{id-MQ}) we get
\begin{equation}\label{id-M1M1tabu}
   \begin{split}
M_1^+(\lambda)
     :=\frac{\lambda^2}{a_+}\widetilde{M}_1^+(\lambda)
     =\frac{\lambda^2}{a_+}\Big(A^+(\lambda)+a_+W^+(\lambda)\Big),
   \end{split}
\end{equation}
where $A^+(\lambda)=g^+_0(\lambda) QvG_{-1}vQ+  QT_0Q $, and
\begin{equation}\label{eq-W}
   \begin{split}
W^+(\lambda)&=g^+_0(\lambda)^2\lambda^2Q B_3^+Q
          +g^+_0(\lambda)\lambda^2 QB_4^+Q
          +\lambda^2 QB_5^+Q+g_0^+(\lambda)^3\lambda^4 QB_6^+Q \\
          &+g_0^+(\lambda)^2\lambda^4 QB_7^+Q
   +g^+_0(\lambda)\lambda^4 QB_8^+Q
   +\tilde{g}^+_2(\lambda)\lambda^4 QB_9^+Q
     +\lambda^4Q B_{10}^+Q +O\big(\lambda^{6-\epsilon}\big).
 \end{split}
\end{equation}

In order to compute $\big(M_1^+(\lambda)\big)^{-1}$, by (\ref{id-M1M1tabu} ) we need to obtain the expansions of $\big(\widetilde{M_1^+}(\lambda)\big)^{-1}$ on $QL^2$ when $\lambda$ is near zero. Note that $A^+(\lambda)=g^+_0(\lambda) QvG_{-1}vQ+  QT_0Q $ has the logarithm factor $g^+_0(\lambda)$ which becomes $\infty$ as $\lambda\downarrow 0$, and
$QvG_{-1}vQ$ is not invertible on $QL^2$ by Lemma \ref{projiction-spaces-SjL2}(i), hence we will use Lemma \ref{lemma-JN-matrix} to consider the inverse of $A^+(\lambda)$.

To prove the invertibility of $A^+(\lambda)$ by Lemma \ref{lemma-JN-matrix}, we first give a lemma as follows.
\begin{lemma}\label{lem-d-inver}
Suppose that $|V(x)| \lesssim (1+|x|)^{-\beta}$ with some $\beta >10$. Then the inverse operator
\begin{equation}\label{def-d}
   \begin{split}
 d:= \Big( g^+_0(\lambda)(Q-S_0)vG_{-1}v(Q-S_0) + (Q-S_0)\big(T_0- T_0D_1T_0\big)(Q-S_0) \Big)^{-1}
   \end{split}
\end{equation}
 exists on  $(Q-S_0)L^2(\mathbb{R}^2) $ for $0<\lambda \ll 1$. Moreover, we have
\begin{align*}
 d = g^+_0(\lambda)^{-1} (Q-S_0)F(\lambda)(Q-S_0),
\end{align*}
where $F(\lambda)= \big(A+g^+_0(\lambda)^{-1}B\big)^{-1}$ with
$ A= (Q-S_0)vG_{-1}v(Q-S_0) $ and $ B= (Q-S_0)\big(T_0- T_0D_1T_0\big)(Q-S_0) $.

In particular,  $A$ and $B$  are  bounded operators on $(Q-S_0)L^2$ with
$\hbox{rank} \ A=2$ and $\hbox{rank} \ B\le2$, and  $ \|F(\lambda)\|_{(Q-S_0)L^2}  = O_1(1)$
for small enough $\lambda >0$.
\end{lemma}
\begin{proof}
Since $ QL^2(\mathbb{R}^2)=  \hbox{span} \{v \}^\bot$ and $S_0L^2(\mathbb{R}^2)= \hbox{span} \{v, x_1v(x), x_2v(x)\}^\bot$, then
$$ (Q-S_0)L^2(\mathbb{R}^2)= \hbox{span} \{Q(x_1v), Q(x_2v)\}.$$
Note that $\hbox{rank}(Q-S_0)=2$, and
\begin{align*}
 d=& g^+_0(\lambda)^{-1} (Q-S_0)\Big[(Q-S_0) vG_{-1}v (Q-S_0)\\
  &+ g^+_0(\lambda)^{-1}(Q-S_0)\big(T_0-T_0D_1T_0\big)(Q-S_0)\Big]^{-1}(Q-S_0).
\end{align*}
Hence, in order to prove the operator $d$ exists, it suffices to prove that  $ (Q-S_0) vG_{-1}v (Q-S_0)$
is invertible on $(Q-S_0)L^2$ by Neumamm series if $\lambda$ is small enough.

Without loss of generality, we may assume that $Qx_1v$ and $Qx_2v$ are orthogonal on $(Q-S_0)L^2$. Otherwise,
we can do the same arguments after Schmidt orthogonalization.
 If $f\in (Q-S_0)L^2$, then $f= a_1Q(x_1v)+a_2Q(x_2v)$.
Since  $(Q-S_0)(Q(x_iv))= Q(x_iv)(i=1,2)$, then  we have
 \begin{align*}
&(Q-S_0)vG_{-1}v(Q-S_0)(Q(x_1v))(x)= (Q-S_0)vG_{-1}v(Q(x_1v))(x) \\
= &(Q-S_0)v(x)\int_{\mathbb{R}^2}|x-y|^2v(y)Q(x_1v)(y)dy\\
=&  -2 (Q-S_0)(x_1v)\langle x_1v, Qx_1v \rangle
- 2 (Q-S_0)(x_2 v) \langle x_2v, Qx_1v \rangle.
\end{align*}
Since $\langle x_2v, Qx_1v \rangle = \langle Q(x_2v), Q(x_1v) \rangle =0$
 and $(Q-S_0)(x_iv)=  Q(x_iv)(i=1,2)$, then we get
$$(Q-S_0)vG_{-1}v(Q-S_0)(Qx_1v)=  -2 Q(x_1v) \|Qx_1v\|_{L^2}^2.$$
Similarly,  we have
$$(Q-S_0)vG_{-1}v(Q-S_0)(Qx_2v) =  -2 Q(x_2v)\|Qx_2v\|_{L^2}^2,$$
which leads to
\begin{equation}\label{positive-matrix}
   \begin{split}
&(Q-S_0)vG_{-1}v(Q-S_0)f =-2a_1 Q(x_1v)\|Qx_1v\|_{L^2}^2 - 2 a_2Q(x_2 v) \|Qx_2v\|_{L^2}^2\\
=&\begin{pmatrix}f_1& f_2\end{pmatrix}
\begin{pmatrix}
 -2\|Q(x_1v)\|_{L^2}^2
  & 0\\
  0& -2\|Q(x_2v)\|_{L^2}^2
 \end{pmatrix}
 \begin{pmatrix}a_1\\a_2\end{pmatrix}
 := \begin{pmatrix}f_1 &f_2\end{pmatrix} A\begin{pmatrix}a_1 \\a_2\end{pmatrix}.
\end{split}
\end{equation}
Since $\det A = 4 \|Q(x_1v)\|_{L^2}^2\|Q(x_2v)\|_{L^2}^2\geq 0$, so we need only to prove $\det A \neq 0$. Indeed, if  $\det A = 0$, then $ Q(x_1v)=0$ or $Q(x_2v)=0$, it is contradictory with $rank (Q-S_0)=2$.
Thus, $ (Q-S_0)vG_{-1}v(Q-S_0)$ is invertible and strictly negative on the space $(Q-S_0)L^2$,  and then $ d$ exists on $(Q-S_0)L^2$ for small enough $\lambda$.

Let  $A= (Q-S_0)vG_{-1}v(Q-S_0) $, $ B= (Q-S_0)\big(T_0- T_0D_1T_0\big)(Q-S_0) $ and
$F(\lambda)= \big[A+g^+_0(\lambda)^{-1}B\big]^{-1}$, we obtain that
\begin{align*}
 d = g^+_0(\lambda)^{-1} (Q-S_0)F(\lambda)(Q-S_0).
\end{align*}

Finally, we need to prove that $\|F(\lambda)\|_{(Q-S_0)L^2}= O_1(1)$ for small enough $\lambda$. Indeed,
let $G(\lambda)=A+g_0^+(\lambda)^{-1}B$, then $F(\lambda)=G(\lambda)^{-1}$.
 Note that
$ G'(\lambda)= - \lambda^{-1} g_0^+(\lambda)^{-2}$ , we have
$$\frac{d}{d\lambda}F(\lambda)= \frac{d}{d\lambda}\big(G(\lambda)^{-1}\big)
=-\Big(A+g_0^+(\lambda)^{-1}B\Big)^{-1}\cdot \lambda^{-1} g_0^+(\lambda)^{-2}\cdot \Big(A+g_0^+(\lambda)^{-1}B\Big)^{-1}. $$
Observe that
$$ \|F(\lambda)\|_{(Q-S_0)L^2\rightarrow (Q-S_0)L^2} \leq C, \  \lambda \ll 1 ,$$
then
$$\|\frac{d}{d\lambda}F(\lambda)\|_{(Q-S_0)L^2\rightarrow (Q-S_0)L^2} \lesssim
 \lambda^{-1}|\ln\lambda|^{-2}\lesssim \lambda^{-1},\  \lambda \ll 1, $$
Hence,  we have $\|F(\lambda)\|_{(Q-S_0)L^2\rightarrow (Q-S_0)L^2} = O_1(1)$ for small enough $\lambda$.
\end{proof}

Now we begin to compute the inverse $\big(A^+(\lambda)\big)^{-1}$ as follows.
\begin{lemma}\label{lem-M1-tuba}
Suppose that $|V(x)| \lesssim (1+|x|)^{-\beta}$ with some $\beta >10$.
If zero is a regular point of the spectrum of $H$, then  for $ 0<|\lambda|\ll 1$,
\begin{equation}\label{A+lambdainver}
   \begin{split}
\big(A^+(\lambda)\big)^{-1}=S_0D_1S_0+ g_0^+(\lambda)^{-1}S(\lambda),
\end{split}
\end{equation}
where
\begin{equation*}
   \begin{split}
S:=S(\lambda)=
\begin{pmatrix}
 (Q-S_0)F(\lambda)(Q-S_0)
  &  -(Q-S_0)F(\lambda)(Q-S_0)T_0D_1 \\
  -D_1T_0(Q-S_0)F(\lambda)(Q-S_0)
  &  D_1T_0(Q-S_0)F(\lambda)(Q-S_0)T_0D_1
 \end{pmatrix}
 \end{split}
\end{equation*}
with $F(\lambda)= \big(A+g^+_0(\lambda)^{-1}B\big)^{-1}$ and $D_1= (S_0T_0S_0)^{-1}$ on $S_0L^2$.
Moreover, $ \|S(\lambda)\|_{L^2 \rightarrow L^2}= O_1(1)$ for enough small $\lambda$ .

In particular, if zero is not regular point, then the same formula (\ref{A+lambdainver})
holds for $\big(A^+(\lambda)+S_1\big)^{-1}$ with $D_1= (S_0T_0S_0+ S_1)^{-1}$.
\end{lemma}
\begin{proof}
We prove only the statement when $S_1 \neq 0$ ( i.e. $S_1L^2=Ker(S_0T_0S_0)\neq \{0\}$ ) since the proof of the regular case is similar by setting $S_1=0$.
Observe that
\begin{equation*}
   \begin{split}
A^+(\lambda)+S_1
 =&  g^+_0(\lambda)(Q-S_0)vG_{-1}v(Q-S_0) + (Q-S_0)T_0(Q-S_0) +
  (Q-S_0)T_0S_0 \\
  &+ S_0T_0(Q-S_0) +S_0T_0S_0 +S_1.
 \end{split}
\end{equation*}
Equivalently, by block format, we write $A^+(\lambda)+S_1$ as follows:
\begin{equation*}
   \begin{split}
A^+(\lambda)+S_1  =&
\begin{pmatrix}
 g^+_0(\lambda)(Q-S_0)vG_{-1}v(Q-S_0) + (Q-S_0)T_0(Q-S_0) & (Q-S_0)T_0S_0\\
  S_0T_0(Q-S_0) & S_0T_0S_0 +S_1
 \end{pmatrix}.
\end{split}
\end{equation*}
Note that $ QL^2 = (Q-S_0)L^2\bigoplus S_0L^2$,
 and $ S_0T_0S_0+S_1 $  is invertible on $S_0L^2$,  then by Lemma \ref{lemma-JN-matrix}
 $A^+(\lambda)+S_1$ has a bounded inverse on $S_0L^2$ if and only if the inverse operator
$ d $ defined in (\ref{def-d}) exists on $(Q-S_0)L^2$ and is bounded for enough small $\lambda$.

Let
$D_1= (S_0T_0S_0+S_1)^{-1}:  S_0 L^2 \rightarrow S_0 L^2$. Then
by Lemma \ref{lem-d-inver}, we know that the operator $ d $
exists on $(Q-S_0)L^2$ and
\begin{align*}
 d = g^+_0(\lambda)^{-1} (Q-S_0)F(\lambda)(Q-S_0).
\end{align*}
Hence by using Lemma \ref{lemma-JN-matrix} again, we have $A^+(\lambda)+S_1$ is invertible on $S_0L^2$, and
\begin{align*}
&\big( A^+(\lambda) +S_1\big)^{-1}\\
=& g^+_0(\lambda)^{-1}
\begin{pmatrix}
(Q-S_0)F(\lambda)(Q-S_0)
  &  -(Q-S_0)F(\lambda)(Q-S_0)T_0D_1 \\
  -D_1T_0(Q-S_0)F(\lambda)(Q-S_0)
  &  D_1T_0(Q-S_0)F(\lambda)(Q-S_0)T_0D_1 +g^+_0(\lambda)D_1
 \end{pmatrix}\\
 :=& S_0D_1S_0 +g^+_0(\lambda)^{-1}S(\lambda),
\end{align*}
where
$$
S(\lambda)=S=\begin{pmatrix}
 (Q-S_0)F(\lambda)(Q-S_0)
  &  -(Q-S_0)F(\lambda)(Q-S_0)T_0D_1 \\
  -D_1T_0(Q-S_0)F(\lambda)(Q-S_0)
  &  D_1T_0(Q-S_0)F(\lambda)(Q-S_0)T_0D_1
 \end{pmatrix}.
$$
It is obvious that $ \|S(\lambda)\|_{L^2 \rightarrow L^2}= O_1(1)$ for $\lambda \ll 1$ by Lemma \ref{lem-d-inver}.
\end{proof}
\begin{remark}\label{ortho-rela-SSS}
We can deduce some orthogonal relations from Lemma \ref{lem-M1-tuba} as follows:
 \begin{equation}\label{id-S1S}
    \begin{split}
S_1S(\lambda)=& -S_1T_0(Q-S_0)F(\lambda)(Q-S_0)
 + S_1T_0(Q-S_0)F(\lambda)(Q-S_0)T_0D_1,
 \end{split}
\end{equation}
 \begin{equation}\label{id-SS1}
    \begin{split}
S(\lambda) S_1= -(Q-S_0)F(\lambda)(Q-S_0)T_0 S_1
 + D_1T_0(Q-S_0)F(\lambda)(Q-S_0)T_0S_1,
 \end{split}
\end{equation}
\begin{equation}\label{id-S1SS1}
S_1SS_1= S_1T_0(Q-S_0)F(\lambda)(Q-S_0)T_0S_1,
\end{equation}
\begin{equation}
S_2T_0(Q-S_0)= (Q-S_0)T_0S_2=0,
\end{equation}
\begin{equation}
S_2S(\lambda)=S(\lambda)S_2=0, \quad  S_2\big(A^+(\lambda) +S_1 \big)^{-1}=\big(A^+(\lambda) +S_1 \big)^{-1}S_2=S_2.
\end{equation}
\end{remark}

Now we use Lemma \ref{lem-M1-tuba} to deduce the expansion of $\big(M^+(\lambda)\big)^{-1}$ in the regular case corresponding to the existence of $D_1=(S_0T_0S_0)^{-1}$ on $S_0L^2$ ( i.e. the projection $S_1=0$ ).

\textbf{(I) Regular point case.}
If zero is a regular point of the spectrum of $H$, then $S_0T_0S_0$ is invertible on $S_0L^2$.
Note that the \eqref{id-M1M1tabu} gives the identity
$$ \big(\widetilde{M}^+_1(\lambda)\big)^{-1}
= \big(A^+(\lambda)\big)^{-1}\Big(I +a_+W^+(\lambda)\big(A^+(\lambda)\big)^{-1} \Big)^{-1}.$$
By using Lemma \ref{lem-M1-tuba}, the \eqref{eq-W} and Neumann series,  for small enough $\lambda$, we have
\begin{equation*}
   \begin{split}
\big(\widetilde{M}^+_1(\lambda)\big)^{-1}
=& D_1+g_0^+(\lambda)^{-1}S- a_+g_0^+(\lambda)^2\lambda^2D_1B_3^+D_1\\
& -a_+g_0^+(\lambda)\lambda^2\big(D_1B_4^+D_1+D_1B_3^+QS+ SQB_3^+D_1 \big)+O_1(\lambda^2).
   \end{split}
\end{equation*}
Note that the equality (\ref{id-M1M1tabu}) also gives
$\big( M^+_1(\lambda)\big)^{-1}= \frac{a_+}{\lambda^2}\big(\widetilde{M}^+_1(\lambda)\big)^{-1}$.
Hence we have
\begin{align*}
  \big( M^+_1(\lambda)\big)^{-1} =&\frac{a_+}{\lambda^2}D_1 +\frac{a_+g_0^+(\lambda)^{-1}}{\lambda^2}S
      -a_+^2g_0^+(\lambda)^2 D_1B_3^+D_1 -a_+^2g_0^+(\lambda)\big(D_1B_4^+D_1\\
   &+D_1B_3^+QS+ SQB_3^+D_1 \big) +O_1(1).
 \end{align*}
Using the orthogonality $QD_1=D_1Q=D_1$,  by (\ref{id-MQ}) and Lemma \ref{lemma-JN}, we get
\begin{align*}
\big(\widetilde{M}^+(\lambda)\big)^{-1}
=&\frac{a_+}{\lambda^2}D_1 +\frac{a_+g_0^+(\lambda)^{-1}}{\lambda^2}QSQ
      -a_+^2g_0^+(\lambda)^2 D_1B_3^+D_1 \\
   &-a_+^2g_0^+(\lambda)\big(D_1B_4^+D_1+D_1B_3^+QSQ+ QSQB_3^+D_1 \big) +O_1(1).
\end{align*}
Since
$\displaystyle  \big(M^+(\lambda)\big)^{-1}=\frac{\lambda^2}{a_+}\big(\widetilde{M}^+(\lambda)\big)^{-1}$,
using the orthogonality $S_0D_1=D_1S_0=D_1$, we have
\begin{align}\label{Eq-regular}
\big(M^+(\lambda)\big)^{-1}
=&  S_0D_1S_0+ g_0^+(\lambda)^{-1}Q\Gamma^0_{0,1}Q
  + g_0^+(\lambda)^2\lambda^2S_0\Gamma^0_{2,1}S_0\nonumber\\
&+ g_0^+(\lambda)\lambda^2\Big(S_0\Gamma^0_{2,2}Q+Q\Gamma^0_{2,3}S_0\Big)
 +O_1(\lambda^2).
\end{align}
Thus  the expansions of $\big(M^+(\lambda)\big)^{-1}$ in regular case is completed.
\cqd
\vskip0.3cm

\begin{remark}
We know that $S_0D_1S_0$ is absolutely bounded by Proposition \ref{Pro-absulu-oper}.  It is clear that
$B^+_i(i=1,\cdots,9)$ are Hilbert-Schmidt operator by  (\ref{id-MQ}), and $S$ is  finite-rank operator by $\dim(Q-S_0)L^2=2$ and Lemma \ref{lem-d-inver},
thus $\Gamma^0_{i,j}(\lambda)$  are also Hilbert-Schmidt operator, which leads to each operator in the expansion of $\big(M^+(\lambda)\big)^{-1}$ above
are absolutely bounded uniformly for sufficiently small $\lambda>0$.
\end{remark}

Next, let's turn to consider the resonance cases, that is, zero is not regular point of the spectrum of $H$. Let $S_1$ be the Riesz projection onto the kernel of $S_0T_0S_0$. Then $S_1\neq 0$ and  $S_0T_0S_0+ S_1$ is invertible on $S_0L^2$.
In this case, we define $D_1= (S_0T_0S_0+S_1)^{-1}$ as a bounded operator on $S_0L^2$,
then from the (\ref{id-M1M1tabu}), Lemma \ref{lem-M1-tuba} and Neumann series one has
\begin{equation}\label{M1tuta+S1-inver}
  \begin{split}
&\Big(\widetilde{M}^+_1(\lambda)+S_1\Big)^{-1}\\
=&D_1 +g_0^+(\lambda)^{-1}S - g_0^+(\lambda)^2\lambda^2\widetilde{B}_1
 - g_0^+(\lambda)\lambda^2\widetilde{B}_2-\lambda^2\widetilde{B}_3
 - g_0^+(\lambda)^{-1}\lambda^2\widetilde{B}_4 \\
 &- g_0^+(\lambda)^{-2}\lambda^2\widetilde{B}_5
 -g_0^+(\lambda)^4\lambda^4\widetilde{B}_6 -g_0^+(\lambda)^3\lambda^4\widetilde{B}_7
 -g_0^+(\lambda)^2\lambda^4\widetilde{B}_8 -g_0^+(\lambda)\lambda^4\widetilde{B}_9\\
 &-\tilde{g}_2^+(\lambda)\lambda^4\widetilde{B}_{10}- \lambda^4\widetilde{B}_{11}
 -\tilde{g}_2^+(\lambda) g_0^+(\lambda)^{-1}\lambda^4\widetilde{B}_{12}
 -g_0^+(\lambda)^{-1}\lambda^4\widetilde{B}_{13}\\
 &-\tilde{g}_2^+(\lambda)g_0^+(\lambda)^{-2}\lambda^4\widetilde{B}_{14}
 -g_0^+(\lambda)^{-2}\lambda^4\widetilde{B}_{15} -g_0^+(\lambda)^{-3}\lambda^4\widetilde{B}_{16}
 +O(\lambda^{6-\epsilon}),
\end{split}
\end{equation}
where $\widetilde{B}_j(1\leq j\leq 15)$  are bounded operators on $L^2$ as follows:
\begin{align*}
&\widetilde{B}_1= a_+D_1 B_3 D_1, \widetilde{B}_2= a_+\big(D_1 B_4 D_1+D_1B_3QS +SQB_3D_1 \big),\\
&\widetilde{B}_3= a_+\big[D_1 B_5 D_1+D_1B_4QS +SQB_4D_1 + SQB_3QS \big],\\
&\widetilde{B}_4= a_+\big[D_1 B_5 QS +SQB_5D_1 + SQB_4QS \big],\,
\widetilde{B}_5= a_+SQB_5QS,\,
\widetilde{B}_6= -a_+^2D_1B_3D_1B_3D_1,
\end{align*}
\begin{equation*}
\begin{split}
\widetilde{B}_7=& a_+D_1B_6D_1- a_+^2\big[D_1B_3D_1B_4D_1 +D_1B_4D_1B_3D_1+D_1B_3D_1B_3QS+SQB_3D_1B_3D_1\\
 &+D_1B_3QSQB_3D_1 \big],\\
\widetilde{B}_8=&a_+\big[ D_1B_7D_1+D_1B_6QS +SQB_6D_1 \big]
   -a_+^2\big[ D_1B_4D_1B_4D_1 \\
  & + D_1B_3D_1B_5D_1 + D_1B_5D_1B_3D_1 + D_1B_3D_1B_4QS +SQB_4D_1B_3D_1\\
  &+ D_1B_4D_1B_3QS +SQB_3D_1B_4D_1 +D_1B_3QSQB_4D_1 +D_1B_4QSQB_3D_1 \\
  &  +D_1B_3QSQB_3QS + SQB_3QSQB_3D_1 \big],\\
\widetilde{B}_9=&a_+\big[ D_1B_8D_1+D_1B_7QS +SQB_7D_1+SQB_6QS \big]
   -a_+^2\big[ D_1B_4D_1B_5D_1 + D_1B_5D_1B_4D_1 \\
  & + D_1B_3D_1B_5QS+ SQB_5D_1B_3D_1 + SQB_3D_1B_4QS + SQB_4D_1B_3QS\\
   & + SQB_3D_1B_5D_1+D_1B_5D_1B_3QS +D_1B_3QSQB_5D_1 +D_1B_5QSQB_3D_1 \\
   &+D_1B_3QSQB_4QS + SQB_4QSQB_3D_1 + D_1B_4QSQB_3QS +SQB_3QSQB_4D_1 \\
   &+ D_1B_4QSQB_4D_1 + SQB_3QSQB_3QS \big],\,
   \widetilde{B}_{10}=a_+D_1B_9D_1,
   \end{split}
\end{equation*}
\begin{equation*}
  \begin{split}
\widetilde{B}_{11}=&a_+\big[D_1B_{10}D_1+ D_1B_8QS+QSB_8D_1+SQB_7QS  \big]
       -a_+^2\big[ D_1B_5D_1B_5D_1 \\
 &+ D_1B_5D_1B_4QS + SQB_4D_1B_5D_1 + D_1B_4D_1B_5QS +SQB_5D_1B_4D_1\\
&+ D_1B_4QSQB_5D_1+D_1B_5QSQB_4D_1+ D_1B_3QSQB_5QS+ SQB_5QSQB_3D_1\\
&+D_1B_4QSQB_4QS +QSB_4QSQB_4D_1 + D_1B_5QSQB_3QS + SQB_3QSQB_5D_1\\
&+SQB_3D_1B_5QS+ SQB_5D_1B_3QS + SQB_3QSQB_4QS + SQB_4QSQB_3QS\\
&+SQB_4D_1B_4QS\big],\,
\widetilde{B}_{12}=a_+\big[D_1B_9QS+SQB_9D_1 \big],
 \end{split}
\end{equation*}
\begin{equation*}
  \begin{split}
\widetilde{B}_{13}=&a_+\big[D_1B_{10}QS+SQB_{10}D_1 +SQ B_8QS \big]
    -a_+^2\big[ D_1B_5QSQB_5D_1+SQB_5D_1B_5D_1 \\
&+D_1B_5D_1B_5QS+ D_1B_4QSQB_5QS +SQB_5QSQB_4D_1+SQB_4QSQB_5D_1\\
& +D_1B_5QSQB_4QS+SQB_5D_1B_4QS+SQB_4D_1B_5QS +SQB_3QSQB_5QS\\
& +SQB_5QSQB_3QS +SQB_4QSQB_4QS\big],\,
\widetilde{B}_{14}=a_+SQB_9QS,\\
\widetilde{B}_{15}=&a_+SQB_{10}QS -a_+^2\big[ SQB_5D_1B_5QS+ D_1B_5QSQB_5QS +SQB_5QSQB_5D_1\\
&+ SQB_4QSQB_5QS +SQB_5QSQB_4QS  \big],
\widetilde{B}_{16}=-a_+^2SQB_5QSQB_5QS.
  \end{split}
\end{equation*}

According to Lemma \ref{lemma-JN}, $ \widetilde{M}^+_1(\lambda)$ has bounded inverse if and only if
\begin{equation}\label{H-lambda}
M_2^+(\lambda):= S_1- S_1\big(\widetilde{M}^+_1(\lambda) +S_1\big)^{-1}S_1
 \end{equation}
has bounded inverse on $S_1L^2$. Using the orthogonality in Remark \ref{ortho-rela-SSS}, one has
\begin{equation*}\label{W2M2tuba}
\begin{split}
M_2^+(\lambda)=&  g_0^+(\lambda)^{-1}\Big(\widetilde{ T}_1
+ g_0^+(\lambda)^3\lambda^2S_1\widetilde{B}_1S_1
 +g_0^+(\lambda)^2\lambda^2 S_1\widetilde{B}_2S_1+g_0^+(\lambda)\lambda^2S_1\widetilde{B}_3S_1
 +\lambda^2S_1\widetilde{B}_4S_1 \\
 &+ g_0^+(\lambda)^{-1}\lambda^2S_1\widetilde{B}_5S_1
 +g_0^+(\lambda)^5\lambda^4S_1\widetilde{B}_6S_1 +g_0^+(\lambda)^4\lambda^4S_1\widetilde{B}_7S_1
 +g_0^+(\lambda)^3\lambda^4S_1\widetilde{B}_8S_1\\
&  +g_0^+(\lambda)^2\lambda^4S_1\widetilde{B}_9S_1
 +\tilde{g}_2^+(\lambda)g_0^+(\lambda) \lambda^4S_1\widetilde{B}_{10}S_1
 + g_0^+(\lambda)\lambda^4S_1\widetilde{B}_{11}S_1
 +\tilde{g}_2^+(\lambda) \lambda^4S_1\widetilde{B}_{12}S_1\\
 &+\lambda^4S_1\widetilde{B}_{13}S_1
 +\tilde{g}_2^+(\lambda)g_0^+(\lambda)^{-1}\lambda^4S_1\widetilde{B}_{14}S_1
 +g_0^+(\lambda)^{-1}\lambda^4S_1\widetilde{B}_{15}S_1\\
&  +g_0^+(\lambda)^{-2}\lambda^4S_1\widetilde{B}_{16}S_1
 +O(\lambda^{6-\epsilon})\Big)\\
 :=&g_0^+(\lambda)^{-1}\big( \widetilde{T}_1 +W_1^+(\lambda)\big)
 :=g_0^+(\lambda)^{-1}\widetilde{M}_2^+(\lambda),
  \end{split}
\end{equation*}
 where
\begin{equation}\label{id-T1}
    \begin{split}
\widetilde{T}_1=- S_1T_0(Q-S_0)\big(A+g^+_0(\lambda)^{-1}B\big)^{-1}(Q-S_0)T_0S_1.
\end{split}
\end{equation}
Let $$T_1=- S_1T_0(Q-S_0)A^{-1}(Q-S_0)T_0S_1,$$
then $\widetilde{T}_1$ is invertible on $S_1L^2$ for enough small $\lambda$ if and only if
$T_1$ is invertible on $S_1L^2$.

\textbf{(II) The first kind of resonance.}
 If there is a resonance of first kind at zero, then the operator $T_1$
is invertible on $S_1L^2$. In this case we define $D_2 = T_1^{-1}$ as an operator on $S_1L^2$.
Then  $\widetilde{T}_1$ is invertible on $S_1L^2$ and $D_2 = \widetilde{T}_1^{-1}$ for enough small $\lambda$.  Note that
\begin{align*}
\big(\widetilde{M}_2^+(\lambda)\big)^{-1}= D_2\big( I+W_1^+(\lambda)D_2\big)^{-1}
\end{align*}
and $ \big(M_2^+(\lambda)\big)^{-1} = g_0^+(\lambda)\big(\widetilde{M}_2^+(\lambda)\big)^{-1} $,
by Neumann series one has
\begin{equation*}\label{first-M2-inver}
 \begin{split}
\big(M_2^+(\lambda)\big)^{-1}
= & g_0^+(\lambda)D_2+ g_0^+(\lambda)^4\lambda^2 D_2 \widetilde{B}_1D_2 + g_0^+(\lambda)^3\lambda^2D_2\widetilde{B}_2D_2
+ g_0^+(\lambda)^2\lambda^2D_2\widetilde{B}_3D_2\\
&+ g_0^+(\lambda)\lambda^2D_2\widetilde{B}_4D_2+ \lambda^2D_2\widetilde{B}_5D_2
 + O_1(\lambda^{4-\epsilon}).
\end{split}
\end{equation*}
Using Lemma \ref{lemma-JN} and (\ref{M1tuta+S1-inver}) one has
\begin{equation*}
 \begin{split}
\big(\widetilde{M}_1^+(\lambda)\big)^{-1}
= & g_0^+(\lambda)S_1\Gamma^1_{0,1}S_1 + \Big(Q\Gamma^1_{0,2}S_1 +S_1\Gamma^1_{0,3}Q  \Big)
+ g_0^+(\lambda)^{-1}Q\Gamma^1_{0,5}Q\\
& +g_0^+(\lambda)^{-2}Q\Gamma^1_{0,6}Q
+g_0^+(\lambda)^4\lambda^2S_1\Gamma^1_{2,1}S_1 +g_0^+(\lambda)^3\lambda^2Q\Gamma^1_{2,1}Q\\
&+g_0^+(\lambda)^2\lambda^2Q\Gamma^1_{2,1}Q+g_0^+(\lambda)\lambda^2Q\Gamma^1_{2,1}Q
 + O_1(\lambda^2).
\end{split}
\end{equation*}
By the same arguments as in the proof of Theorem \ref{thm-main-inver-M}(i) one has
\begin{align*}
\big(M^+ (\lambda)\big)^{-1}
 = &g^+_0(\lambda) S_1D_2 S_1+\Big( S_0\Gamma_{0,1}^1 S_0 + S_1\Gamma_{0,2}^1 Q
+ Q\Gamma_{0,3}^1 S_1\Big)+g^+_0(\lambda)^{-1}Q\Gamma_{0,4}^1Q\\
&+g^+_0(\lambda)^4\lambda^2S_1\Gamma_{2,1}^1S_1
+g^+_0(\lambda)^3\lambda^2\Big(S_1\Gamma_{2,2}^1Q
+ Q\Gamma_{2,3}^1S_1 \Big)
+g^+_0(\lambda)^2\lambda^2\\
&\times \Big(Q \Gamma_{2,4}^1Q + S_1\Gamma_{2,5}^1 +\Gamma_{2,6}^1S_1 \Big)
+g^+_0(\lambda)\lambda^2\Big( Q\Gamma_{2,7}^1 +\Gamma_{2,8}^1Q  \Big)+O_1(\lambda^2).
\end{align*}
Thus the expansions of $\big(M^+ (\lambda)\big)^{-1}$ in the first kind resonance is completed.
\cqd
\vskip0.3cm
If zero is not the first kind of resonance, then $T_1$ is  not invertible on $S_1L^2$.
Let $S_2$ is the Riesz transform onto the kernel of $ T_1 $,
then $T_1 +S_2 $ is invertible on $S_1L^2$, which implies that $(\widetilde{T}_1+S_2)^{-1}$ exists for  sufficiently small $\lambda>0$. If we define $D_2= (\widetilde{T}_1+S_2)^{-1}$, then
$S_1D_2=D_2S_1=D_2$, $ S_2D_2=D_2S_2=S_2$ and $D_2D_1=D_1D_2=D_2$. Note that
 \begin{equation*}
  \begin{split}
\big(\widetilde{M}_2^+(\lambda)+S_2\big)^{-1}= D_2\Big(I +W_1^+(\lambda)D_2\Big)^{-1}.
\end{split}
\end{equation*}
By using Neumann series and (\ref{id-T1}) one has
\begin{equation}\label{id-Wtuta+S2-inver}
  \begin{split}
&\big(\widetilde{M}_2^+(\lambda)+S_2\big)^{-1}\\
= &D_2- g_0^+(\lambda)^3\lambda^2D_2\widetilde{B}_1D_2
 -g_0^+(\lambda)^2\lambda^2 D_2\widetilde{B}_2D_2- g_0^+(\lambda)\lambda^2D_2\widetilde{B}_3D_2
 -\lambda^2D_2\widetilde{B}_4D_2\\
 & - g_0^+(\lambda)^{-1}\lambda^2D_2\widetilde{B}_5D_2
 -g_0^+(\lambda)^6\lambda^4B_1^1 -g_0^+(\lambda)^5\lambda^4B_2^1
  -g_0^+(\lambda)^4\lambda^4B_3^1
-g_0^+(\lambda)^3\lambda^4B_4^1\\
 &-g_0^+(\lambda)^2\lambda^4B_5^1
-\tilde{g}_2^+(\lambda)g_0^+(\lambda)\lambda^4B_6^1
-g_0^+(\lambda)\lambda^4B_7^1
-\tilde{g}_2^+(\lambda)\lambda^4B_8^1 -\lambda^4B_9^1\\ &-\tilde{g}_2^+(\lambda)g_0^+(\lambda)^{-1}\lambda^4B_{10}^1
-g_0^+(\lambda)^{-1}\lambda^4B_{11}^1 -g_0^+(\lambda)^{-2}\lambda^4B_{12}^1
+O(\lambda^{6-\epsilon}),
\end{split}
\end{equation}
where $ B_j^1(1\leq j\leq 12 )$ are bounded operators on on $L^2$ as follows:
\begin{equation*}
  \begin{split}
B_1^1=& -D_2\widetilde{B}_1D_2\widetilde{B}_1D_2,\
B_2^1= D_2\widetilde{B}_6D_2
       - \big[D_2\widetilde{B}_1D_2\widetilde{B}_2D_2+D_2\widetilde{B}_2D_2\widetilde{B}_1D_2  \big],\\
B_3^1=& D_2\widetilde{B}_7D_2- \big[D_2\widetilde{B}_2D_2\widetilde{B}_2D_2
 +D_2\widetilde{B}_1D_2\widetilde{B}_3D_2 +D_2\widetilde{B}_3D_2\widetilde{B}_1D_2 \big],\\
B_4^1=& D_2\widetilde{B}_8D_2
   - \big[D_2\widetilde{B}_1D_2\widetilde{B}_4D_2  +D_2\widetilde{B}_4D_2\widetilde{B}_1D_2
   +D_2\widetilde{B}_2D_2\widetilde{B}_3D_2+D_2\widetilde{B}_3D_2\widetilde{B}_2D_2\big],\\
B_5^1=& D_2\widetilde{B}_9D_2
   - \big[D_2\widetilde{B}_3D_2\widetilde{B}_3D_2  +D_2\widetilde{B}_2D_2\widetilde{B}_4D_2
   +D_2\widetilde{B}_4D_2\widetilde{B}_2D_2+D_2\widetilde{B}_1D_2\widetilde{B}_5D_2\\
   & +D_2\widetilde{B}_5D_2\widetilde{B}_1D_2\big],\
B_6^1= D_2\widetilde{B}_{10}D_2, \\
B_7^1=& D_2\widetilde{B}_{11}D_2
   - \big[D_2\widetilde{B}_2D_2\widetilde{B}_5D_2  +D_2\widetilde{B}_5D_2\widetilde{B}_2D_2
   +D_2\widetilde{B}_3D_2\widetilde{B}_4D_2+D_2\widetilde{B}_4D_2\widetilde{B}_3D_2\big],\\
B_8^1=& D_2\widetilde{B}_{12}D_2,
B_9^1= D_2\widetilde{B}_{13}D_2
   - \big[D_2\widetilde{B}_4D_2\widetilde{B}_4D_2  +D_2\widetilde{B}_3D_2\widetilde{B}_5D_2
   +D_2\widetilde{B}_5D_2\widetilde{B}_3D_2\big],\\
B_{10}^1=& D_2\widetilde{B}_{14}D_2, \
B_{11}^1= D_2\widetilde{B}_{15}D_2
   - \big[D_2\widetilde{B}_4D_2\widetilde{B}_5D_2  +D_2\widetilde{B}_5D_2\widetilde{B}_4D_2\big],\\
B_{12}^1=& D_2\widetilde{B}_{16}D_2
   - D_2\widetilde{B}_5D_2\widetilde{B}_5D_2.
\end{split}
\end{equation*}
According to Lemma \ref{lemma-JN}, $ \widetilde{M}_2^+(\lambda)$ has bounded inverse on $S_1L^2$
if and only if
\begin{equation}\label{id- M2-tata}
M_3^+(\lambda):= S_2- S_2(\widetilde{M}_2^+(\lambda)+S_2)^{-1}S_2
 \end{equation}
has bounded inverse on $ S_2L^2(\mathbb{R}^2)$. By Neumann series we have
\begin{equation*}
  \begin{split}
M^+_3(\lambda)
= &g_0^+(\lambda)^3\lambda^2S_2\widetilde{B}_1S_2
 +g_0^+(\lambda)^2\lambda^2 S_2\widetilde{B}_2S_2+ g_0^+(\lambda)\lambda^2S_2\widetilde{B}_3S_2
 +\lambda^2S_2\widetilde{B}_4S_2\\
 & + g_0^+(\lambda)^{-1}\lambda^2S_2\widetilde{B}_5S_2
 +g_0^+(\lambda)^6\lambda^4S_2B_1^1S_2 +g_0^+(\lambda)^5\lambda^4S_2B_2^1S_2
  +g_0^+(\lambda)^4\lambda^4S_2B_3^1S_2\\
&+g_0^+(\lambda)^3\lambda^4S_2B_4^1S_2
 +g_0^+(\lambda)^2\lambda^4S_2B_5^1S_2
+\tilde{g}_2^+(\lambda)g_0^+(\lambda)\lambda^4S_2B_6^1S_2
+g_0^+(\lambda)\lambda^4S_2B_7^1S_2\\
&+\tilde{g}_2^+(\lambda)\lambda^4S_2B_8^1S_2 +\lambda^4S_2B_9^1S_2 +\tilde{g}_2^+(\lambda)g_0^+(\lambda)^{-1}\lambda^4S_2B_{10}^1S_2
+g_0^+(\lambda)^{-1}\lambda^4S_2B_{11}^1S_2\\
&+g_0^+(\lambda)^{-2}\lambda^4S_2B_{12}^1S_2
+O(\lambda^{6-\epsilon}).
\end{split}
\end{equation*}
Note that $ S_2S=SS_2 =0$, $QvG_{-1}vS_0 = 0, QT_0 S_2 =0 $, hence by (\ref{M1tuta+S1-inver}) we have
$$ S_2\widetilde{B}_1S_2= -\frac{1}{a_+}S_2 v G_{-1}v vG_{-1}vS_2 , \
S_2\widetilde{B}_2S_2=-\frac{1}{a_+}\big( S_2 v G_{-1}vT_0S_2+ S_2T_0vG_{-1}vS_2 \big), $$
$$ S_2\widetilde{B}_3S_2=c_+ S_2vG_1vS_2-\frac{1}{a_+}S_2T_0^2S_2,$$
$$ S_2\widetilde{B}_4S_2= S_3\widetilde{B}_5S_2 =S_2B_{8}^1S_2= S_2B_{10}^1S_2=S_2B_{12}^1S_2=0. $$
Moreover, we get
\begin{equation}\label{A1lambda-W2}
  \begin{split}
M^+_3(\lambda)
&= -\frac{1}{a_+}g_0^+(\lambda)\lambda^2\Big( A_1^+(\lambda)
-g_0^+(\lambda)^5\lambda^2 a_+S_2B_1^1S_2 -g_0^+(\lambda)^4\lambda^2a_+S_2B_2^1S_2\\
 &-g_0^+(\lambda)^3\lambda^2a_+S_2B_3^1S_2
-g_0^+(\lambda)^2\lambda^2a_+S_2B_4^1S_2-g_0^+(\lambda)\lambda^2a_+S_2B_5^1S_2\\
&-\tilde{g}_2^+(\lambda)\lambda^2a_+S_2B_6^1S_2-\lambda^2a_+S_2B_7^1S_2
 -g_0^+(\lambda)^{-1}\lambda^2a_+S_2B_9^1S_2\\
&-g_0^+(\lambda)^{-2}\lambda^2a_+S_2B_{11}^1S_2+O(\lambda^{4-\epsilon})\Big)\\
&:= -\frac{1}{a_+}g_0^+(\lambda)\lambda^2 \big(A^+_1(\lambda)+W_2^+(\lambda) \big)
:=-\frac{1}{a_+}g_0^+(\lambda)\lambda^2\widetilde{ M}^+_3(\lambda),
\end{split}
\end{equation}
where
\begin{equation*}
  \begin{split}
A_1^+(\lambda)=&g_0^+(\lambda)^2S_2 v G_{-1}v vG_{-1}vS_2
 +g_0^+(\lambda)\big( S_2 v G_{-1}vT_0S_2+ S_2T_0vG_{-1}vS_2 \big)\\
 &+ \big(S_2T_0^2S_2 -c_+a_+ S_2vG_1vS_2\big).
\end{split}
\end{equation*}
To compute $\big(M_3^+(\lambda)\big)^{-1}$, by (\ref{A1lambda-W2}) we need to compute
the expansions of $\big(\widetilde{M_3^+}(\lambda)\big)^{-1}$ on $S_2L^2$.
Noting that $A_1^+(\lambda)$ with a logarithm factor, and  $ T_2$ is not invertible on $S_2L^2$ by Lemma \ref{projiction-spaces-SjL2}(iv).
Hence we will use Lemma \ref{lemma-JN-matrix} to study the inverse of $A_1^+(\lambda)$.

To compute the inverse of $A_1^+(\lambda)$ by Lemma \ref{lemma-JN-matrix},
we need the following lemma.
\begin{lemma}\label{lem-d1-inver}
Suppose that $|V(x)| \lesssim (1+|x|)^{-\beta}$ with some $\beta >18$. Let
\begin{equation*}
   \begin{split}
L_1=& (S_2-S_3)\Big(vG_{-1}vPvG_{-1}v- vG_{-1}vT_0D_4T_0vG_{-1}v \Big)(S_2-S_3),\\
L_2=& (S_2-S_3)\Big(vG_{-1}vT_0 + T_0vG_{-1}v- vG_{-1}vT_0D_4(T_0^2-c_+a_+vG_1v) \\
      &-(T_0^2-c_+a_+vG_1v)D_4T_0vG_{-1}v \Big)(S_2-S_3),\\
L_3=&(S_2-S_3)\Big(T_0^2-c_+a_+vG_1v  - (T_0^2-c_+a_+vG_1v)D_4(T_0^2-c_+a_+vG_1v) \Big)(S_2-S_3).
 \end{split}
\end{equation*}
Then the inverse operator
$$ d_1:= g^+_0(\lambda)^{-2}\Big(L_1 + g^+_0(\lambda)^{-1}L_2+ g^+_0(\lambda)^{-2}L_3\Big)^{-1}$$
exists on $(S_2-S_3)L^2$ for enough small $\lambda>0$. Moreover, we have
$$d_1= g^+_0(\lambda)^{-2}h^+(\lambda)(S_2-S_3),$$
where $h^+(\lambda)= \Big(c_1+ g^+_0(\lambda)^{-1}c_2 + g^+_0(\lambda)^{-2}c_3\Big)^{-1}$ with
 $ c_1= \hbox{tr} (L_1) $,  $ c_2= \hbox{tr} (L_2) $ and  $ c_3= \hbox{tr}( L_3 )$.
In particular, $ \displaystyle  h^+(\lambda)=O_1(1)$.
\end{lemma}
\begin{proof}
 We first show that $L_1$ is invertible on $(S_2-S_3)L^2$. Indeed, by Lemma \ref{projiction-spaces-SjL2}, we know that $(S_2-S_3)L^2= \hbox{span}\{ S_2(|x|^2v)\}$, thus it suffices to prove
 $\big\langle |x|^2v, L^1(|x|^2v) \big\rangle >0$. Using $S_2v=S_2(x_iv)=0(i=1,2)$ and
 $S_3v=S_3(x_iv)=S_3(|x|^2v)=0(i=1,2)$, one has
 \begin{equation}\label{esti-PvG-1vS2}
   \begin{split}
PvG_{-1}vS_2(|x|^2v)(x)
 =&P\Big(v(x)\int_{\mathbb{R}^2}\big( |x|^2-2x\cdot y +|y|^2\big)v(y)S_2(|x|^2v)(y)dy\Big)\\
 =&P\big\langle |x|^2v, S_2(|x|^2v)\big\rangle v(x)
 =\| S_2(|x|^2v)\|^2_{L^2}v(x).
\end{split}
\end{equation}
Let $\phi(x)=PvG_{-1}vS_2(|x|^2v)(x)$, by $S_3(|x|^2v)=0$  and (\ref{esti-PvG-1vS2}), one has
 \begin{equation*}
   \begin{split}
&(S_2-S_3)vG_{-1}vPvG_{-1}v(S_2-S_3)(|x|^2v)(x)
=\big[(S_2-S_3)vG_{-1}v\phi\big](x)\\
 =&(S_2-S_3)v(x)\int_{\mathbb{R}^2}\big( |x|^2-2x\cdot y +|y|^2\big)v(y)\phi(y)dy\\
 =& \| S_2(|x|^2v)\|^2_{L^2} (S_2-S_3)|x|^2v(x)\int_{\mathbb{R}^2}v(y)v(y)dy\\
 =& \|V\|_{L^1}\| S_2(|x|^2v)\|^2_{L^2}v(x).
\end{split}
\end{equation*}
Furthermore,
\begin{equation}\label{L1-firstterm}
 \Big\langle S_2(|x|^2v), (S_2-S_3)vG_{-1}vPvG_{-1}v(S_2-S_3)(|x|^2v)\Big\rangle
= \|V\|_{L^1}\| S_2(|x|^2v)\|^4_{L^2}.
\end{equation}
Using (\ref{esti-PvG-1vS2}) and $S_3(|x|^2v)=0$, we have
\begin{equation}\label{L1-secondterm}
   \begin{split}
&\big\langle S_2(|x|^2v), (S_2-S_3)vG_{-1}vT_0D_4T_0vG_{-1}vS_2(|x|^2v) \big\rangle\\
=&\big\langle D_4T_0PvG_{-1}vS_2(|x|^2v), T_0PvG_{-1}vS_2(|x|^2v) \big\rangle
=\| S_2(|x|^2v)\|^4_{L^2}\big\langle D_4T_0v, T_0v \big\rangle.
\end{split}
\end{equation}
Combing (\ref{L1-firstterm}) and (\ref{L1-secondterm}), one has
 \begin{equation}\label{esti-L1-twoterm}
   \begin{split}
\big\langle S_2(|x|^2v), L_1(|x|^2v)\big\rangle
=\| S_2(|x|^2v)\|^4_{L^2} \|V\|_{L^1} \Big( 1- \frac{ \langle D_4T_0v, T_0v\rangle}{\|V\|_{L^1}}\Big).
\end{split}
\end{equation}
Let $\eta=D_4S_3T_0v$ and $\xi=S_3T_0v$, then $\xi=D_4^{-1}\eta=(T_3+S_4)\eta$.
By $S_3D_4=D_4S_3=D_4$, thus we have
$$ 1- \frac{ \langle D_4T_0v, T_0v\rangle}{\|V\|_{L^1}}
= 1- \frac{ \langle S_3D_4T_0v, S_3T_0v\rangle}{\|V\|_{L^1}}
= 1- \frac{ \langle \eta, \xi\rangle}{\|V\|_{L^1}}.$$
Noting that
$$T_3+S_4= \big(-c_+a_+S_3vG_1vS_3 +S_4\big)+S_3T_0^2S_3:= A_1+B_1,$$
by (\ref{S3vG1vS3}) we know that for  $f\in S_3L^2$,
\begin{equation*}
  \begin{split}
\langle f, A_1f\rangle
=& \frac{\|V\|_{L^1}}{1024}\Big(\langle x_1^2v, f \rangle^2
+\langle x_2^2v, f \rangle^2+2\langle x_1x_2v, f \rangle^2 \Big) +\|S_4f\|_{L^2}^2 > 0,
\end{split}
\end{equation*}
since $\eta=D_4\xi=S_3D_4\xi\in S_3L^2$, thus $\langle \eta, A_1\eta\rangle>0$.
Furthermore,
\begin{equation}\label{ineq-etaA1eta}
0<\langle \eta, A_1\eta\rangle= \langle \eta, (T_3+S_4)\eta-B_1\eta \rangle
= \langle \eta, \xi\rangle -\langle \eta, B_1\eta \rangle.
\end{equation}
By $S_3T_0Q=QT_0S_3=0$, then
$$B_1\eta=S_3T_0^2S_3\eta=S_3T_0PT_0S_3\eta=\frac{\langle T_0S_3\eta, v\rangle S_3T_0v}{\|V\|_{L^1}}
= \frac{\langle \eta, S_3T_0v\rangle S_3T_0v}{\|V\|_{L^1}}
=\frac{\langle \eta, \xi\rangle \xi}{\|V\|_{L^1}}, $$
which by (\ref{ineq-etaA1eta}) implies that
$$0<\langle \eta, A_1\eta\rangle = \langle \eta, \xi\rangle
- \frac{\big|\langle\eta, \xi\rangle\big|^2}{\|V\|_{L^1}}, $$
which implies that $ \langle \eta, \xi\rangle\neq 0$ and $ 1- \frac{ \langle \eta, \xi\rangle}{\|V\|_{L^1}} \neq 0$. Then we have $ 1- \frac{ \langle D_4T_0v, T_0v\rangle}{\|V\|_{L^1}}\neq 0$ and $ \big\langle S_2(|x|^2v), L_1(|x|^2v)\big\rangle\neq 0$.
Thus, $L_1$ is invertible on $(S_2-S_3)L^2$. Hence $d_1$ exists on $(S_2-S_3)L^2$ for
enough small $\lambda$.

Note that $(S_2-S_3)L^2= \hbox{span}\{ S_2(|x|^2v)\}$ ( i.e. $rank(S_2-S_3)=1$ ). Let $f \in L^2$, then
$L_1 f = (S_2-S_3) L_1f  \in (S_2-S_3)L^2  $. Thus,
$L_1 f = \langle S_2(|x|^2v), L_1f \rangle (S_2|x|^2 v) $, and  $ L_1= \hbox{tr}(L_1):=c_1\neq 0$
by $ \big\langle S_2(|x|^2v), L_1(|x|^2v)\big\rangle\neq 0$.
Similarly, $L_2=  \hbox{tr}(L_2):=c_2 $ and $L_3= \hbox{tr}(L_3):=c_3$. Thus, we have
\begin{equation*}
   \begin{split}
d_1=& \Big(g^+_0(\lambda)^2L_1 + g^+_0(\lambda)L_2+ L_3\Big)^{-1}\\
  =& g^+_0(\lambda)^{-2}(S_2-S_3)\Big(L_1 + g^+_0(\lambda)^{-1}L_2
        + g^+_0(\lambda)^{-2}L_3\Big)^{-1}(S_2-S_3)\\
  =& g^+_0(\lambda)^{-2}(S_2-S_3) \Big(c_1 + g^+_0(\lambda)^{-1}c_2
         + g^+_0(\lambda)^{-2}c_3\Big)^{-1}(S_2-S_3)\\
  :=& g^+_0(\lambda)^{-2}h^+(\lambda)(S_2-S_3),
\end{split}
\end{equation*}
where $h^+(\lambda)= \Big(c_1 + g^+_0(\lambda)^{-1}c_2+ g^+_0(\lambda)^{-2}c_3\Big)^{-1}$ and  $h^+(\lambda)=O_1(1)$.
\end{proof}

Now we start to compute the inverse of $A_1^+(\lambda)$ as follows:
\begin{lemma}\label{lem-A1-inver}
Suppose that $|V(x)| \lesssim (1+|x|)^{-\beta}$ with some $\beta >18$.
If zero is the second kind of resonance of the spectrum of $H$,
then for $ 0<|\lambda|\ll 1$,
\begin{equation*}
   \begin{split}
 \big(A_1^+(\lambda)\big)^{-1}
=&D_4 +h^+(\lambda)\widetilde{S}(\lambda),
\end{split}
\end{equation*}
where $D_4=T_3^{-1}$ and
\begin{align}\label{matrix-S1}
 \nonumber \widetilde{S}(\lambda)=\widetilde{S}=
\begin{pmatrix}
 b_{11}
  &  b_{12}\\
  b_{21}
  &  b_{22}
 \end{pmatrix},
\end{align}
which  $h(\lambda)$ is defined in Lemma \ref{lem-d1-inver}, and
$b_{ij}=b_{ij}(\lambda)$ are absolutely bounded operators for enough small $\lambda$ as follows:
\begin{equation}\label{bijlambda}
   \begin{split}
b_{11}=& g^+_0(\lambda)^{-2}(S_2-S_3),\\
b_{12}=&-\big(S_2-S_3\big) \Big( g^+_0(\lambda)^{-1}vG_{-1}vT_0+ g^+_0(\lambda)^{-2}(T_0^2-c_+a_+vG_1v ) \Big)D_4,\\
b_{21}=&-D_4 \Big( g^+_0(\lambda)^{-1}T_0vG_{-1}v+ g^+_0(\lambda)^{-2}(T_0^2-c_+a_+vG_1v ) \Big)\big(S_2-S_3\big),\\
b_{22}=& D_4\Big( T_0vG_{-1}v+ g^+_0(\lambda)^{-1}(T_0^2-c_+a_+vG_1v) \Big)\big(S_2-S_3\big)\\
  &\times \Big( vG_{-1}vT_0+ g^+_0(\lambda)^{-1}(T_0^2-c_+a_+vG_1v )\Big)D_4.
\end{split}
\end{equation}
Moreover, $\|\widetilde{S}(\lambda)\|_{L^2\rightarrow L^2}= O_1(1)$ for small enough $\lambda$.

In addition, if zero is not the first kind of resonance, the same formula holds for $\big(A^+_1(\lambda)+S_4\big)^{-1}$ with $D_4= (T_3+ S_4)^{-1}$.
\end{lemma}
\begin{proof}
We only prove the case of $S_4 \neq 0$ since the case of $S_4=0$ is identical.
In this case, $T_3+S_4$ is invertible on  $S_3L^2$.
Observe that $ vG_{-1}vS_3 =0$, we have
\begin{equation*}
   \begin{split}
A_1^+(\lambda)+S_4=&\Big((S_2-S_3)+S_3\Big)\Big(g_0^+(\lambda)^2 v G_{-1}vPvG_{-1}v
 +g_0^+(\lambda)\big(  v G_{-1}vT_0+ T_0vG_{-1}v\big)\\
 & +T_0^2 -c_+a_+ vG_1v \Big)\Big((S_2-S_3)+S_3\Big).
 \end{split}
\end{equation*}
By block format we write $A_1^+(\lambda)+S_4$ as follows:
\begin{equation}
   \begin{split}
A_1^+(\lambda)+S_4  =&
\begin{pmatrix}
 a_{11} & a_{12}\\
  a_{21} & a_{22}
 \end{pmatrix},
\end{split}
\end{equation}
where $a_{ij}=a_{ij}(\lambda)$ are bounded operators for small enough $\lambda$ as follows:
\begin{equation*}
   \begin{split}
a_{11}= &(S_2-S_3)
\Big(g_0^+(\lambda)^2 v G_{-1}vPvG_{-1}v\\
 &+g_0^+(\lambda)\big(  v G_{-1}vT_0+ T_0vG_{-1}v\big) +T_0^2
  -c_+a_+ vG_1v \Big)(S_2-S_3),\\
a_{12}=& (S_2-S_3)\Big(g_0^+(\lambda)vG_{-1}vT_0 + T_0^2-c_+a_+ vG_1v   \Big)S_3,\\
a_{21}=&S_3\Big(g_0^+(\lambda)T_0 vG_{-1}v + T_0^2 -c_+a_+ vG_1v   \Big)(S_2-S_3),\\
a_{22}=& S_3T_0^2S_3 -c_+a_+S_3vG_1vS_3 +S_4=T_3+S_4.
\end{split}
\end{equation*}
 By Lemma \ref{lemma-JN-matrix}, then $A^+(\lambda)+S_4$ has a bounded inverse if and only if
$$ d_1 := (a_{11} -a_{12}a_{22}^{-1}a_{21})^{-1}$$
 exist and is bounded. Let
$D_4= (T_3+S_4)^{-1} :  S_3 L^2 \rightarrow S_3 L^2$. Note that $rank (S_2-S_3)= 1$,
 and $S_3D_4=D_4S_3=D_4$. By Lemma \ref{lem-d1-inver}, the inverse operator
$$ d_1:= g^+_0(\lambda)^{-2}\Big(L_1 + g^+_0(\lambda)^{-1}L_2+ g^+_0(\lambda)^{-2}L_3\Big)^{-1}$$
exists on $(S_2-S_3)L^2$ for enough small $\lambda>0$, and
$$d_1= g^+_0(\lambda)^{-2}h^+(\lambda)(S_2-S_3).$$
Using Lemma \ref{lemma-JN-matrix} again, $A_1^+(\lambda)+S_4$ is invertible and
\begin{align*}
\big( A_1^+(\lambda) +S_4\big)^{-1}=& D_4 +h^+(\lambda)\widetilde{S}(\lambda),
\end{align*}
where
\begin{align}\label{matrix-S1}
 \nonumber \widetilde{S}(\lambda)=\widetilde{S}=
\begin{pmatrix}
 b_{11}
  &  b_{12}\\
  b_{21}
  &  b_{22}
 \end{pmatrix},
\end{align}
where $b_{i,j}(\lambda)$ are bounded operators defined in (\ref{bijlambda}).

It is easy to check that $\|\widetilde{S}(\lambda)\|_{L^2\rightarrow L^2} = O_1(1)$ for
small enough $\lambda$.
\end{proof}

\textbf{(III) The second kind of resonance.}
If zero is the second kind of resonance, then $ T_3$ is invertible on $S_3L^2$, let $D_4= T_3^{-1}$.
Since
$$\big(\widetilde{M}^+_3(\lambda)\big)^{-1}
   =\big(A_1^+(\lambda)\big)^{-1} \Big(I +W_2^+(\lambda)\big(A_1^+(\lambda)\big)^{-1}\Big)^{-1}.$$
Note that $ D_4B_1^1=B_1^1D_4=0$, $D_4B_2^1D_4=0$, by  Lemma \ref{lem-A1-inver} and (\ref{A1lambda-W2}),
 we have
\begin{equation*}
   \begin{split}
\big(\widetilde{M}^+_3(\lambda)\big)^{-1}
=& h^+(\lambda)S_2\Gamma^2_{0,1}S_2 +g^+_0(\lambda)^5 \lambda^2 S_2\Gamma^2_{2,1}S_2
+ \lambda^2S_2\Gamma^2_{2,2}S_2\\
&+g^+_0(\lambda)^{10} \lambda^4 S_2\Gamma^2_{4,1}S_2 +\lambda^4 S_2\Gamma^2_{4,2}S_2
+O_1(\lambda^4\ln\lambda).
\end{split}
\end{equation*}
By (\ref{A1lambda-W2}) one has
\begin{equation*}
   \begin{split}
\big(M^+_3(\lambda)\big)^{-1}=&\frac{g_0^+(\lambda)^{-1} h^+(\lambda)}{\lambda^2}S_2\Gamma^2_{-2,1}S_2
+g_0^+(\lambda)^{4}S_2\Gamma^2_{0,1}S_2 +S_2\Gamma^2_{0,2}S_2\\
&+g_0^+(\lambda)^{9}\lambda^2S_2\Gamma^2_{2,1}S_2+\lambda^2 S_2\Gamma^2_{2,1}S_2
+O_1\big(\lambda^2(\ln\lambda)^{-1}\big).
\end{split}
\end{equation*}
According to Lemma \ref{lemma-JN} and (\ref{id-Wtuta+S2-inver}),
\begin{equation*}
  \begin{split}
\big(\widetilde{M}_2^+(\lambda)\big)^{-1}
= &\frac{h^+(\lambda)}{\lambda^2}S_2\Gamma^2_{-2,1}S_2
+g_0^+(\lambda)^{5}S_2\Gamma^2_{0,1}S_2 + g_0^+(\lambda) S_2\Gamma^2_{0,2}S_2
+g_0^+(\lambda)^{10}\lambda^2S_2\Gamma^2_{2,1}S_2\\
&+g_0^+(\lambda)\lambda^2 S_2\Gamma^2_{2,1}S_2
+O_1\big(\lambda^2\big).
\end{split}
\end{equation*}
By the same argument with  the proof of the first kind of resonance,
 we immediately get
\begin{equation*}
  \begin{split}
\big(M^\pm(\lambda)\big)^{-1}
=&\frac{h^\pm(\lambda)}{\lambda^2} S_2\Gamma^2_{-2,0}S_2
+g^\pm_0(\lambda)^5\Big( S_2\Gamma^2_{0,1} +\Gamma^2_{0,2}S_2 +
S_1\Gamma^2_{0,3}Q + Q\Gamma^2_{0,4}S_1 \Big)\\
&+g^\pm_0(\lambda)^{10}\lambda^2 \Big( S_2\Gamma^2_{2,1}
+\Gamma^2_{2,2}S_2 +S_1\Gamma^2_{2,3}Q + Q\Gamma^2_{2,4}S_1 \Big)\\
&+g^\pm_0(\lambda)\lambda^2\Big( S_2\Gamma^2_{2,5} +\Gamma^2_{2,6}S_2
 +S_1\Gamma^2_{2,7}Q + Q\Gamma^2_{2,8}S_1 \Big)+ O_1(\lambda^2).
\end{split}
\end{equation*}
Thus the expansions of $\big(M^+ (\lambda)\big)^{-1}$ of the second kind resonance is completed.
\cqd
\vskip0.3cm
If zero is not the second kind of resonance. Then $ T_3$ is not invertible on $S_3L^2$.
Let $S_4$ is the Riesz projection onto the kernel of $ T_3 $. Then $T_3+S_4$ is invertible
on $S_3L^2$. Let $D_4=(T_3+S_4)^{-1}$. Noting that the identity,
$$\big(\widetilde{M}^+_3(\lambda) +S_4\big)^{-1} =\big(A_1^+(\lambda)+S_4\big)^{-1}
     \Big(I +W_2^+(\lambda)\big(A_1^+(\lambda)+S_4\big)^{-1}\Big)^{-1},$$
by using Lemma \ref{lem-A1-inver} and (\ref{A1lambda-W2}), we have by Neumann series
\begin{equation}\label{inver-M3+S4}
   \begin{split}
&\big(\widetilde{M}^+_3(\lambda)+S_4\big)^{-1}\\
=& D_4+ \widetilde{S} -g^+_0(\lambda)^5 \lambda^2 B^2_1
-g^+_0(\lambda)^4 \lambda^2B^2_2- g^+_0(\lambda)^3 \lambda^2 B^2_3
-g^+_0(\lambda)^2 \lambda^2 B^2_4\\
&-g^+_0(\lambda) \lambda^2B^2_5(\lambda)
-\widetilde{g}^+_2(\lambda) \lambda^2 B^2_6- \lambda^2B^2_7
-g^+_0(\lambda)^{-1} \lambda^2B^2_9-g^+_0(\lambda)^{-2} \lambda^2B^2_{11}\\
&+O_1(\lambda^{4-\epsilon}),
\end{split}
\end{equation}
where $B^2_j= -a_+\big(D_4+h^+(\lambda)\widetilde{S}\big)S_2B_j^1S_2\big(D_4+h^+(\lambda)\widetilde{S}\big),
1\leq j\leq 11$.

According to Lemma \ref{lemma-JN}, $ \widetilde{M}_3^+(\lambda)$ has bounded inverse if and only if
\begin{equation}\label{id- W3-tata}
M_4^+(\lambda):= S_4- S_4\big(\widetilde{M}_3^+(\lambda)+S_4\big)^{-1}S_4
 \end{equation}
has bounded inverse on $ S_4L^2$.
Since
$$ S_4\widetilde{S}=\widetilde{S}S_4=0,\ S_4 \Big(A_1^+(\lambda)+ S_4\Big)^{-1}= \Big(A_1^+(\lambda)+ S_4\Big)^{-1} S_4 = S_4,$$
thus
$ S_4B_j=B_jS_4=0( j= 1,2,3,4,6,7 )$,
$ S_4\widetilde{B}_j=\widetilde{B}_j S_4= 0(j=1,2,3,7)$. Moreover,
$$S_4B_j^2S_4 =0(j=1,\cdots,5,8,10,11),$$
$$S_4B_6^2S_4=  -a_+S_4vG_2vS_4,  \ S_4B_7^2S_4= - a_+S_4vG_3vS_4, $$
$$S_4B_9^2S_4= c_+^2 a_+ \Big( S_4vG_1vQSQvG_1vS_4+ S_4vG_1vQSD_2SQvG_1vS_4  \Big). $$
Hence, by (\ref{inver-M3+S4}) and (\ref{id- W3-tata}) one has
\begin{equation*}
   \begin{split}
M^+_4(\lambda)
=&-a_+\widetilde{g}_2^+(\lambda)\lambda^2\Big( S_4vG_2vS_4 + \widetilde{g}_2^+(\lambda)^{-1}S_4vG_3vS_4
-c_+^2\widetilde{g}_2^+(\lambda)^{-1}g_0^+(\lambda)^{-1}\\
&\times\big( S_4vG_1vQSQvG_1vS_4+ S_4vG_1vQSD_2SQvG_1vS_4  \big) +O(\lambda^{2-\epsilon})\Big)\\
 :=& -a_+\widetilde{g}_2^+(\lambda)\lambda^2\widetilde{M}^+_4(\lambda).
\end{split}
\end{equation*}

Let $T_4= S_4vG_2vS_4$.
Then
\begin{equation}\label{T4+W3}
   \begin{split}
\widetilde{M}^+_4(\lambda)=&T_4+ W_3^+(\lambda),
\end{split}
\end{equation}
where
\begin{equation*}
   \begin{split}
 W_3^+(\lambda)=&\widetilde{g}_2^+(\lambda)^{-1}S_4vG_3vS_4
-\widetilde{g}_2^+(\lambda)^{-1}g_0^+(\lambda)^{-1}c_+^2
\Big( S_4vG_1vQSQvG_1vS_4\\
&+S_4vG_1vQSD_2SQvG_1vS_4 \Big) +O(\lambda^{2-\epsilon}).
\end{split}
\end{equation*}

\textbf{(IV) The third kind of resonance.}
 If zero is the third kind of resonance of $H$. Then $T_4$ is invertible on $S_4L^2$, and set $D_5=T_4^{-1}$.
Since
$$\big(\widetilde{M}^+_4(\lambda)\big)^{-1} =D_5\Big(I+W_3^+(\lambda)D_5\Big)^{-1},\
\displaystyle \big(M^+_4(\lambda)\big)^{-1}=-\frac{\widetilde{g}_2^+(\lambda)^{-1}}{a_+\lambda^2}
\big(\widetilde{M}^+_4(\lambda)\big)^{-1} ,
$$
by Neumann series one has
\begin{equation*}
   \begin{split}
\big(M^+_4(\lambda)\big)^{-1}=&\frac{ -\widetilde{g}_2^+(\lambda)^{-1}D_5}{a_+\lambda^2} +\frac{\widetilde{g}_2^+(\lambda)^{-2}D_5vG_3vD_5}{a_+\lambda^2}
 + O_1(\lambda^{-2}(\ln\lambda)^{-3}).
\end{split}
\end{equation*}
Using Lemma \ref{lemma-JN} and (\ref{inver-M3+S4}), we have
\begin{equation*}
   \begin{split}
\big(\widetilde{M}^+_3(\lambda)\big)^{-1}
= &\frac{ \widetilde{g}_2^+(\lambda)^{-1}g_0^+(\lambda)^{-1}D_5 }{\lambda^4} -\frac{\widetilde{g}_2^+(\lambda)^{-2} g_0^+(\lambda)^{-1}D_5vG_3vD_5}{\lambda^4}
 + O_1\big(\lambda^{-4}(\ln\lambda)^{-4}\big).
\end{split}
\end{equation*}
Using  (\ref{A1lambda-W2}) one has
\begin{equation*}
   \begin{split}
\big(M^+_3(\lambda)\big)^{-1}
= &\frac{ \widetilde{g}_2^+(\lambda)^{-1}g_0^+(\lambda)^{-1}D_5 }{\lambda^4} -\frac{\widetilde{g}_2^+(\lambda)^{-2} g_0^+(\lambda)^{-1}D_5vG_3vD_5}{\lambda^4}
 + O_1\big(\lambda^{-4}(\ln\lambda)^{-4}\big).
\end{split}
\end{equation*}
By the same arguments as  in the proof of the second kind of resonance, we immediately obtain that
\begin{equation*}
   \begin{split}
\big(M^+(\lambda) \big)^{-1}
= &\frac{ \widetilde{g}_2^+(\lambda)^{-1} }{\lambda^4}D_5 -\frac{\widetilde{g}_2^+(\lambda)^{-2}}{\lambda^4}D_5vG_3vD_5
 + O_1\big(\lambda^{-4}(\ln\lambda)^{-3}\big)\\
 :=& \frac{ \widetilde{g}_2^+(\lambda)^{-1} }{\lambda^4}S_4\Gamma_{-4,1}^3S_4 +\frac{\widetilde{g}_2^+(\lambda)^{-2}}{\lambda^4}S_4\Gamma_{-4,2}^3S_4
 +O_1\big(\lambda^{-4}(\ln\lambda)^{-3}\big).
\end{split}
\end{equation*}
Thus the expansions of $\big(M^+ (\lambda)\big)^{-1}$ of the third kind resonance is completed.
\cqd
\vskip0.3cm

Before computing the expansions of $\big(M^+(\lambda) \big)^{-1}$ of the fourth kind resonance,
we first give a lemma as follows.
\begin{lemma} \label{lemma-S5vG3vS5}
$\hbox{ker} (S_5vG_3vS_5)= \{ 0 \}$,  thus the operator $ T_5=S_5vG_3vS_5$ is invertible on $S_5L^2$.
\end{lemma}
\begin{proof}
If $f\in \hbox{ker} (S_5vG_3vS_5)$, then  $ S_5vG_3vS_5f=0$ for $f\in S_5L^2$, and
$$ 0=\langle S_5vG_3vS_5f , f \rangle= \langle G_3vf, vf\rangle.$$
Since
$$ R_0^+(\mu^4)= \frac{b_+}{\lambda^2}I + g^+_0(\lambda)G_{-1}+G_0+ c_+\lambda^2 G_1
+ \widetilde{g}^+_2(\lambda)G_2 +\lambda^4G_3+ O(\lambda^{6-\epsilon}|x-y|^{8-\epsilon}),$$
so,
 $$G_3=\frac{R_0^+(\lambda^4)-
 \frac{b_+}{\lambda^2}I - g^+_0(\lambda)G_{-1}-G_0- c_+\lambda^2 G_1- \widetilde{g}^+_2(\lambda)G_2
  +O(\lambda^{6-\epsilon}|x-y|^{8-\epsilon}) }{\lambda^4}. $$
Noting that for $f\in S_5L^2 $,
  $$\langle G_2vf, vf \rangle=\langle G_1vf, vf \rangle=\langle G_{-1}vf, vf \rangle =0,$$
$$\langle I vf, vf \rangle= \langle v(x)Iv(y)f(y), f(y)\rangle
=\langle Pf , f   \rangle =0, $$
we get that
\begin{equation*}
   \begin{split}
0=&\langle S_5vG_3vS_5f , f \rangle= \langle G_3vf, vf\rangle\\
=&\lim_{\lambda \rightarrow 0}
\Big\langle\frac{R_0^+(\lambda^4)-
 \frac{b_+}{\lambda^2}I - g^+_0(\lambda)G_{-1}-G_0- c_+\lambda^2 G_1- \widetilde{g}^+_2(\lambda)G_2
  +O(\lambda^{6-\epsilon}|x-y|^{8-\epsilon}) }{\lambda^4} vf, vf \Big\rangle\\
=&\lim_{\lambda \rightarrow 0} \frac{1}{\lambda^4}\big\langle
 \big(R_0^+(\lambda^4)-G_0\big)  vf, vf \big\rangle.
\end{split}
\end{equation*}
Notice that $ G_0= (\Delta^2)^{-1}$, by using the Fourier transform one has
 \begin{align*}
0&= \lim_{\lambda \rightarrow 0} \frac{1}{\lambda^4}
             \Big\langle \Big(\frac{1}{\xi^4-\lambda^4}-\frac{1}{\xi^4} \Big)\widehat{vf}(\xi),
                 \widehat{vf}(\xi) \Big\rangle \\
    &=\lim_{\lambda \rightarrow 0}\Big \langle \frac{\widehat{vf}(\xi)}{(\xi^4-\lambda^4)\xi^4},
                    \widehat{vf}(\xi) \Big\rangle
   = \Big \langle \frac{\widehat{vf}(\xi)}{\xi^8},  \widehat{vf}(\xi) \Big\rangle \\
   &=\int_{\mathbb{R}^2} \frac{| \widehat{vf}(\xi) |^2}{\xi^8} d\xi
   =\langle G_0vf, G_0vf\rangle,
\end{align*}
where we use the monotone convergence theorem as $\lambda\rightarrow 0$ by choosing
$\lambda\in \mathbb{C}$ such that $\lambda^4 <0$. Thus, $vf=0$ since $\widehat{vf}=0$. Noting that
$f\in S_5L^2$, hence $f=-UvG_0vf=0$.
\end{proof}
\textbf{(V) The fourth kind of resonance.}
If zero is not the fourth  kind of resonance,  then $ T_4$ is not invertible on $S_4L^2$.
Let $S_5$ is the Riesz projection onto the kernel of $ T_4 $, then $T_4+S_5$ is invertible
on $S_4L^2$, and set $D_5=(T_4+S_5)^{-1}$. Noting that the identity
$$\big(\widetilde{M}^+_4(\lambda) +S_5\big)^{-1} =D_5\Big(I+W_3^+(\lambda)D_5\Big)^{-1},$$
by (\ref{T4+W3}) and Neumann series,  we have
\begin{equation*}\label{inver-M4+S5}
   \begin{split}
\big(\widetilde{M}^+_4(\lambda)+S_5\big)^{-1}=& D_5- \widetilde{g}_2^+(\lambda)^{-1}D_5vG_3vD_5
+\widetilde{g}_2^+(\lambda)^{-1}g_0^+(\lambda)^{-1}c_+^2
\Big( D_5vG_1vQSQvG_1vD_5\\
&+D_5vG_1vQSD_2SQvG_1vD_5 \Big) +O((\ln\lambda)^{-3}).
\end{split}
\end{equation*}
According to Lemma \ref{lemma-JN}, $ \widetilde{M}_4^+(\lambda)$ has bounded inverse if and only if
\begin{equation}\label{id- M5}
M_5^+(\lambda):= S_5- S_5\big(\widetilde{M}_4^+(\lambda)+S_5\big)^{-1}S_5
 \end{equation}
has bounded inverse on $ S_5L^2(\mathbb{R}^2)$.
Note that $S_5vG_1vQ=QvG_1vS_5=0$, by Neumann series one has
\begin{equation*}
   \begin{split}
M_5(\lambda)=\widetilde{g}_2^+(\lambda)^{-1}S_5vG_3vS_5 + O\big( (\ln\lambda)^{-2} \big).
\end{split}
\end{equation*}
By Lemma \ref{lemma-S5vG3vS5},  $ T_5= S_5vG_3vS_5 $ is invertible on $S_5L^2$.  Let $D_6= T_5^{-1}$.
Then we have
$$\big(M_5(\lambda)\big)^{-1}=\widetilde{g}_2^+(\lambda)D_6+O\big( (\ln\lambda)^{-2}\big) . $$
By the same arguments as in the proof of the third kind of resonance, we immediately obtain that
\begin{equation*}
   \begin{split}
\big(M^+(\lambda)\big)^{-1}
= &\frac{1}{\lambda^4}S_5D_6S_5 +\frac{\widetilde{g}_2^+(\lambda)^{-1}}{\lambda^4} \Big(S_4\Gamma_{-4,1}^4S_4\Big)
 +\frac{\widetilde{g}_2^+(\lambda)^{-2}}{\lambda^4}S_4\Gamma_{-4,2}^4S_4
+ O_1\big(\lambda^{-4}(\ln\lambda)^{-3}\big).
\end{split}
\end{equation*}
Thus, the expansions of $\big(M^+ (\lambda)\big)^{-1}$ of the fourth kind resonance is completed.
\cqd

\bigskip

\section{Classification of threshold spectral subspaces } \label{classification}
By Definition \ref{definition of resonance} we know that the four projection subspaces $S_iL^2(i=1,2,4,5)$ define the four zero resonance types.  In this section, we will identity these four subspace by the distributional solutions of $H\phi=0$ in the weighted spaces $L^\infty_s$ with some $s\in\mathbb{R}$.

We define $\ln^+ t= \ln t$    for $t>1$ and $\ln^+ t =0$ for $0<t\leq 1$; $\ln^- t= 0$ for  $t>1$ and $\ln^- t =-\ln t $ for $0<t\leq 1$. Then $\ln t=\ln^+ t- \ln^- t$.
\begin{proposition}\label{lemma-spectral-S1L2}
 Assume that $|V(x)|\lesssim  (1+|x|)^{-\beta}$ with $\beta>14$.   Then $ f\in S_1L^2(\mathbb{R}^2)\setminus \{ 0\}$ if and only if  $ f= Uv \phi$ with
$\phi \in L^\infty_{-1}(\mathbb{R}^2)$ such that $H\phi=0$ in the distributional sense, and
$$\phi= -G_0vf +c_3x_2+c_2x_1+c_1,$$
where $c_1, c_2, c_3$ are some constants as follows:
 $$c_1= \frac{\langle f_2 , T_0f \rangle \langle x_2v , f_1 \rangle\langle x_1v , v \rangle}
     {\|f_2\|^2_{L^2(\mathbb{R}^2)}\| f_1\|^2_{L^2(\mathbb{R}^2)}\| V\|_{L^1(\mathbb{R}^2)}}
         -  \frac{\langle f_1 , T_0f \rangle \langle x_1v , v \rangle}
          {\| f_1\|^2_{L^2(\mathbb{R}^2)}\| V\|_{L^1(\mathbb{R}^2)}}
          + \frac{\langle v , T_0f \rangle}{\| V\|_{L^1(\mathbb{R}^2)}},$$
 $$\displaystyle c_2=  - \frac{\langle f_2 , T_0f \rangle \langle x_2v , f_1 \rangle}
     {\| f_2\|^2_{L^2(\mathbb{R}^2)}\| f_1\|^2_{L^2(\mathbb{R}^2)}}
     +  \frac{\langle f_1 , T_0f \rangle}{\| f_1\|^2_{L^2(\mathbb{R}^2)}}, \,\ \,
\displaystyle c_3=\frac{\langle f_2 , T_0f \rangle}{\| f_2\|^2_{L^2(\mathbb{R}^2)}}.$$
\end{proposition}
\begin{proof}
Let $f\in S_1L^2$. Then one has $S_0(U+vG_0v)f=0$. Hence we have
$$ 0= S_0(U+vG_0v)f= (I-\widetilde{P})(U+vG_0v) f= Uf+vG_0vf- \widetilde{P}T_0f,  $$
where $ \widetilde{P}$ is the projection onto the span$\{v , x_1v , x_2v\}= (S_0L^2)^\perp$.
Let $v, f_1$ and $f_2$  are the orthogonal basis of $ (S_0L^2)^\perp$ by the Schmidt orthogonalization, i.e. setting
$ \displaystyle f_1= x_1v- \frac{ \langle x_1v, v \rangle}{\|V\|_{L^1}}v  , $
 $\displaystyle f_2=x_2v- \frac{\langle x_2v , f_1 \rangle}{\| f_1 \|^2_{L^2}}f_1
 - \frac{\langle x_2v , v \rangle}{\| V\|_{L^1}}v$.
Then it follows that
\begin{equation*}
\begin{split}
\widetilde{P}T_0f=&\frac{\langle f_2 , T_0f \rangle}{\| f_2\|^2_{L^2}}f_2
                   + \frac{\langle f_1 , T_0f \rangle}{\| f_1\|^2_{L^2}}f_1
                      + \frac{\langle v , T_0f \rangle}{\| V\|_{L^1}}v \\
    =& x_2v \frac{\langle f_2 , T_0f \rangle}{\| f_2\|^2_{L^2}}
     + x_1v\Big( - \frac{\langle f_2 , T_0f \rangle \langle x_2v , f_1 \rangle}
     {\| f_2\|^2_{L^2}\| f_1\|^2_{L^2}}
     +  \frac{\langle f_1 , T_0f \rangle}{\| f_1\|^2_{L^2}}\Big) \\
     &+ v \bigg(  \frac{\langle f_2 , T_0 f\rangle \langle x_2v , f_1 \rangle\langle x_1v , v \rangle}
     {\|f_2\|^2_{L^2}\| f_1\|^2_{L^2}\| V\|_{L^1}}
         -  \frac{\langle f_1 , T_0f \rangle \langle x_1v , v \rangle}
          {\| f_1\|^2_{L^2}\| V\|_{L^1}}
          + \frac{\langle v , T_0f \rangle}{\| V\|_{L^1}}  \bigg)\\
 :=&c_3x_2v +c_2x_1v+c_1v.
\end{split}
\end{equation*}
Furthermore, we have
$$Uf= -vG_0vf + \widetilde{P}T_0f = -vG_0vf +c_3x_2v+ c_2x_1v+c_1v.$$
Since $U^2=1$, then
$$f= U^2f= Uv(-G_0vf +c_3x_2+ c_2x_1+c_1):= Uv\phi.$$
Thus, $ vf= vUv\phi= V\phi$. To show that $H\phi= (\Delta^2+V)\phi=0$, noting that
$\Delta^2(c_3x_2+c_2x_1+c_1)=0$ and $ G_0= \big(\Delta^2\big)^{-1}$, one has
$(\Delta^2+V)\phi= -vf +V\phi=0.$

Next, we show that $\phi \in L^\infty_{-1}:= \{ \langle x\rangle^ {-1} \phi \in L^\infty\}$.
  Since $c_3x_2 +c_2x_1+c_1 \in L^\infty_{-1}$, it suffices to show that
 $$ G_0vf(x)= \frac{1}{8\pi}\int_{\mathbb{R}^2}|x-y|^2\ln|x-y|v(y)f(y)dy\in L^\infty_{-1}.$$
In fact, if $f\in S_0L^2$, by the orthogonality
$ \langle x_1v , f \rangle=\langle x_2v , f \rangle= \langle v , f \rangle=0 $, then
$$ \int_{\mathbb{R}^2}|x|^2\ln|x|v(y)f(y)dy =
\int_{\mathbb{R}^2}x\cdot y \ln|x|v(y)f(y)dy=0, $$
which leads to
\begin{equation}\label{pro51-esti-G0vf-I123}
\begin{split}
 G_0vf(x)
 =& \frac{1}{8\pi}\int_{\mathbb{R}^2}|x|^2(\ln|x-y|-\ln|x|)v(y)f(y)dy\\
 &-\frac{1}{4\pi}\int_{\mathbb{R}^2}x\cdot y (\ln|x-y|-\ln|x|)v(y)f(y)dy \\
  &+  \frac{1}{8\pi}\int_{\mathbb{R}^2}|y|^2\ln|x-y|v(y)f(y)dy
  :=I_1 +I_2 +I_3.
\end{split}
\end{equation}

We now begin to estimate the first integral $I_1$.  We have
\begin{equation}\label{I1-twoparts}
\begin{split}
I_1=& \frac{1}{8\pi}\int_{|x|>2|y|}|x|^2(\ln|x-y|-\ln|x|)v(y)f(y)dy\\
&+ \frac{1}{8\pi}\int_{|x|\leq2|y|}|x|^2(\ln|x-y|-\ln|x|)v(y)f(y)dy:=I_{11}+I_{12}.
\end{split}
\end{equation}
For $I_{11}$.  If  $ |x|>2|y|$, then  $|x-y|<2|x|$, hence $ \frac{|x-y|-|x|}{|x|}< 1$. By Taylor's formula we have
\begin{equation}\label{taylor-ln1}
\ln\frac{|x-y|}{|x|}=\ln\Big(1+ \frac{|x-y|-|x|}{|x|}\Big)
= O\Big(\frac{|y|}{|x|}\Big).
\end{equation}
By H\"{o}lder's inequality and \eqref{taylor-ln1}, one has
\begin{equation*}
\begin{split}
|I_{11}|=&\Big|\int_{|x|>2|y|}|x|^2(\ln|x-y|-\ln|x|)v(y)f(y)dy\Big|\\
 \lesssim& |x|\int_{\mathbb{R}^2} |yv(y)| \cdot |f(y)|dy
\lesssim  |x| \|yv(y)\|_{L^2}\|f\|_{L^2} \lesssim |x|\in L_{-1}^\infty.
\end{split}
\end{equation*}
For $I_{12}$. We have
\begin{equation*}
\begin{split}
\int_{|x|\le2|y|}|x|^2(\ln|x-y|-\ln|x|)v(y)&f(y)dy = \int_{|x|\leq 2|y|} |x|^2 \big( \ln^+|x-y|-\ln^+|x|\big) \cdot v(y)f(y) dy \\
 &- \int_{|x|\leq 2|y|}|x|^2 \big( \ln^-|x-y|-\ln^-|x| \big)\cdot v(y)f(y) dy
:=J_{1}+J_{2}.
\end{split}
\end{equation*}
For the first term $J_1$. If $|x|\leq 2|y|$, then $|x-y|\leq 3(1+|y|)$ and $|x|\leq 2(1+|y|)$.  Since  $\log^+(\cdot)$ is a nondecreasing function,
then one has
\begin{equation}\label{estimate-log}
\big|\ln^+|x-y|-\ln^+|x|\big| \leq \ln^+|x-y|+\ln^+|x|
 \lesssim 1+ \log^+(1+|y|) \lesssim \langle y \rangle^{0+},\,\,\hbox{as}\, \, |x|\leq 2|y|.
\end{equation}
Thus by  H\"{o}lder's inequality and (\ref{estimate-log}),  one has
\begin{equation*}
\begin{split}
|J_{1}|&\lesssim \int_{|x|\leq 2|y|} |x|^2 \big( \ln^+|x-y|+ \ln^+|x|\big) \cdot |v(y)f(y)| dy \\
&\lesssim\int_{|x|\leq 2|y|} |x|^2 \langle y \rangle^{0+} \cdot| v(y)f(y)| dy\\
&\lesssim\int_{|x|\leq 2|y|}  \langle y \rangle^{2+} |v(y)f(y)| dy
\leq   \|v(\cdot) \langle \cdot \rangle^{2+}\|_{L^2} \|f\|_{L^2}<\infty.
\end{split}
\end{equation*}
For the second term $J_{2}$. By using the definition of function $\ln^-(\cdot) $, we have $|x|^2\ln^-|x| \lesssim 1$.
Then by H\"{o}lder's inequality again, we have
\begin{equation*}
	\begin{split}
		|J_{2}|
		\leq & \int_{|x|\leq 2|y|, \,|x-y|<1}|x|^2 \big( -\ln|x-y|\big) \cdot |v(y)f(y)| dy
		+\int_{|x|\leq 2|y|}|x|^2 \ln^-|x| | \cdot |v(y)f(y)| dy\\
		\lesssim &  \int_{|x-y|<1}|y|^2 \frac{1}{|x-y|^{0+}} \cdot |v(y)f(y)| dy
		+ |x|^2 \ln^-|x| \int_{|x|\leq 2|y|}  |v(y)f(y)| dy\\
		\lesssim &  \int_{|x-y|<1}  \frac{1}{|x-y|^{0+}} \cdot |y^2v(y)f(y)| dy
		+   \int_{|x|\leq 2|y|} |yv(y)f(y)| dy \\
		\lesssim & \left(\int_{|x-y|<1}  \Big(\frac{1}{|x-y|^{0+}}\Big)^2 dy\right)^{1/2}  \,\| y^2v(y) \|_{L^\infty} \|f\|_{L^2}
		+  \|yv(y)\|_{L^2} \|f\|_{L^2}\\
		\lesssim & \| y^2v(y) \|_{L^\infty} \|f\|_{L^2}  +  \|yv(y)\|_{L^2} \|f\|_{L^2}
		< \infty.
	\end{split}
\end{equation*}
Combining the estimates for $J_1$ and $J_2$ above, then $|I_{12}|\leq |J_1|+|J_2| <\infty$.  Hence combining with the estimates for $I_{11}$ and $I_{12}$ above, it follows that $I_1\in L^\infty_{-1}$.

By the same arguments with $I_1$,  and notice that $|x\cdot y| \leq |x||y|$, we immediately obtain that $ I_2 \in L^\infty_{-1}$.
Next, we turn to estimate the last integral $I_3$. By H\"{o}lder's inequality again we have
\begin{equation*}
	\begin{split}
		| I_3| = & \Big| \int_{\mathbb{R}^2} \ln|x-y| \cdot |y|^2 v(y)f(y)dy\Big|\\
		=&\Big| \int_{|x-y|<1} \ln|x-y|\cdot |y|^2v(y)f(y)dy
		+\int_{|x-y|\ge1}\ln|x-y|\cdot  |y|^2v(y)f(y)dy    \Big|\\
		\lesssim & \int_{|x-y|<1} \frac{|y|^2|v(y)f(y)|}{|x-y|^{0+}}dy  +
		\int_{|x-y|\ge1}|x-y|^{0+}\cdot |y|^2v(y)f(y)dy   \\
		\lesssim&  \int_{|x-y|<1} \frac{|y|^2 |v(y)f(y)|}{|x-y|^{0+}}dy
		+ | x|^{0+}\int_{\mathbb{R}^2} |y|^{2+} |v(y)f(y)|dy\\
		\lesssim& 1+| x|^{0+}\in L^\infty_{-1}.
	\end{split}
\end{equation*}
Hence combining with all estimates for $I_i(i=1,2,3)$ above and \eqref{pro51-esti-G0vf-I123}, we obtain that $G_0vf\in L^\infty_{-1}$.
That is,  we get $ \phi \in L^\infty_{-1}$.

Conversely,  let  $\phi\in L^\infty_{-1}$ satisfy $H\phi=0$ and $f=Uv\phi$. Then we will  prove $f\in S_2L^2$ below.  Firstly,  we show that $f\in S_0L^2$, i.e.   $\langle v, f\rangle=\langle x_1v, f\rangle=\langle x_2v, f\rangle=0$. Indeed, let $\eta\in C_0^\infty(\mathbb{R}^2)$
such that $\eta(x)=1$ for $|x|\leq 1$ and $\eta(x)=0$ for $|x|>2$. Notice that $vf=V\phi=-\Delta^2\phi$, then we have for any $\delta>0$,
\begin{equation}\label{ortoganial-vf}
	\begin{split}
		\Big|\int_{\mathbb{R}^2} v(y)f(y)\eta(\delta y) dy \Big|
		=& \Big|\int_{\mathbb{R}^2}  \big(\Delta^2\phi(y)\big)\eta(\delta y)dy\Big|
		=\Big|\int_{\mathbb{R}^2} \phi(y) \cdot \Delta^2\big(\eta(\delta y)\big)dy\Big| \\
		=&\delta^4 \Big|\int_{\mathbb{R}^2} \langle y\rangle^{-1}\phi(y)\cdot   \langle y\rangle (\Delta^2\eta)(\delta y)dy\Big| \\
		\leq & \delta^4 \|\langle y \rangle^{-1}\phi(y)\|_{L^\infty}\int_{\mathbb{R}^2} \big|\langle y\rangle (\Delta^2\eta)(\delta y)\big|dy \\
		\lesssim&  \delta^2 \|\langle y \rangle^{-1}\phi(y)\|_{L^\infty}  \| \Delta^2\eta \|_{L^1}.
	\end{split}
\end{equation}
Taking the limit $\delta\rightarrow0$ for \eqref{ortoganial-vf},
by the dominated convergence theorem we obtain that $\langle v, f\rangle=0$. Similarly, we also have
$\langle x_1v, f\rangle=\langle x_2v, f\rangle=0.$
Hence we get $f\in S_0L^2$.

Next, we turn to prove $\phi=-G_0vf+c_3x_2+c_2x_1+c_1$. In fact, because of $f\in S_0L^2$, by using the estimates of $I_i(i=1,2,3)$ in \eqref{pro51-esti-G0vf-I123}, we have $G_0vf\in L^\infty_{-1}$ and  $\phi+G_0vf \in L^\infty_{-1}$. Noting that
\begin{equation}\label{delta-act}
	\Delta^2(\phi+G_0vf)=\Delta^2\big(\phi+(\Delta^2)^{-1}vf \big)
	=\Delta^2(\phi+vf)=(\Delta^2+V)(\phi)=0,
\end{equation}
then it follows from \eqref{delta-act}  that  Fourier transform of the function  $\phi+G_0vf $ is  supported at the origin,
which is the sum of finite derivatives of Dirac distribution $\delta$ ( See e.g. \cite{GTM249} ).  This immediately  implies that $\phi+G_0vf$ is a polynomial on $\mathbb{R}^2$. Since $\phi+G_0vf \in L^\infty_{-1}$,  the degree of the polynomial $\phi+G_0vf$  is at most $1$.
Hence we may write that
$$ \phi=-G_0vf+ \widetilde{c}_3x_2
+\widetilde{c}_2x_1+\widetilde{ c}_1,$$
where  $\widetilde{ c}_i (i=1,2,3)$  are some constants.
Since $T_0=U+vG_0v$ and
\begin{equation*}
	\begin{split}
		0=H\phi=&\big(\Delta^2+V \big)\phi=-\big(\Delta^2+V \big)\big(G_0vf\big)
		+\big(\Delta^2+V \big)\big(\widetilde{c}_3x_2+\widetilde{c}_2x_1+\widetilde{c}_1 \big)\\
		=& -Uv\big(U+vG_0v\big)f +U\big(\widetilde{c}_3x_2v^2+\widetilde{c}_2x_1v^2+\widetilde{c}_1v^2\big)\\
		=&U\big(-vT_0f+\widetilde{c}_3x_2v^2+\widetilde{c}_2x_1v^2+\widetilde{c}_1v^2 \big),
	\end{split}
\end{equation*}
we obtain
\begin{equation}\label{vT0f-11}
	\begin{split}
		\widetilde{c}_1v^2+\widetilde{c}_2x_1v^2+\widetilde{c}_3x_2v^2=vT_0f,
	\end{split}
\end{equation}
which leads to
\begin{equation}\label{vT0f-11-1}
	\widetilde{c}_1 \langle v, v\rangle
	+ \widetilde{c}_2 \langle v, x_1v\rangle+\widetilde{c}_3 \langle v, x_2v\rangle
	=\langle v,  T_0f\rangle.
\end{equation}
By \eqref{vT0f-11},  we also have
$$\widetilde{c}_1x_1v^2+\widetilde{c}_2x_1^2v^2+\widetilde{c}_3x_1x_2v^2=x_1vT_0f,\,\ \,
\widetilde{c}_1x_2v^2+\widetilde{c}_2x_1x_2v^2+\widetilde{c}_3x_2^2v^2=x_2vT_0f,
$$
which imply that
\begin{equation}\label{vT0f-11-2}
	\begin{split}
		& \widetilde{c}_1 \langle x_1v, v\rangle+\widetilde{c}_2 \langle x_1v, x_1v\rangle+\widetilde{c}_3 \langle x_1v, x_2v\rangle
		=\langle x_1v,  T_0f\rangle,\\
		&\widetilde{c}_1 \langle x_2v, v\rangle+\widetilde{c}_2 \langle x_2v, x_1v\rangle+\widetilde{c}_3 \langle x_2v, x_2v\rangle
		=\langle x_2v,  T_0f\rangle.
	\end{split}
\end{equation}
Let $\widetilde{c}= \big(\widetilde{c}_1, \widetilde{c}_2, \widetilde{c}_3  \big)^\intercal$,
$ b=\big(\langle v,  T_0f\rangle, \langle x_1v,  T_0f\rangle, \langle x_2v,  T_0f\rangle\big)^\intercal $, and
$$A=
\begin{pmatrix}
	\langle v, v\rangle & \langle v, x_1v\rangle &  \langle v, x_2v\rangle\\
	\langle x_1v, v\rangle &  \langle x_1v, x_1v\rangle &  \langle x_1v, x_2v\rangle\\
	\langle x_2v, v\rangle &  \langle x_2v, x_1v\rangle &  \langle x_2v, x_2v\rangle
\end{pmatrix}.
$$
Then  the identities \eqref{vT0f-11-1} and \eqref{vT0f-11-2}  can be written as  $A\widetilde{c}=b$.
Let $ k:=(k_1, k_2, k_3)\in \mathbb{R}^3$. Since  $v, x_1v, x_2v$ are linear independent,
so for each $k\neq 0$, we have $k_1v+k_2x_1v+k_3x_2v\neq0$  and
$$ 0<\big(k_1v+k_2x_1v+k_3x_2v, k_1v+k_2x_1v+k_3x_2v \big) = kAk^\intercal, $$
hence $kAk^\intercal$ is a positive definite quadratic form, and  the metrix $A$ is  positive definite,  which gives  that $\det( A) >0$. Thus the linear equations $A\widetilde{c}=b$  has an unique solution. Notice that
$$ x_1v=  f_1+\frac{ \langle x_1v, v \rangle}{\|V\|_{L^1}}v, \,\ \,
x_2v= f_2+\frac{\langle x_2v , f_1 \rangle}{\| f_1 \|^2_{L^2}}f_1+\frac{\langle x_2v , v \rangle}{\| V\|_{L^1}}v,
$$
it is easy to check that the specific $c=(c_1, c_2, c_2)^\intercal$ from  Proposition \ref{lemma-spectral-S1L2} is a solution of the linear equations $A\widetilde{c}=b$.
Hence, we must have  $\widetilde{c}_i=c_i( i=1,2,3)$.

Finally, by $f=Uv\phi$ and $H\phi=0$ one has
\begin{align}\label{S0T0S0f}
	\begin{split}
		S_0T_0S_0f &=S_0T_0(Uv\phi)=S_0(U+vG_0f)(Uv\phi)= S_0(vG_0v\phi +v\phi) \\
		&=S_0vG_0\big(\Delta^2 +V\big)  \phi = S_0vG_0H\phi = 0.
	\end{split}
\end{align}
Thus we have concluded that $f\in S_1L^2$.
\end{proof}

\begin{proposition}\label{lemma-spectral-S2L2}
 Assume that $|V(x)|\lesssim (1+|x|)^{-\beta}$ with $\beta>18$.   Then
$f\in S_2L^2(\mathbb{R}^2)\setminus \{ 0\}$  if and only if  $f=Uv\phi$ with
$\phi \in L^\infty(\mathbb{R}^2)$ such that $H\phi=0$ in the distributional sense,  and
$$\phi= -G_0vf +\frac{\langle v , T_0f \rangle}{ \|V\|_{L^1(\mathbb{R}^2)}}.$$
\end{proposition}
\begin{proof}
Since $ S_2\leq S_1 $, by Proposition \ref{lemma-spectral-S1L2}, we know that $f\in S_1L^2$ if and
only if $f=Uv\phi$ with $ \phi \in L^\infty_{-1}$  such  that $H\phi =0$ in the distributional sense and $$ \phi = -G_0v f +c_3x_2 +c_2x_1+c_1.$$
Thus, we just need to further   prove that
$ c_1= \frac{\langle v , T_0f \rangle}{ \|V\|_{L^1}}$, $c_2= c_3=0$, and
 $\phi \in L^\infty$.
Indeed, by the proof of Proposition \ref{lemma-spectral-S1L2}, we know that
 $$  f_1= x_1v- \frac{ \langle x_1v, v \rangle}{\|V\|_{L^1}}v \in QL^2,\
 f_2=x_2v- \frac{\langle x_2v , f_1 \rangle}{\| f_1 \|^2_{L^2}}f_1
 - \frac{\langle x_2v , v \rangle}{\| V\|_{L^1}}v \in QL^2. $$
If $ f\in S_2L^2$, then $QT_0f = 0$. Hence
$$ \langle f_i, T_0f \rangle= \langle Qf_i, T_0f \rangle =\langle f_i, QT_0f \rangle=0, i=1,2,$$
which leads to
$c_1= \frac{\langle v , T_0f \rangle}{\| V\|_{L^1}}$, and $ c_2=c_3=0$.

Next, we show that $\phi \in L^\infty$.
Since $\frac{\langle v , T_0f \rangle}{ \|V\|_{L^1}}\in L^\infty$,
it suffices to show that
 $$ G_0vf(x)= \frac{1}{8\pi}\int_{\mathbb{R}^2}|x-y|^2\ln|x-y|v(y)f(y)dy \in L^\infty.$$
If $f\in S_2L^2 \subset S_1L^2$, then using the orthogonality $S_2v=0, S_2(x_iv)=0(i=1,2)$, one has
 $$ \int_{\mathbb{R}^2}|x|^2\ln|x|v(y)f(y)dy =
\int_{\mathbb{R}^2}x\cdot y \ln|x|v(y)f(y)dy=
 \int_{\mathbb{R}^2} \frac{x\cdot y}{|x|^2}v(y)f(y)dy =0.$$
Furthermore, we have
\begin{equation}\label{pro53-esti-G0vf-I123}
\begin{split}
 G_0vf(x)
  =& \frac{1}{16\pi}\int_{\mathbb{R}^2}|x|^2\Big[\ln\Big(\frac{|x-y|^2}{|x|^2} \Big)+
    2\frac{x\cdot y}{|x|^2}\Big]v(y)f(y)dy\\
  &- \frac{1}{4\pi}\int_{\mathbb{R}^2}x\cdot y(\ln|x-y|-\ln|x|)v(y)f(y)dy\\
  & +\frac{1}{8\pi}\int_{\mathbb{R}^2}|y|^2\ln|x-y|v(y)f(y)dy
  :=I_1 +I_2 +I_3.
\end{split}
\end{equation}

For the first integral $I_1$. We have
\begin{equation}\label{esti-I11-twoparts}
\begin{split}
I_1=&\frac{1}{16\pi}\int_{|x|>4y}|x|^2\Big[\ln\Big(\frac{|x-y|^2}{|x|^2} \Big)+
    2\frac{x\cdot y}{|x|^2}\Big]v(y)f(y)dy\\
   & + \frac{1}{16\pi}\int_{|x|\leq 4|y|}|x|^2\Big[\ln\Big(\frac{|x-y|^2}{|x|^2} \Big)+
    2\frac{x\cdot y}{|x|^2}\Big]v(y)f(y)dy:=I_{11}+I_{12}.
\end{split}
\end{equation}
For $I_{11}$. Since $|x|>4|y|$, then $\frac{|x-y|^2-|x|^2}{|x|^2}<1$.
By Taylor's formula one has
\begin{equation}\label{taylor-ln2}
\begin{split}
 \ln\Big(\frac{|x-y|^2}{|x|^2}\Big)+2\frac{x\cdot y}{|x|^2}
 =\ln\Big(1+\frac{|x-y|^2-|x|^2}{|x|^2}\Big)+2\frac{x\cdot y}{|x|^2}
= \frac{|y|^2}{|x|^2}+ O\Big(\frac{|y|^{2+}}{|x|^{2+}}   \Big).
\end{split}
\end{equation}
By \eqref{taylor-ln2} and H\"{o}lder's inequality, one has
\begin{equation*}
\begin{split}
|I_{11}| \leq  \Big|\int_{|x|>4|y|}|x|^2\Big[\ln\Big(\frac{|x-y|^2}{|x|^2} \Big)+ 2\frac{x\cdot y}{|x|^2}\Big]v(y)f(y)dy \Big|
\lesssim \|y^2v(y) \|_{L^2} \|f\|_{L^2} <\infty.
\end{split}
\end{equation*}
 For $I_{12}$.  Note that the definition of the functions $\ln^+(\cdot)$ and $\ln^-(\cdot)$, by H\"{o}lder's inequality one has
\begin{equation*}
\begin{split}
|I_{12}|= &\Big|\int_{|x|\leq 4|y|}|x|^2\Big[\ln\Big(\frac{|x-y|^2}{|x|^2} \Big)+
    2\frac{x\cdot y}{|x|^2}\Big]v(y)f(y)dy\Big| \\
   \lesssim & \Big |\int_{|x|\leq 4|y|}|x|^2\big( \ln|x-y|-\ln|x|\big) v(y)f(y)dy\Big|
   +\int_{|x|\leq 4|y|}|x|^2 \frac{| y|^2}{|x|^2} |v(y)f(y)|dy  \\
 \lesssim&\big| \int_{|x|\leq 4|y|}|x|^2  \big(\ln^+|x-y|-\ln^+|x|\big) \cdot |v(y)f(y)| dy\Big| \\
 &+\Big|\int_{|x|\leq 4|y|}|x|^2\big( \ln^-|x-y|-\ln^-|x|\big) \cdot |v(y)f(y)| dy\Big|
 +\int_{\mathbb{R}^2} y^2|v(y)f(y)|dy\\
 :=&J_1+J_2+\int_{\mathbb{R}^2} y^2|v(y)f(y)|dy.
\end{split}
\end{equation*}
Hence by the same arguments as in the proof of the terms $ J_{1}$ and $J_{2}$ of
Proposition \ref{lemma-spectral-S1L2}, we get that $I_{12}\in L^\infty$.
Combining with the estimates $I_{11}$ and $I_{12}$ above, we immediately get $I_1\in L^\infty$.

For $I_2$ and $I_3$, by using the same arguments with the proof of $I_2$ and $I_3$ in Proposition \ref{lemma-spectral-S1L2}, we obtain that both $I_2$ and $I_3$ in \eqref{pro53-esti-G0vf-I123} are in $L^\infty$.

Combining  all computations above, we obtain that  $G_0vf  \in  L^\infty$ and $ \phi \in L^\infty$.

On the other hand, let   $\phi\in L^\infty$ satisfy $H\phi=0$ and $f=Uv\phi$. By the similar methods of
\eqref{ortoganial-vf} and \eqref{S0T0S0f}, we have
$$\langle v,f \rangle=0, \,\ \, \langle x_iv, f\rangle=0, \,i=1,2,\,\ \,S_0T_0S_0f=0. $$
Then we can obtain that $f\in S_1L^2$. By the estimates of $I_i(i=1,2,3)$ in \eqref{pro53-esti-G0vf-I123},
we have $G_0vf \in L^\infty \subset L^\infty_{-1}$.
By  Proposition  \ref{lemma-spectral-S1L2}, then  $\phi=-G_0vf + c_3x_2+c_2x_1+c_1$.
Thus $c_3x_2+c_2x_1=\phi+G_0vf -c_1 \in L^\infty$, hence we must have $c_2=c_3=0$, which implies that
$\phi=-G_0vf + \frac{\langle v, T_0f\rangle}{\|V\|_{L^1}}$ and $T_0f \in span\{v\}$ which means $QT_0f=0$. Hence $f\in S_2L^2$.

The proof of this proposition is completed.
\end{proof}

\begin{proposition}\label{lemma-spectral-S3L2}
 Assume that $|V(x)|\lesssim  (1+|x|)^{-\beta} $ with $\beta>18$.   Then
$f\in S_4L^2(\mathbb{R}^2)\setminus \{ 0\}$  if and only if  $f=Uv\phi$ with
$\phi \in L_1^\infty(\mathbb{R}^2)$ such that $H\phi=0$ in the distributional sense and
$$\phi= -G_0vf .$$
\end{proposition}
\begin{remark}\label{rem5-2}
By this proposition above, if $f\in S_4L^2(\mathbb{R}^2)$, then  $f=Uv\phi$ with
$\phi \in L_1^\infty(\mathbb{R}^2)$ such that $H\phi=0$ in the distributional sense, and we have
$f=-UvG_0vf$. Notice that $UvG_0v$ is Hilbert-Schmidt operator, then $S_4$ is a finite rank operator by Fredholm alternative theorem.
\end{remark}
\begin{proof}
Since $ S_4\leq S_3 \leq S_2$, by Proposition \ref{lemma-spectral-S2L2}, we know that $f\in S_2L^2$ if and
only if $f=Uv\phi$ with $ \phi \in L^\infty$  such  that $H\phi =0$ in the distributional sense and $$ \phi = -G_0vf +\frac{\langle v , T_0f \rangle}{ \|V\|_{L^1}}.$$
Thus, we only need to  prove that $\langle v , T_0f \rangle=0$ and $ \phi \in L^\infty_1 $.
Indeed, if $ f\in S_4L^2$, by Lemma \ref{projiction-spaces-SjL2}, then $T_0f = PT_0f +QT_0f=0$, which gives
$\langle v, T_0f \rangle=0$. Thus $ \phi= -G_0vf $.

Next we begin to prove $ \phi= -G_0vf \in  L^\infty_1$.
Indeed, if $f\in S_4L^2$, then
$$\int_{\mathbb{R}^2}  x \cdot y \ln|x|v(y)f(y)dy=
 \int_{\mathbb{R}^2}  |x|^2 \ln|x|v(y)f(y)dy=
\int_{\mathbb{R}^2}  \ln|x| |y|^2v(y)f(y) dy=0,$$
 which leads to
\begin{equation}\label{pro54-esti-G0vf-I123}
\begin{split}
 G_0vf(x)
  =& \frac{1}{16\pi}\int_{\mathbb{R}^2}|x|^2\Big[\ln\Big(\frac{|x-y|^2}{|x|^2} \Big)+
    2\frac{x\cdot y}{|x|^2}-\frac{|y|^2}{|x|^2}\Big]v(y)f(y)dy\\
  &- \frac{1}{4\pi}\int_{\mathbb{R}^2}x\cdot y \Big[\ln\Big(\frac{|x-y|^2}{|x|^2} \Big)+
    2\frac{x\cdot y}{|x|^2}\Big]   v(y)f(y)dy\\
  &+\frac{1}{8\pi}\int_{\mathbb{R}^2}|y|^2(\ln|x-y|-\ln|x|)v(y)f(y)dy
  :=I_1 +I_2 +I_3.
\end{split}
\end{equation}

Now we estimate the first integral $I_1$. We have
\begin{equation}\label{esti-firsrterm-I12}
\begin{split}
I_1=& \frac{1}{16\pi}\int_{|x|>4|y|}|x|^2\Big[\ln\Big(\frac{|x-y|^2}{|x|^2} \Big)+
    2\frac{x\cdot y}{|x|^2}-\frac{|y|^2}{|x|^2}\Big]v(y)f(y)dy\\
  &+\frac{1}{16\pi}\int_{|x|\leq 4|y|}|x|^2\Big[\ln\Big(\frac{|x-y|^2}{|x|^2} \Big)+
    2\frac{x\cdot y}{|x|^2}-\frac{|y|^2}{|x|^2}\Big]v(y)f(y)dy
  :=I_{11}+I_{12}.
\end{split}
\end{equation}
For $I_{11}$. Since $|x|>4|y|$, then by Taylor's formula one has
\begin{equation*}
\begin{split}
\ln\Big(\frac{|x-y|^2}{|x|^2} \Big)+
    2\frac{x\cdot y}{|x|^2}-\frac{|y|^2}{|x|^2}= O\Big(\frac{|y|^3}{|x|^3}\Big).
\end{split}
\end{equation*}
By H\"{o}lder's inequality, one has
\begin{align*}
|I_{11}| \lesssim\int_{|x|>4|y|}|x|^2 \frac{|y|^3}{|x|^3} |v(y)f(y)|dy
\lesssim \frac{1}{|x|}   \langle |y|^3 |v(y), f\rangle
\lesssim   \frac{1}{|x|}.
\end{align*}
For $I_{12}$. By the same argument with the proof of $I_{12}$ in \eqref{I1-twoparts}, one has
\begin{equation}\label{estimate-ln-1}
\begin{split}
\left|\int_{|x|\leq 4|y|} |x|^2 \ln\Big(\frac{|x-y|^2}{|x|^2}\Big) v(y)f(y)dy\right|
 \lesssim \frac{1}{|x|}.
\end{split}
\end{equation}
 Hence by \eqref{estimate-ln-1} and H\"{o}lder's inequality again, one has
\begin{equation*}
\begin{split}
|I_{12}|\lesssim & \Big|\int_{|x|\leq 4|y|}|x|^2 \Big[\ln\Big(\frac{|x-y|^2}{|x|^2} \Big)+
    2\frac{x\cdot y}{|x|^2}-\frac{|y|^2}{|x|^2}\Big]v(y)f(y)dy\Big|\\
   \lesssim & \frac{1}{|x|}+  \int_{|x|\leq 4|y|} |y|^2 |v(y)f(y)|dy\\
     \lesssim & \frac{1}{|x|} +  \int_{|x|\leq 4|y|}\frac{|y|}{|x|}|y|^2 |v(y)f(y)|dy
  \lesssim  \frac{1}{|x|}.
\end{split}
\end{equation*}
Hence, combining the estimates $I_{11}$ and $I_{12}$ above, we get that $I_1$ is bounded by $|x|^{-1}$ which belongs to $L_1^\infty$. Similarly, we obtain that both $I_2$ and $I_3$ above are bounded by $|x|^{-1}$.

On the other hand, let $\phi\in L^\infty_1$ satisfy $H\phi=0$ and $f=Uv\phi$, by the similar methods of
\eqref{ortoganial-vf} and \eqref{S0T0S0f}, we have
$$\langle v,f \rangle=0, \,\ \, \langle x_iv, f\rangle=0, i=1,2, \, \ \,\langle x_ix_j, f \rangle=0, i,j=1,2,\,\ \,S_0T_0S_0f=0.$$
By the proof of Proposition  \ref{lemma-spectral-S2L2}, we get  $QT_0f=0$,
which gives $f\in S_2L^2$. By the estimates of $I_i(i=1,2,3)$ in \eqref{pro54-esti-G0vf-I123}, we have $G_0vf \in L^\infty_1$.
By  Proposition  \ref{lemma-spectral-S2L2}, then  $\phi=-G_0vf+ \frac{\langle v, T_0f\rangle}{\|V\|_{L^1}}$.
Thus $\frac{\langle v, T_0f\rangle}{\|V\|_{L^1}}=\phi+G_0vf \in L^\infty_1$, so we must have $ \langle v, T_0f\rangle=0$, which implies that
$\phi=-G_0vf$ and $PT_0f=0$. Hence $f\in S_4L^2$.

The proof of this proposition is completed.
\end{proof}

\begin{proposition}\label{lemma-spectral-S5L2}
 Assume that $|V(x)|\lesssim (1+|x|)^{-\beta}$ with $\beta>18$.   Then
$f\in S_5L^2(\mathbb{R}^2)\setminus \{ 0\}$  if and only if  $f=Uv\phi$ with
$\phi \in L_2^\infty(\mathbb{R}^2)$ such that $H\phi=0$ in the distributional sense,  and
$$\phi= -G_0vf .$$
\end{proposition}
\vskip0.3cm
\begin{remark}
 Noting that $L^\infty_2(\mathbb{R}^2)\subset L^2(\mathbb{R}^2)$, hence the fourth kind zero resonance is actually an eigenvalue of $H$. Conversely, let $\phi \in L^2(\mathbb{R}^2)$ such that $H\phi=0$, then we can adapt the proof of Proposition \ref{lemma-spectral-S5L2} to show the eigenfunction $\phi\in L_2^\infty(\mathbb{R}^2)$, which implies that zero eigenvalue is exactly the fourth kind resonance from Definition \ref{definition of resonance}.
\end{remark}
\begin{proof}
Since $ S_5\leq S_4 $, by Proposition \ref{lemma-spectral-S3L2}, we know that $f\in S_4L^2$ if and
only if $f=Uv\phi$ with $ \phi \in L_2^\infty$  such  that $H\phi =0$ in the distributional sense and $ \phi = -G_0vf .$
Thus, we just need to  prove that $\phi= -G_0vf\in L_2^\infty  $.

Now we begin to prove $ \phi= -G_0vf \in  L^\infty_2 $.
Indeed, if $f\in S_5L^2$, then
$$ \langle v, f\rangle
=\langle x_jv, f\rangle= \langle x_ix_jv, f\rangle
= \langle x_ix_jx_kv, f\rangle=0, \ i,j,k=1,2.$$
Moreover, we obtain that
\begin{equation}\label{pro55-esti-G0vf}
\begin{split}
 G_0vf(x)
  =& \frac{1}{32\pi}\int_{\mathbb{R}^2}|x|^2 \Big[\ln\Big(\frac{|x-y|^4}{|x|^4} \Big)+
    4\frac{x\cdot y}{|x|^2}-2\frac{|y|^2}{|x|^2}  + 4\frac{(x\cdot y)^2}{|x|^4}\Big]v(y)f(y)dy\\
  &+\frac{1}{16\pi}\int_{\mathbb{R}^2}|y|^2\Big[\ln\Big(\frac{|x-y|^2}{|x|^2} \Big)+
    2\frac{x\cdot y}{|x|^2}\Big] v(y)f(y)dy\\
  :=&I_1 +I_2 +I_3.
\end{split}
\end{equation}
For $I_1$. We have
\begin{equation*}
\begin{split}
I_1=& \frac{1}{32\pi}\int_{|x|>16|y|}|x|^2 \Big[\ln\Big(\frac{|x-y|^4}{|x|^4} \Big)+
    4\frac{x\cdot y}{|x|^2}-2\frac{|y|^2}{|x|^2}  + 4\frac{(x\cdot y)^2}{|x|^4}\Big]v(y)f(y)dy\\
   & +\frac{1}{32\pi}\int_{|x|\leq16|y|}|x|^2 \Big[\ln\Big(\frac{|x-y|^4}{|x|^4} \Big)+
    4\frac{x\cdot y}{|x|^2}-2\frac{|y|^2}{|x|^2}  + 4\frac{(x\cdot y)^2}{|x|^4}\Big]v(y)f(y)dy
    :=I_{11}+I_{12}.
\end{split}
\end{equation*}
For $I_{11}$. Since
$$ |x-y|^4= |x|^4 +|y|^4-4x\cdot y|y|^2 -4x\cdot y |x|^2+ 2|x|^2|y|^2 + 4(x\cdot y)^2,$$
then for $|x|>16|y|$,
$$ \Big||x-y|^4-|x|^4\Big| =\Big| |y|^4-4x\cdot y|y|^2 -4x\cdot y |x|^2+ 2|x|^2|y|^2 + 4(x\cdot y)^2\Big|<|x|^4.$$
By Taylor's formula we have
\begin{equation}
\begin{split}
\ln\Big(\frac{|x-y|^4}{|x|^4} \Big)=&\ln\Big(1+ \frac{|x-y|^4- |x|^4}{|x|^4} \Big) \\
=&-4\frac{x\cdot y}{|x|^2}+ 2\frac{|y|^2}{|x|^2}+4\frac{(x\cdot y)^2}{|x|^4}
 + O\Big( \frac{|y|^4}{|x|^4}\Big).
\end{split}
\end{equation}
Thus, by H\"{o}lder inequality again, we obtain
\begin{align*}
|I_1|\lesssim &\Big|\int_{|x|>16|y|}|x|^2 \Big[\ln\Big(\frac{|x-y|^4}{|x|^4} \Big)+
    4\frac{x\cdot y}{|x|^2}-2\frac{|y|^2}{|x|^2}  + 4\frac{(x\cdot y)^2}{|x|^4}\Big]v(y)f(y)dy\Big|\\
\lesssim &\int_{|x|>16|y|}|x|^2\frac{|y|^4}{|x|^4} |v(y)f(y)|dy
\lesssim   \frac{1}{|x|^2}.
\end{align*}
For $I_{12}$. By the same argument with the proof of $I_{12}$ in \eqref{pro54-esti-G0vf-I123},  we obtain that $I_{12}$ is bounded by $|x|^{-2}$. Combining with the estimates $I_{11}$ and $I_{12}$ above, we get that $I_1$ is bounded by $|x|^{-2}$.

We also obtain that $ I_2$ and $I_3$ above are bounded by $|x|^{-2}$ by using similar method with the proof
of Proposition \ref{lemma-spectral-S3L2}. Hence $ G_0vf \in L_2^\infty$.

On the other hand, let $\phi\in L^\infty_2$ satisfy $H\phi=0$ and $f=Uv\phi$, by the similar methods of
\eqref{ortoganial-vf} and \eqref{S0T0S0f}, we have
$$\langle v,f \rangle=0, \,\ \, \langle x_iv, f\rangle=0, i=1,2, \, \ \,\langle x_ix_j, f \rangle=0, i,j=1,2, \,\,\,
\langle x_ix_jx_k, f \rangle=0,\,i,j,k=0.$$
By the proof of Proposition \ref{lemma-spectral-S3L2}, we also have $T_0f=0$,
which gives $f\in S_5L^2$.

The proof of this proposition is completed.
\end{proof}

As the consequence of  Definition \ref{definition of resonance} and Propositions \ref{lemma-spectral-S1L2}--\ref{lemma-spectral-S5L2}, we immediately have the
following theorem.
\begin{theorem}\label{resonance solutions}
Let $H=\Delta^2+V$ on $L^{2}(\mathbb{R}^2)$ and $|V(x)|\lesssim(1+|x|)^{-\beta}$ for some $\beta>0$. Then we have following statements:

(i)~If $\beta>10$,  then zero is the regular point of $H$ if and only if there exists only zero solution $\phi(x)\in L^\infty_{-1}(\mathbb{R}^2) $ such that $H\phi=0$ in the distributional sense,

(ii)~If $\beta>14$, then zero is  the first kind resonance  of $H$ if and if there exists some nonzero $\phi(x)\in L^\infty_{-1}(\mathbb{R}^2)$ but no nonzero $ \phi\in L^\infty(\mathbb{R}^2) $ such that $H\phi=0$.

(iii)~If $\beta>18$, then zero is  the second kind resonance  of $H$ if and only if there exists some nonzero  $\phi(x)\in L^\infty(\mathbb{R}^2)$ but no nonzero $ \phi\in L^\infty_1(\mathbb{R}^2) $ such that $H\phi=0$ .

(iv)~If $\beta>18$, then zero is  the third kind resonance  of $H$ if and only if there exists some  nonzero $\phi(x)\in L^\infty_1(\mathbb{R}^2)$ but no nonzero $  \phi\in L^2(\mathbb{R}^2) $ such that $H\phi=0$.

(v)~If $\beta>18$, then zero is  the fourth kind resonance ( i.e. eigenvalue ) of $H$ if and only if there exists some nonzero $\phi(x)\in L^2(\mathbb{R}^2) $ such that $H\phi=0$.

\end{theorem}
Let $\phi\in C^\infty(\mathbb{R}^2)$ be a positive function which is equal to $c\ln|x|+d$ for $|x|>1$, where $c, d$ are some positive constants.  Then we can easily check that $H\phi(x)=0$ when taking $\displaystyle V(x)=-(\Delta^2\phi)/\phi.$  It is obvious that $V(x)\in C_0^\infty(\mathbb{R}^2)$ and $\phi(x)\in L_{-1}^\infty(\mathbb{R}^2),$  hence by Theorem \ref{resonance solutions}, it follows that the zero is a resonance point  of $H$.  On the other hand, as shown in Lemma 3.5 of \cite{SWY21}, zero cannot be an eigenvalue of $H$ in one dimensional case assuming that $V$ has  fast enough decay ( e.g. $|V(x)|\le C(1+|x|)^{-\beta}$ for $\beta>25$ ). However,  it would be an interesting problem to find whether there exist some sufficient decay potentials $V$ such that zero is an eigenvalue of $H=\Delta^2+V$ in $\mathbb{R}^2$. Clearly, note that let $\phi=(1+|x |^2)^{-s/2}$, then $\phi\in L^2(\mathbb{R}^2)$ for any $s>1$ and $(\Delta^2+V)\phi=0$ as we take
 $$V(x)=-{\Delta^2\phi\over\phi}\sim |x|^{-4}, \ \ {\text{as} } \ |x|\rightarrow\infty,$$
which actually means that the our required decay rate of $V$ in question above is bigger than $4$ at least.
\vskip0.5cm

{\bf Acknowledgements:} The authors would like to thank the anonymous referees for careful reading the manuscript and providing valuable
suggestions, which substantially helped improving the quality of this paper. Avy Soffer is partially supported by the
Simon's Foundation (No. 395767).
Ping Li and Xiaohua Yao are partially supported by  NSFC ( No.11771165 and 12171182).
 We would like to thank Dr. Zhao Wu for his help and useful discussions. The authors also thank Professor William Green for his insightful comments that  improves the presentation
in the previous version.



\end{document}